\documentclass[twoside,bezier,12pt, reqno]{amsart}
\usepackage{amssymb,latexsym,epsfig,mathrsfs,graphpap,graphics,amsmath,amssymb}
\usepackage{tikz}
\usetikzlibrary{arrows.meta,positioning}
\usepackage{amstext}
\usepackage{amsthm}
\usepackage[utf8]{inputenc}
\usepackage[english]{babel}
\usepackage{helvet}
\usepackage{courier}
\usepackage{tikz}
\usepackage{enumerate}
\usepackage{graphicx}
\usepackage{float}
\usepackage{subfigure}
\usepackage{framed}
\usetikzlibrary{shapes,arrows}
\usepackage{bm}

\newtheorem{theorem}{Theorem}[section]

\theoremstyle{Definition}
\newtheorem{definition}{Definition}[section]

\theoremstyle{remark}

\numberwithin{equation}{section}

\begin{document}

\begin{flushleft}
 {\bf\Large{  New Fractional Ambiguity Function Integrated with CNN-Based Machine Learning for Signal Classification}}

\parindent=0mm \vspace{.2in}

\bf{ Aamir H. Dar*$^1$, Prakhar Kumar Sonkar$^2$ and    Neeraj Kumar Sharma$^3$}
\end{flushleft}
\noindent
{\it $^{1,2,3}$Mehta Family School of Data Science \& Artificial Intelligence\\
Indian Institute of Technology Guwahati, Guwahati-781039, India
\\E-mail: $\text{ ahdkul740@gmail.com; prakharkumar81@gmail.com; neerajs@iitg.ac.in}$}
\begin{quotation}
\noindent
{\footnotesize {\sc Abstract.} A new fractional ambiguity function (NFrAF) derived from the fractional Fourier transform is introduced as a generalization of the classical ambiguity function.  The fundamental analytical properties of the NFrAF, including symmetry, marginality, and Moyal-type identities, are rigorously established. After verifying its ability to detect and localize mono-component and multi-component linear frequency-modulated (LFM) signals, the NFrAF is integrated into a convolutional neural network–based machine-learning framework for signal classification. Owing to its superior time–frequency resolution and localization, the NFrAF provides a more informative input representation than conventional methods such as the spectrogram and classical ambiguity function. Experimental results on simulated datasets demonstrate consistent improvements in classification accuracy, highlighting the effectiveness of the proposed representation for data-driven signal analysis.\\

}
\end{quotation}

\section{Introduction} \label{sec intro}
\noindent

Within the time-frequency analysis, one of the most popular study topics is non-stationary signal processing and analysis.
Several innovative theories in signal processing have been put out to examine non-stationary signals. These include the short-time Fourier transform (STFT) \cite{af1,sm2}, wavelet transform (WT) \cite{af2,sim} and fractional Fourier transform (FrFT) \cite{af3,af4}.  Due to its intrinsic peculiarities, the fractional Fourier transform (FrFT), which is a generalization of the classical FT, has garnered increasing attention. It has been demonstrated that the FrFT may be viewed as a unified time-frequency transform \cite{af4}. In quantum mechanics, Namias discussed the concept of the fractional Fourier operator in 1980 \cite{fr1}, and the applied mathematics community rediscovered it in 1987 \cite{af6}. Almeida \cite{af7} and Santhanam and McClellan \cite{af8} were the first to introduce it to the signal processing community. The FrFT's applicability in actual, practical scenarios are made possible by the discrete and digital computing techniques that were presented in \cite{af9,af10}. Differential equations, quantum physics, neural networks, optics, pattern recognition, communication systems, radar, sonar, signal and image processing, and other fields are among the many areas in which it is widely used \cite{fr3}-\cite{fr24}. The fractional Fourier transform (FrFT) is dependent on a parameter called $\theta$. In the time-frequency plane, it can be interpreted as the rotation of a signal by an angle $\theta$, which facilitates signal processing. The angular Fourier transform is another name for the FrFT. 

For  any signal $ {\bf x}(t) \in L^2(\mathbb R)$, the $\theta-$order FrFT is  defined by \cite{fr221}
\begin{equation}\label{eqn QPFT}
\mathcal F^\theta[{\bf x}(t)](u)=\int_{\mathbb R}{\bf x}(t)\mathcal K_\theta(t,u)dt,
\end{equation}
where the FrFT kernel $\mathcal K_\theta(t,u)$ is given by
\begin{equation}\label{q-kernal}
\mathcal K_\theta(t,u)=\left\{\begin{array}{ll}
  \Omega_\theta e^{\frac{i}{2}(t^2+u^2){\cot\theta}-i ut\csc\theta},&\theta\neq n\pi, \\
\dfrac{1}{\sqrt{2\pi}}\,e^{-itu}, &\theta=\dfrac{\pi}{2}, \\\
\delta(t-u), &\theta=2n\pi, \\
\delta(t+u), &\theta=(2n\pm 1)\pi, \\
\end{array}{}\right.
\end{equation}
Where  $ \Omega_\theta=\sqrt{(1-i\cot\theta)/2\pi}$ and  the  inverse  FrFT is given by
\begin{equation}
{\bf x}(t)=\int_{\mathbb R}\mathcal F^\theta[{\bf x}(t)](u)\mathcal K^*_\theta(t,u)du.
\end{equation}
Time-frequency analysis of a wider class of signals, particularly LMF chirp signals, is substantially improved by the fact that the FrFT is controlled by the free parameter $\theta$. A popular approach to take advantage of the FrFT's benefits is to extend the Wigner distributions and ambiguity function into the domain of fractional Fourier transforms \cite{fr222a}-\cite{fr222d}.\\
On the flip side, a common non-stationary signal used in radar and sonar systems, communications, is the linear frequency-modulated (LFM) signal. Since LFM signal processing is so crucial, numerous algorithms and techniques have been put forth \cite{new26}–\cite{new30}. One of the most crucial time-frequency instruments in LFM signal processing is the ambiguity function (AF) connected to the FT \cite{fra19,fra20}. Numerous other significant and practical methods have also been proposed in relation to the LFM signal parameter estimation and spectral analysis. These include the following: the minimum mean square error (MMSE) estimation \cite{fra21}, the scaled AF \cite{scale1,scale3}, the AF associated with the linear canonical transform (LCT) \cite{fraa1}, the Wigner transform in the LCT and quadratic phase Fourier  domain \cite{zhang}-\cite{owncar}, the Chirp FT \cite{fra26}, the Wiger Ville distribution (WVD) in FrFT, and the Radon-ambiguity transform (RAT) \cite{fra28}. 


For a finite energy signal ${\bf x}(t)$,  the classical AF is defined
denoted by
\begin{equation}\label{int 1}
\mathcal  A_{\bf x}(\tau,u)=\int_{\mathbb R}\left({\bf x}\underset{\frac{\tau}{2}}{\otimes}{\bf x}^*\right)(t) e^{-iut}dt,
\end{equation}
where $\left({\bf x}\underset{\frac{\tau}{2}}{\otimes}{\bf x}^*\right)(t)$  represents the tensor product ({\it superscript $\ast$ denote complex conjugate}) and is given by  $\left({\bf x}\underset{\frac{\tau}{2}}{\otimes}{\bf x}^*\right)(t) ={\bf x}\left(t+\frac{\tau}{2}\right){\bf x}^*\left(t-\frac{\tau}{2}\right).$

The authors in \cite{fra24,fr222b} first study the AF connected to the FrFT and discover several significant properties by adhering to the classical notion of AF. In the optical signal processing community, Dar et al.'s  \cite{fr222a1,fr222c}  recent study  explores the characteristics and physical significance of the WD and AF connected to the FrFT. In contrast to the description found in \cite{fra24}-\cite{fr222a1}, this article suggests a different definition for AF that is linked to the FrFT. It also goes over the key characteristics and uses of the newly defined AF in LFM signal processing.\\

On the flip side, the convolutional neural networks (CNNs)  \cite{siam1}-\cite{siam2} have emerged as a leading method in machine learning for signal classification due to their ability to automatically extract features from data. CNNs are now commonly used with time-frequency representations (TFRs) of non-stationary signals, despite their initial design for image analysis \cite{CNN1,CNN2}. Despite its low resolution, which is limited by the window selection in the short-time Fourier transform, the majority of current works rely on the spectrogram \cite{cnn1}-\cite{cnn3}. Superior energy concentration and time-frequency localization are provided by more sophisticated TFRs, such as those derived from the WD/AF. Even though recent research suggests that these cross-terms may contain useful discriminatory information for machine learning, they are frequently avoided because of cross-terms. WD-based CNN inputs are investigated in a few works \cite{CNN4,CNN3}, but usually without explanation or comparison with standard representations. Recent research trends highlight  the need for better-localized and higher-resolution TFRs as inputs to data-driven classification models in response to these gaps. This encourages the creation and study of new quadratic TFRs, especially those derived from fractional Fourier transforms, which provide more  flexibility by rotating in the time-frequency plane.

In this regard, our work suggests that a better time-frequency input for CNN-based machine learning is the new fractional ambiguity function (NFrAF), which is a generalization of the classical ambiguity function and scaled AF. We show that, in comparison to the spectrogram and the classical AF, NFrAF-based representations greatly improve CNN classification performance by taking advantage of its improved concentration, tunable fractional parameters, and capacity to capture subtle differences between similar signal classes.

The objectives of this article are as follows:\\
\begin{itemize}
\item To introduce new fractional ambiguity function and establish its relationship with other signal processing tools.\\

 \item To investigate the NFrAF's characteristics, including its conjugate symmetry, nonlinearity, scaling, inverse, marginal, and Moyal formulas, as well as time and frequency shifts.\\

    \item  To illustrate the effectiveness, single-component and bi-component LFM signals are identified using the proposed NFrAF..\\

\item The results of the simulations unambiguously indicate that the suggested NFrAF performs better at detection than both scaled AF and conventional AF.\\
\item To employ the proposed NFrAF as a feature representation for CNN-based signal classification and demonstrate its superiority over spectrogram and AF inputs.
\end{itemize}
\subsection{Paper Outlines}
This paper is organized as follows. Section \ref{sec 2} introduces the proposed NFrAF and presents a detailed examination of its fundamental mathematical properties.  Section \ref{sec 4} examines NFrAF applications to LFM signal detection and introduces the main contribution of this work, namely the integration of NFrAF into a CNN-based framework for signal classification, demonstrating improved performance over conventional time–frequency representations. Section \ref{sec 5}  concludes the paper.

 \section{The New Fractional Ambiguity Function}\label{sec 2}
This section will present the idea of the new fractional ambiguity function (NFrAF) and examine some of its key characteristics that are relevant to signal processing. Let us begin by defining the scaled version of classical ambiguity function(SAF).

\subsection{Scaled ambiguity function (SAF)}
Using the fractional instantaneous auto-correlation $\left({\bf x}\underset{k\frac{\tau}{2}}{\otimes}{\bf x}^*\right)(t)$ instead of the instantaneous auto-correlation function $\left({\bf x}\underset{\frac{\tau}{2}}{\otimes}{\bf x}^*\right)(t)$, Dar et al. \cite{scale1,scale2} presented a scaling form of the conventional AF. The formal definition of the scaling ambiguity function(SAF) for every finite energy signal ${\bf x}(t)$ in terms of the tensor product is
\begin{equation}\label{eqn swd}
\mathcal  A^k_{\bf x}(\tau,u)=\int_{\mathbb R}\left({\bf x}\underset{k\frac{\tau}{2}}{\otimes}{\bf x}^*\right)(t) e^{-iut}dt,
\end{equation}
where $k\in \mathbb Q^+$ and $\left({\bf x}\underset{k\frac{\tau}{2}}{\otimes}{\bf x}^*\right)(t) ={\bf x}\left(t+k\frac{\tau}{2}\right){\bf x}^*\left(t-k\frac{\tau}{2}\right)$
is the fractional instantaneous
auto-correlation function.

\subsection{Definition of the NFrAF}
Equation (\ref{eqn swd}) can be rewritten  as follows
\begin{eqnarray}
\nonumber \mathcal  A^k_{\bf x}(\tau,u)&=&\int_{\mathbb R}\left({\bf x}\underset{k\frac{\tau}{2}}{\otimes}{\bf x}^*\right)(t) e^{-iut}dt\\
\label{eqn svd modified}&=&\int_{\mathbb R}\left(\bar{\bf x}_u\underset{k\frac{\tau}{2}}{\otimes}\hat{{\bf x}}^*_u\right)(t) dt,
\end{eqnarray}
where
 \begin{equation}\label{fun fu}
\bar{\bf x}_u(t)={\bf x}(t)e^{-i\frac{u}{2}t} \quad \mbox{and} \quad \hat{{\bf x}}_u(t)= {\bf x}(t)e^{i\frac{u}{2}t}
\end{equation}
Replacing the FT kernel in  (\ref{fun fu}) with the FrFT kernel, we get
\begin{equation}
\label{fun f uA}\bar{\bf x}^\theta_{u}(t)={\bf x}(t)\mathcal K_{\theta}\left(t,\frac{u}{2}\right) \quad \mbox{and} \quad \hat{{\bf x}}^\theta_u(t)={\bf x}(t)\mathcal K_{\theta}\left(t,\frac{-u}{2}\right).
\end{equation}

By substituting $\bar{\bf x}_u(t)$ with $\bar{\bf x}^\theta_{u}(t)$ and $\hat{{\bf x}}_u(t)$ with $\hat{{\bf x}}^\theta_u(t)$ in (\ref{eqn svd modified}), we are able to obtain a new type of AF in the FrFT domain.
\begin{eqnarray}
\nonumber \mathcal  A^{\theta,k}_{\bf x}(\tau,u)&=&\int_{\mathbb R}\left(\bar{\bf x}^\theta_u\underset{k\frac{\tau}{2}}{\otimes}\hat{\bf x}^{\theta*}_u\right)(t)dt\\
\nonumber&=&\int_{\mathbb R}\mathcal K_\theta\left(t+k\frac{\tau}{2},\frac{-u}{2}\right)\left({\bf x}\underset{k\frac{\tau}{2}}{\otimes}{\bf x}^*\right)(t)\mathcal K^*_\theta\left(t-k\frac{\tau}{2},\frac{-u}{2}\right)dt\\
\label{a}&=&|\Omega_\theta|^2\int_{\mathbb R}\left({\bf x}\underset{k\frac{\tau}{2}}{\otimes}{\bf x}^*\right)(t)e^{it(k\tau\cot\theta-u\csc\theta)}dt.
\end{eqnarray}

The following definition results from using the aforementioned equation (\ref{a}).

\begin{definition}\label{def swdolct}
For the pair of parameters  $(\theta,k),$ the new fractional ambiguity function of a signal ${\bf x}(t)\in L^2(\mathbb{R})$   is defined as
\begin{equation}\label{eqn def cbqpwd}
\mathcal  A^{\theta,k}_{\bf x}(\tau,u)=|\Omega_\theta|^2\int_{\mathbb R}\left({\bf x}\underset{k\frac{\tau}{2}}{\otimes}{\bf x}^*\right)(t)e^{it(k\tau\cot\theta-u\csc\theta)}dt,
\end{equation}
where $k\in \mathbb Q^+$ and $\left({\bf x}\underset{k\frac{\tau}{2}}{\otimes}{\bf x}^*\right)(t) ={\bf x}\left(t+k\frac{\tau}{2}\right){\bf x}^*\left(t-k\frac{\tau}{2}\right)$
\end{definition}


It is clear that the NFrAF, the  SAF, and classical AF share the following relationships.
\begin{equation}
\mathcal  A^{\theta,k}_{\bf x}(\tau,u)=|\Omega_\theta|^2\mathcal  A^{k}_{\bf x}\left(\tau,u\csc\theta-k\tau\cot\theta\right)=\frac{|\Omega_\theta|^2}{k}\mathcal  A_{\bf x}\left(\tau,\frac{1}{k}u\csc\theta-\tau\cot\theta\right).
\end{equation}
Moreover,   the  cross-NFrAF of the finite energy signals ${\bf x}_1(t)$ and ${\bf x}_2(t)$ can be defined by
\begin{equation}\label{eqn swdolct cross}
\mathcal  A^{\theta}_{{\bf x}_1,{\bf x}_2}(\tau,u)=|\Omega_\theta|^2\int_{\mathbb R}\left({\bf x}_1\underset{k\frac{\tau}{2}}{\otimes}{\bf x}^*_2\right)(t)e^{it(k\tau\cot\theta-u\csc\theta)}dt.
\end{equation}
As a consequence, the proposed NFrAF \eqref{eqn def cbqpwd} retains the structural features of the fractional Fourier transform, while inheriting the key properties of the classical ambiguity function. Furthermore, its generality can be understood through the following special cases. When $\theta=\pi/2$, the NFrAF reduces to the scaled ambiguity function (SAF) given in \eqref{eqn swd}. In addition, for 
$k=1$, it further reduces to the classical ambiguity function.\\
Let us now look at a function that sums two Gaussian beams with centers at $t=0$ and $t=3$, respectively: $${\bf x}(t)= exp\left\{-\frac{t^2}{\sqrt{2}}\right\}+ exp\left\{-\frac{(t-3)^2}{\sqrt{2}}\right\}.$$  The traditional AF is shown in Fig.\ref{fig11}(a), which is the counter plot of NFrAF at $\theta=\pi/2$ and $k=1$. In Fig.\ref{fig11}(b), the SAF is shown by the counter plot of NFrAF at $\theta=\pi/2$ and $k=4$. For $\theta=\pi/3$ and $k=1$, the AF in the FrFT domain seen in Fig.\ref{fig11}(c) is represented by NFrAF.  In Figure \ref{fig11}(d)-(f), we can observe the proposed NFrAF. Thus, it is clear from the plots that  with the variation of the parameter pair  $(\theta,k)$ , the proposed NFrAF (\ref{eqn def cbqpwd})  has
a direct  effect on the cross terms in the
detection process.\\

\begin{figure}[!htbp]
\begin{framed}
\centering
\subfigure[ NFRAF of ${\bf x}(t)$ at $\theta=\pi/2$,$k=1$]
{\includegraphics[width=.49\linewidth]{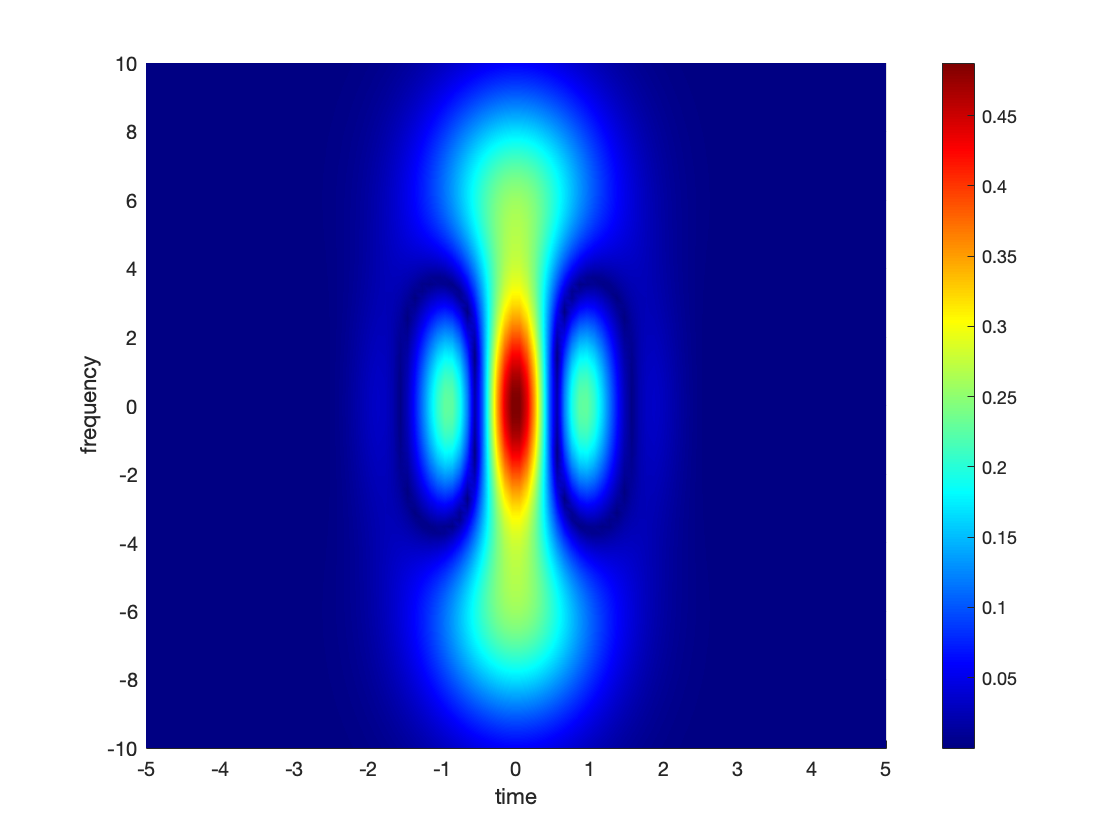}}\hfill
\subfigure[ NFRAF of ${\bf x}(t)$ at $\theta=\pi/2$,$k=4$]
{\includegraphics[width=.49\linewidth]{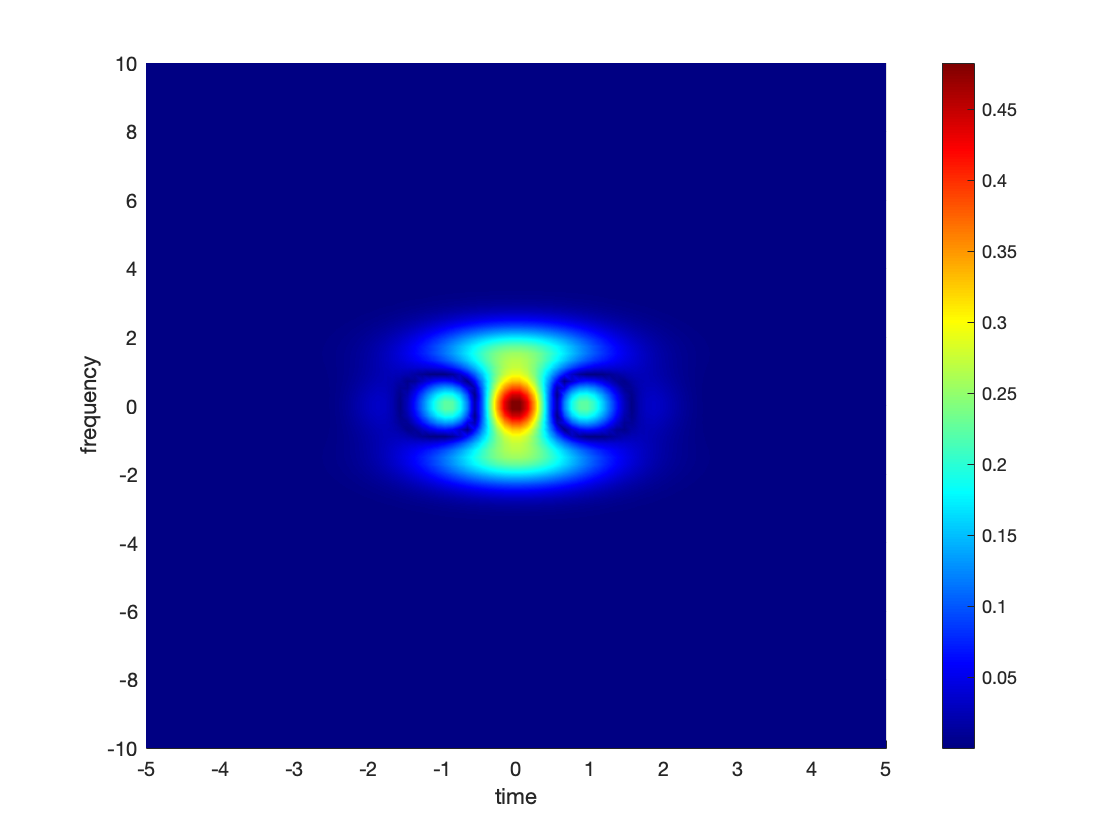}}\hfill
\subfigure[NFRAF of ${\bf x}(t)$ at $\theta=\pi/3$,$k=1$]
 {\includegraphics[width=.49\linewidth]{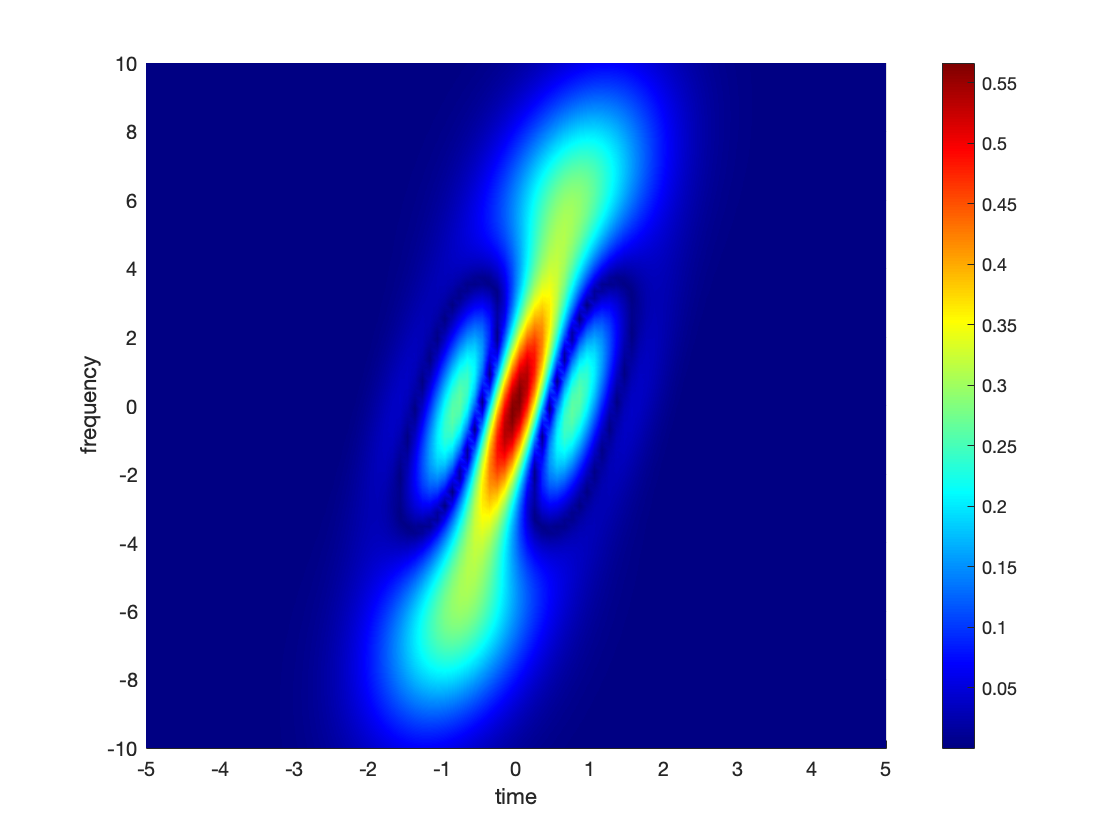}}\hfill
\subfigure[NFRAF of ${\bf x}(t)$ at $\theta=\pi/3$,$k=2$]
 {\includegraphics[width=.49\linewidth]{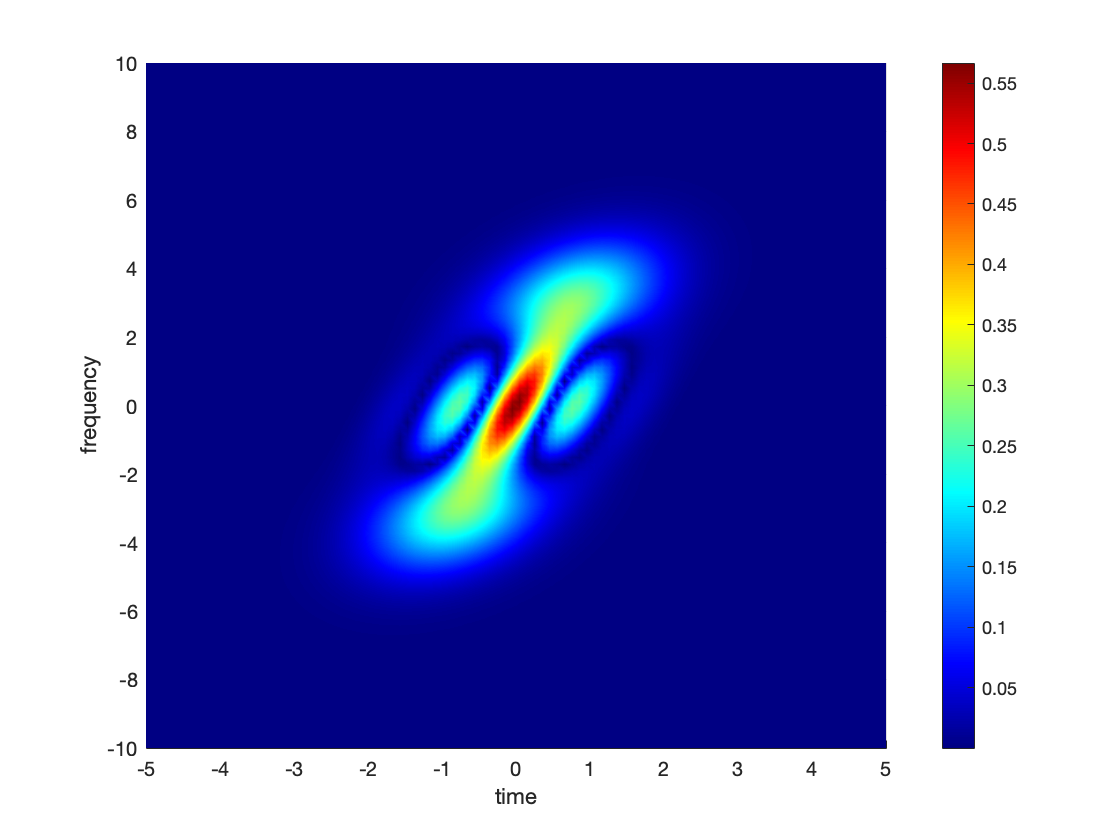}}\hfill
\subfigure[NFRAF of ${\bf x}(t)$ at $\theta=\pi/4$,$k=4$]
 {\includegraphics[width=.49\linewidth]{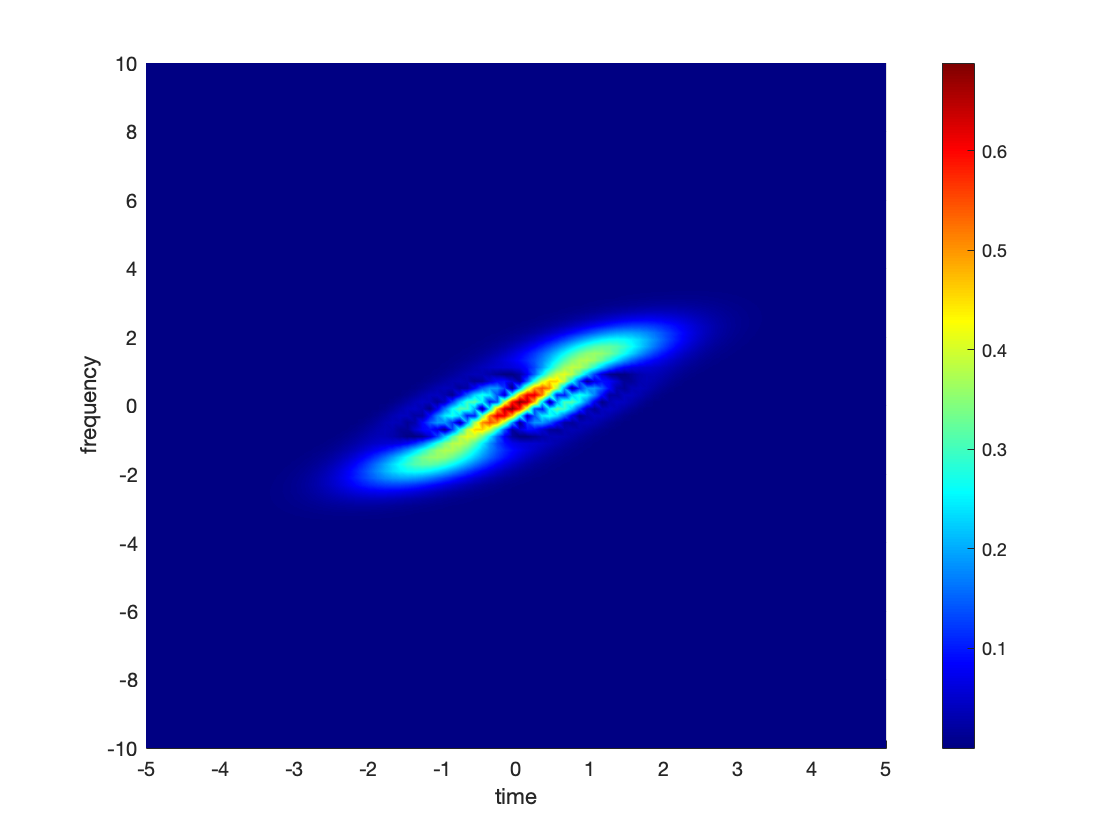}}\hfill
\subfigure[NFRAF of ${\bf x}(t)$ at $\theta=\pi/6$,$k=3$]
 {\includegraphics[width=.49\linewidth]{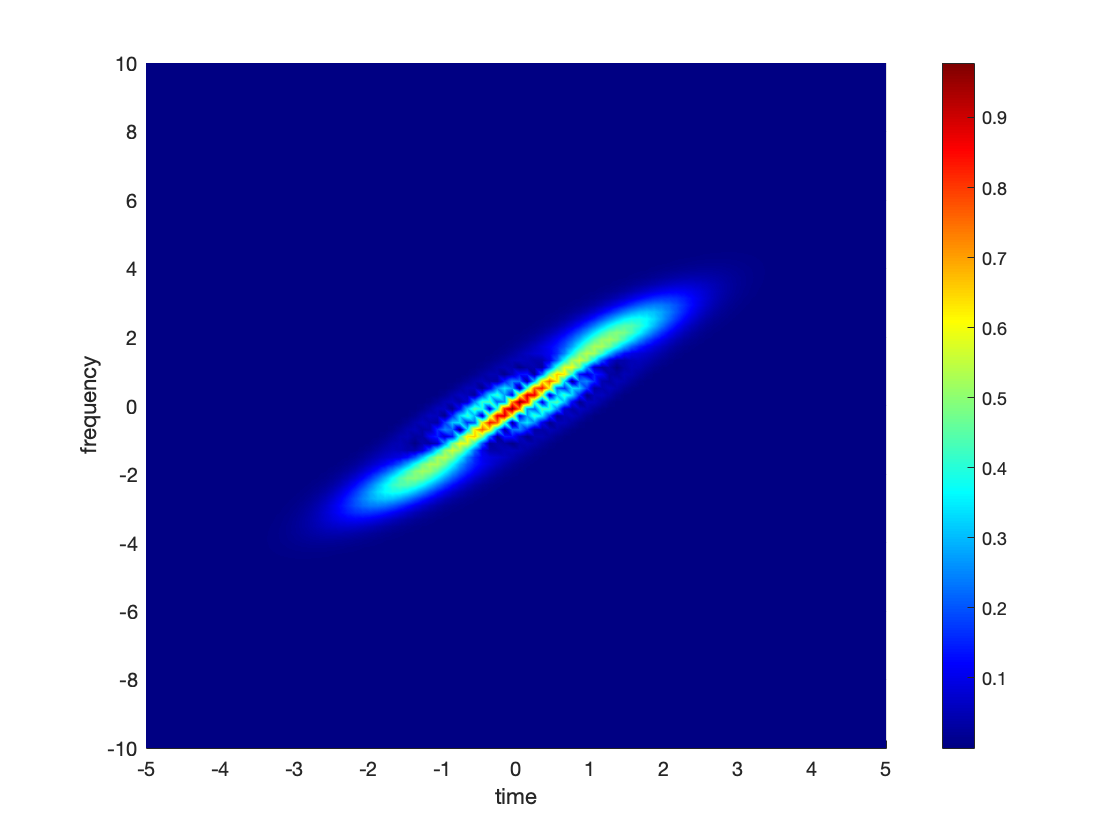}}\hfill
\caption{The comparison of the absolute value of  NFrAF for the detection  of sum of two Gaussian beams corresponding to  different choices of parameters $\theta$ and $k.$  }\label{fig11}
\end{framed}
\end{figure}
We will first determine how NFrAF interacts with SAF and the short-time Fourier transform before looking at some of the key characteristics of NFrAF (\ref{eqn def cbqpwd}).\\

\subsection{Relationship with SAF}\,\\
For the  signal ${\bf h}(t)=e^{-i\frac{\cot\theta}{2}t^2}{\bf x}(t)$,  the following relation holds:
$\mathcal  A^{\theta,k}_{\bf h}(\tau,u)
={|\Omega_\theta|^2}\mathcal  A^k_{\bf x}(\tau,u\csc\theta).$
\begin{proof}
We have from (\ref{eqn def cbqpwd})
\begin{eqnarray*}
&&\mathcal  A^{\theta,k}_{\bf h}(\tau,u)dt\\
&&=|\Omega_\theta|^2\int_{\mathbb R}\left({\bf h}\underset{k\frac{\tau}{2}}{\otimes}{\bf h}^*\right)(t)e^{it(k\tau\cot\theta-u\csc\theta)}dt\\
&&=|\Omega_\theta|^2\int_{\mathbb R}\left({\bf x}\underset{k\frac{\tau}{2}}{\otimes}{\bf x}^*\right)(t)e^{-i\frac{\cot\theta}{2}\left(t+k\frac{\tau}{2}\right)^2}e^{i\frac{\cot\theta}{2}\left(t-k\frac{\tau}{2}\right)^2}e^{it(k\tau\cot\theta-u\csc\theta)}dt\\
&&=|\Omega_\theta|^2\int_{\mathbb R}\left({\bf x}\underset{k\frac{\tau}{2}}{\otimes}{\bf x}^*\right)(t)e^{-itu\csc\theta}dt\\
&&=|\Omega_\theta|^2\mathcal  A^k_{\bf x}(\tau,u\csc\theta).
\end{eqnarray*}
This completes the proof.
\end{proof}
{\footnotesize
\begin{table}[h!]
\caption{Properties of NFrAF}\label{tab1}
\begin{tabular}{lll}
  \hline\\
  &Property&Mathematical formula\\
  \hline\\\\
  &Conjugate  properties& $\overline{\mathcal  A^{\theta,k}_{{\bf x}}(\tau,u)}=\mathcal  A^{\theta,k}_{{\bf x}}(-\tau,-u).$
\\\\
  &&$\mathcal  A^{\theta,k}_{P{\bf x}(t)}(\tau,u)=\mathcal  A^{\theta,k}_{{\bf x}(t)}(-\tau,-u),$ where $P{\bf x}={\bf x}(-t).$\\\\
  &Non-linearity property&$\mathcal  A^{\theta,k}_{{\bf x}_1+{\bf x}_2}(t,u)=\mathcal  A^{\theta,k}_{{\bf x}_1}(\tau,u)+\mathcal  A^{\theta,k}_{{\bf x}_2}(\tau,u)$\\
  &&\qquad\qquad\qquad$+\mathcal  A^{\theta,k}_{{\bf x}_1,{\bf x}_2}(\tau,u)+\mathcal  A^{\theta,k}_{{\bf x}_2,{\bf x}_1}(\tau,u).$\\\\

  &Time shift&$\mathcal  A^{\theta,k}_{{\bf x}(t-t_0)}(t,u)=e^{it_0(k\tau\cot\theta-u\csc\theta)}\mathcal  A^{\theta,k}_{\bf x}\left(\tau,u\right).$\\\\
  &Frequency shift&$ \mathcal  A^{\theta,k}_{\tilde {\bf x}}(\tau,u)=\mathcal  A^{\theta,k}_{\bf x}\left(t,u-ku_0\sin\theta\right),$ where $\tilde{\bf x}={\bf x}(t)e^{iu_0t}.$\\\\
  &Scaling property&$\mathcal  A^{\theta,k}_{\hat {\bf x}}(\tau,u)=\mathcal  A^{\theta,k}_ {\bf x}\left(\lambda \tau,u\frac{\csc\theta}{\lambda \csc\phi}\right),$
where $\hat {\bf x}=\sqrt{\lambda}{\bf x}(\lambda t)$ \& $\phi=arc cot\left(\frac{\cot\theta}{\lambda^2}\right).$\\\\
  
  &Inverse property&${\bf x}(t)=\frac{e^{-i\frac{t^2}{2}\cot\theta}}{{\bf x}^*(0)\csc\theta}\int_{\mathbb R}\mathcal  A^{\theta,k}_{ {\bf x}}\left(\tau,u\right)e^{i\tau u\csc\theta}du.$\\\\
&Time marginal property&$\int_{\mathbb R}\mathcal  A^{\theta,k}_{{\bf x}}(\tau,u)du=\left({\bf x}\underset{k\frac{\tau}{2}}{\otimes}{\bf x}^*\right)(0).$ \\\\
  &Frequency marginal property&$ \int_{\mathbb R}\mathcal  A^{\theta,k}_{\bf x}(t,u)d\tau=\frac{1}{k}\mathcal F^{\theta}[{\bf x}]\left(\frac{u}{2}\right)\mathcal F^{\theta^*}[{\bf x}]\left(\frac{-u}{2}\right).$\\\\
  &Moyal formula&$\int_{\mathbb R}\int_{\mathbb R}\mathcal  A^{\theta,k}_ {{\bf x}_1}(\tau,u)\left[\mathcal  A^{\theta,k}_ {{\bf x}_2}(\tau,u)\right]^*d\tau du=\frac{|\Omega_\theta|^2}{k}\left|\langle {\bf x}_1,{\bf x}_2\rangle\right|^2.$\\\\
  \hline
\end{tabular}
\end{table}
}

Next, we recall the definition of short-time Fourier transform (STFT).
\begin{definition}
The STFT of a signal ${\bf x}(t)$ with respect to a window function ${\bf h}(t)$ is defined as \cite{fra11,fra12}
\begin{equation}
\mathcal  V_{\bf h}[{\bf x}(t)](\tau,u)=\int_{\mathbb R}={\bf x}(y){\bf h}^*(y-\tau)e^{-iuy}dy.
\end{equation}
\end{definition}
Now, we can establish the relationship between NFrAF and STFT.

\subsection{Relationship with STFT}\,\\
The  NFrAF of a signal ${\bf x}(t)$ can be expressed by  the STFT by the following relationship:
\begin{equation}\label{stftrel}
\mathcal  A^{\theta,k}_{\bf x}\left(\frac{\tau}{k},u{\sin\theta}+\tau{\cos\theta}\right)={|\Omega_\theta|^2}e^{iu\frac{t}{2}}\mathcal  V_{\bf h}[{\bf x}](\tau,u),\quad \mbox{where}\quad {\bf h}(t)={\bf x}(t).
\end{equation}

The specific proof  of Eq.(\ref{stftrel}) is shown in {\bf Appendix A}.
The theorem that follows helps us determine the essential relationship between the proposed NFrAF and the  fractional scaled Wigner distribution ($FrSWD$) (see reference \cite{fr222c}).
\begin{theorem} The 2D Fourier transform of the NFrAF gives the  FrSWD of a signal ${\bf x}(t)\in L^2(\mathbb R)$  i.e.,
\begin{equation}
\int_{\mathbb R^2}\mathcal  A^{\theta,k}_{\bf x}(\tau,z)e^{\frac{i}{b}(z t-u\tau)}dz d\tau=\frac{1}{|\Omega_\theta|^2}{[FrSWD]}^{\theta,k}_{\bf x}(t,u).
\end{equation}
\end{theorem}
\begin{proof}
 From Definition \ref{def swdolct}, we have
 \begin{eqnarray*}
 &&\int_{\mathbb R^2}\mathcal  A^{\theta,k}_{\bf x}(\tau,z)e^{{i}(z\csc\theta t-u\csc\theta\tau)}dz d\tau\\
 &&=|\Omega_\theta|^2\int_{\mathbb R^3} \left({\bf x}\underset{k\frac{\tau}{2}}{\otimes}{\bf x}^*\right)(y)e^{{iy}(k\tau\cot\theta-z\csc\theta)}e^{{i}(z\csc\theta t-u\csc\theta\tau)}dydz d\tau\\
 &&=\int_{\mathbb R^2}\left\{ \left({\bf x}\underset{k\frac{\tau}{2}}{\otimes}{\bf x}^*\right)(y)e^{{i}(ky\cot\theta-u\csc\theta)\tau}\left(|\Omega_\theta|^2\int_{\mathbb R}e^{{i}\csc\theta(t-y)z}dz\right)\right\}dyd\tau\\
  &&=\int_{\mathbb R^2}\left\{ \left({\bf x}\underset{k\frac{\tau}{2}}{\otimes}{\bf x}^*\right)(y)e^{{i}(ky\cot\theta-u\csc\theta)\tau}\delta(t-y)\right\}dyd\tau\\
   &&=\int_{\mathbb R} \left({\bf x}\underset{k\frac{\tau}{2}}{\otimes}{\bf x}^*\right)(y)e^{{i}(ky\cot\theta-u\csc\theta)\tau}d\tau\int_{\mathbb R}\delta(t-y)dy\\
 &&=\frac{1}{|\Omega_\theta|^2}[FrSWD]^{\theta,k}_{\bf x}(y,u).
 \end{eqnarray*}
 Which completes the proof.
\end{proof}

\subsection{Properties of NFrAF}\,\\
We summarize the essential properties of the  proposed NFrAF in this subsection, as these properties are indispensable for signal representation and analysis. For convenience, all properties are listed {\bf Table \ref{tab1}}, and their rigorous proofs are presented in {\bf Appendix A.}

\section{Applications}\label{sec 4}
The proposed NFrAF has great promise for LFM signal detection and classification as a generalization of classical and scaled ambiguity functions. In this section, we first demonstrate that the NFrAF can detect and localize both mono-component and multi-component LFM signals, which is a prerequisite for its application in data-driven classification, using straightforward illustrative examples. After that, a CNN uses the generated NFrAF representations as input features to evaluate how well they classify signals.
\subsection{Signal detection} Here, simulations are utilized to identify the one-component and multi-component LFM signals, respectively, using the proposed NFrAF.
\begin{itemize}

\begin{figure}[!htbp]
\centering
\begin{framed}
\begin{minipage}[b]{0.48\linewidth}
	\centering
	\includegraphics[width=\linewidth]{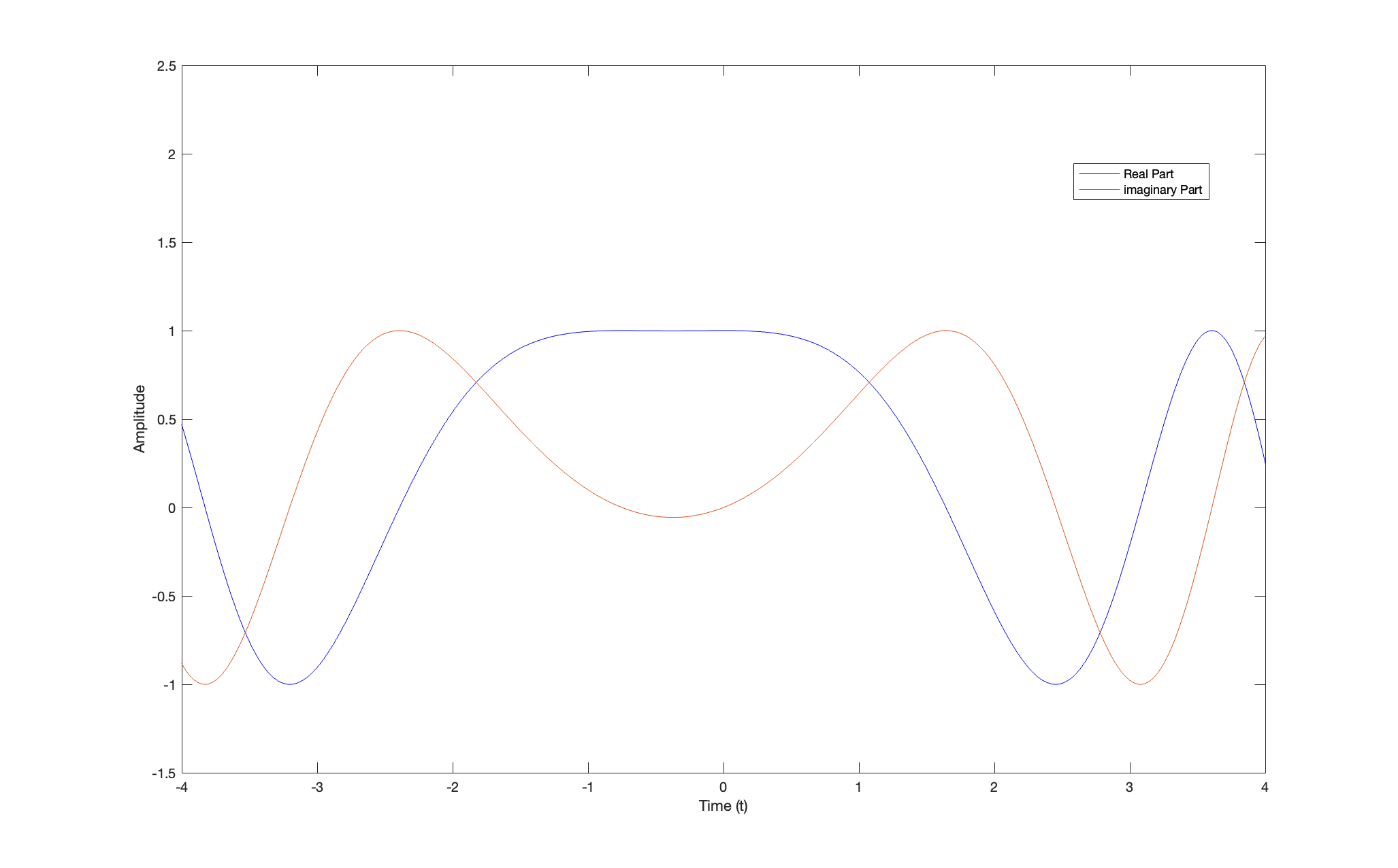}
	{\rm (a)  }
\end{minipage}
\hfill
\begin{minipage}[b]{0.48\linewidth}
	\centering
	\includegraphics[width=\linewidth]{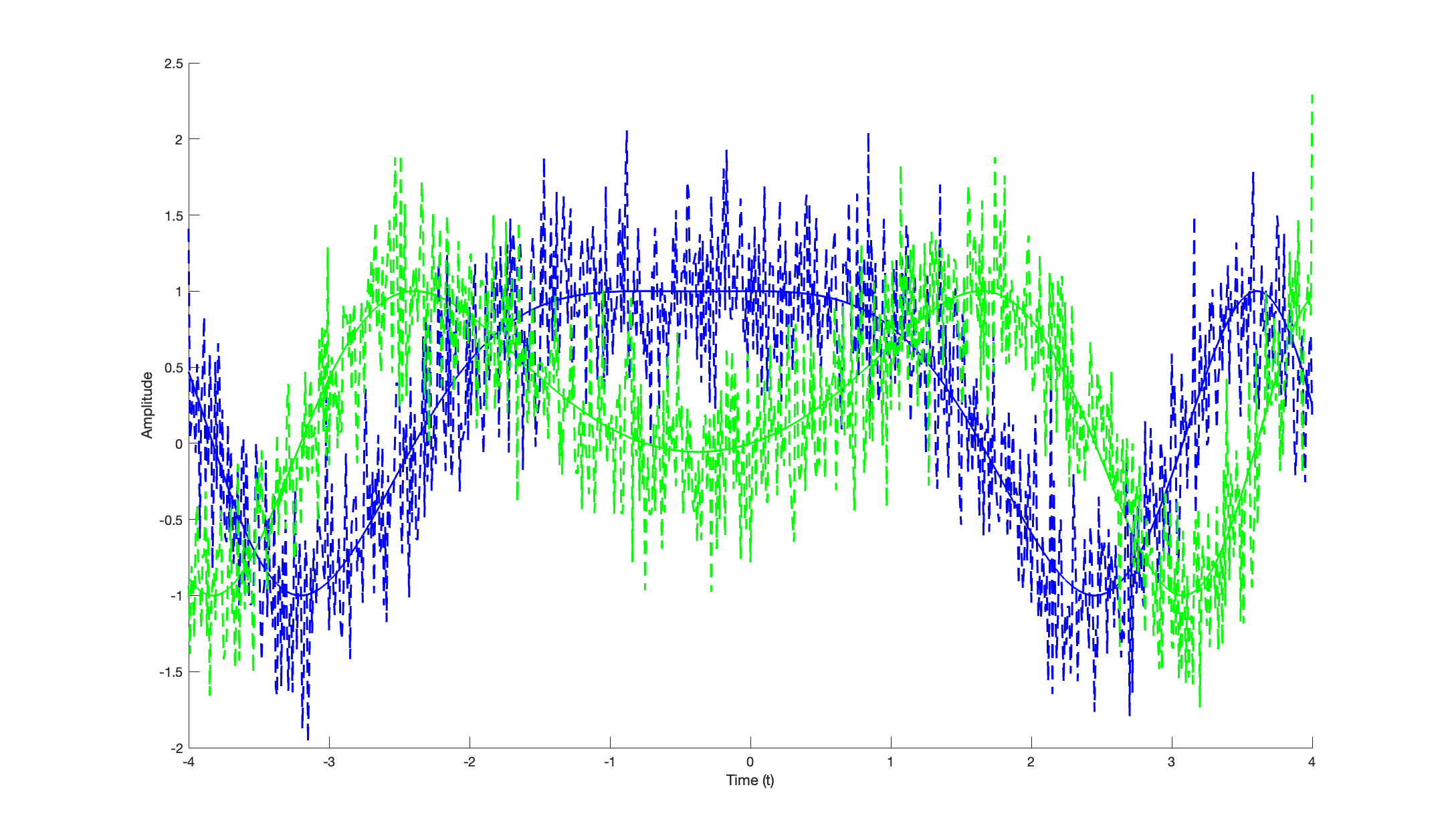}
	{\rm (b)}
\end{minipage}

\caption{\small  (a)Real and Imaginary parts of mono-component signal ${\bf x}(t)=e^{i(0.3t+0.4t^2)}$(b) Real and imaginary parts of the same signal corrupted by additive noise  at $5$dB SNR }
\label{S3F4}
\end{framed}
\end{figure}

    


\item {\bf One component LFM signal:} Consider a one-component LFM signal as\\
\begin{equation} \label{lmf 1}
{\bf x}(t)=\Lambda_0e^{i(\mu_0t+\nu_0t^2)}, \quad -\frac{T}{2}\le t\le \frac{T}{2}.
\end{equation}
where $\Lambda_0,$ $\mu_0$ and $\nu_0$ represent the amplitude, initial frequency and frequency rate of ${\bf x}(t)$, respectively. Hence by Definition \ref{def swdolct}, we obtain
\begin{eqnarray}
\nonumber &&\mathcal  A^{\theta,k}_{\bf x}(\tau,u)\\
\nonumber&&=|\Omega_\theta|^2\int_{-\frac{T}{2}}^{\frac{T}{2}}\left({\bf x}\underset{k\frac{\tau}{2}}{\otimes}{\bf x}^*\right)(t)e^{it(k\tau\cot\theta-u\csc\theta)}dt \\
\nonumber &&=|\Omega_\theta|^2\int_{-\frac{T}{2}}^{\frac{T}{2}}\Lambda_0e^{i\left[\mu_0\left(t+k\frac{\tau}{2}\right)+\nu_0\left(t+k\frac{\tau}{2}\right)^2\right]}\Lambda^*_0e^{-i\left[\mu_0\left(t-k\frac{\tau}{2}\right)+\nu_0\left(t-k\frac{\tau}{2}\right)^2\right]}e^{it(k\tau\cot\theta-u\csc\theta)}dt \\
\nonumber &&=|\Omega_\theta|^2|\Lambda_0|^2e^{ik\mu_0\tau}\int_{-\frac{T}{2}}^{\frac{T}{2}}e^{i\left[k(2\nu_0+\cot\theta )\tau-u\csc\theta\right]t}dt\\
\label{lmf2} &&=|\Omega_\theta|^2|\Lambda_0|^2 Te^{ik\mu_0\tau} sinc\left\{\frac{T}{2}\left[k(2\nu_0+\cot\theta )\tau-u\csc\theta\right]\right\}.
\end{eqnarray}
As can be shown from (\ref{lmf2}), the NFrAF of a one component signal ${\bf x}(t)$ given in (\ref{lmf 1}) can produce impulses in the $(\tau,u)$ plane at a straight line\\ $\left[k(2\nu_0+\cot\theta )\tau-u\csc\theta=0\right],$ which depends on parameters $\theta$ and $k$. The ability to select the parameter $\theta$ and scaling factor $k$ makes it clear that using NFrAF to detect one-component LFM signals is very beneficial and effective.\\
{\bf Numerical example.} Let us take a  one-component LFM signal  as \begin{equation}\label{mosig}
 {\bf x}(t)=e^{i(0.3t+0.4t^2)},\quad T=10
  \end{equation}
 then with the help of (\ref{lmf2}) and taking $\theta=\frac{\pi}{4}$ , the  NFrAF of one-component signal (\ref{mosig})  at $k=1,2,\frac{1}{2}$ are  obtained as
\begin{eqnarray}
\label{mofrwdpi4k1}\mathcal  A^{\frac{\pi}{4},1}_{\bf x}(\tau,u)&=&\frac{5\sqrt{2}}{\pi}e^{i0.3\tau}\left[sinc(9\tau-5\sqrt{2}u)\right]
\end{eqnarray}
\begin{eqnarray}
\label{mofrwdpi4k2}\mathcal  A^{\frac{\pi}{4},2}_{\bf x}(\tau,u)&=&\frac{5\sqrt{2}}{\pi}e^{i0.6\tau}\left[sinc(18\tau-5\sqrt{2}u)\right]
\end{eqnarray}
and
\begin{eqnarray}
\label{monofrwdpi4k1/2}\mathcal  A^{\frac{\pi}{4},\frac{1}{2}}_{\bf x}(\tau,u)&=&\frac{5\sqrt{2}}{\pi}e^{i0.15\tau}\left[sinc(4.5\tau-5\sqrt{2}u)\right]
\end{eqnarray}
respectively. Similarly, taking $\theta=\frac{\pi}{2}$, the NFrAF of  signal  (\ref{mosig}) at scaling parameters $k=1,2,\frac{1}{2}$  are obtained by
\begin{eqnarray}
\label{mofrwdpi2k1}\mathcal  A^{\frac{\pi}{2},1}_{\bf x}(\tau,u)&=&\frac{5}{\pi}e^{i0.3\tau}\left[sinc(4\tau-5u)\right]
\end{eqnarray}
\begin{eqnarray}
\label{mofrwdpi2k2}\mathcal  A^{\frac{\pi}{2},2}_{\bf x}(\tau,u)&=&\frac{5}{\pi}e^{i0.6\tau}\left[sinc(8\tau-5u)\right]
\end{eqnarray}
and
\begin{eqnarray}
\label{monofrwdpi2k1/2}\mathcal  A^{\frac{\pi}{2},\frac{1}{2}}_{\bf x}(\tau,u)&=&\frac{5}{\pi}e^{i0.15\tau}\left[sinc(2\tau-5u)\right]
\end{eqnarray}
respectively.
Figs.\ref{fig1}-Fig.\ref{fig3} show the detection and parameter estimation for the one-component LFM signal (\ref{mosig}) with SNR=10dB using the NFrAF for $\theta=\frac{\pi}{4},\frac{\pi}{2}$ and $k=1,2,\frac{1}{2}$.
 The fact that $\theta=\frac{\pi}{2}$ and $k=1$ implies that Fig.\ref{fig1}(c)-(d) can be understood as classical AF.The classical FrAF is represented as $k=1$ in Fig.\ref{fig1}(a)-(b). The SAF can be seen in Figs.\ref{fig2}(c)-(d) and Fig.\ref{fig3}(c)-(d) since we selected $\theta=\frac{\pi}{2}$ and $k\ne 1$. Fig.\ref{fig2}(a)-(b) and Fig.\ref{fig3}(a)-(b) are obtained, respectively, by NFrAF. As a result, we can see from Figs.\ref{fig1}–\ref{fig3} that NFrAF produces an ideal time–frequency resolution in the FrFT domain.

\begin{figure}[h!]
\begin{framed}
\centering
     \subfigure[The new NFrAF of ${\bf x}(t)$ at $\theta=\frac{\pi}{4}, k=1$]{\includegraphics[width=0.4\textwidth]{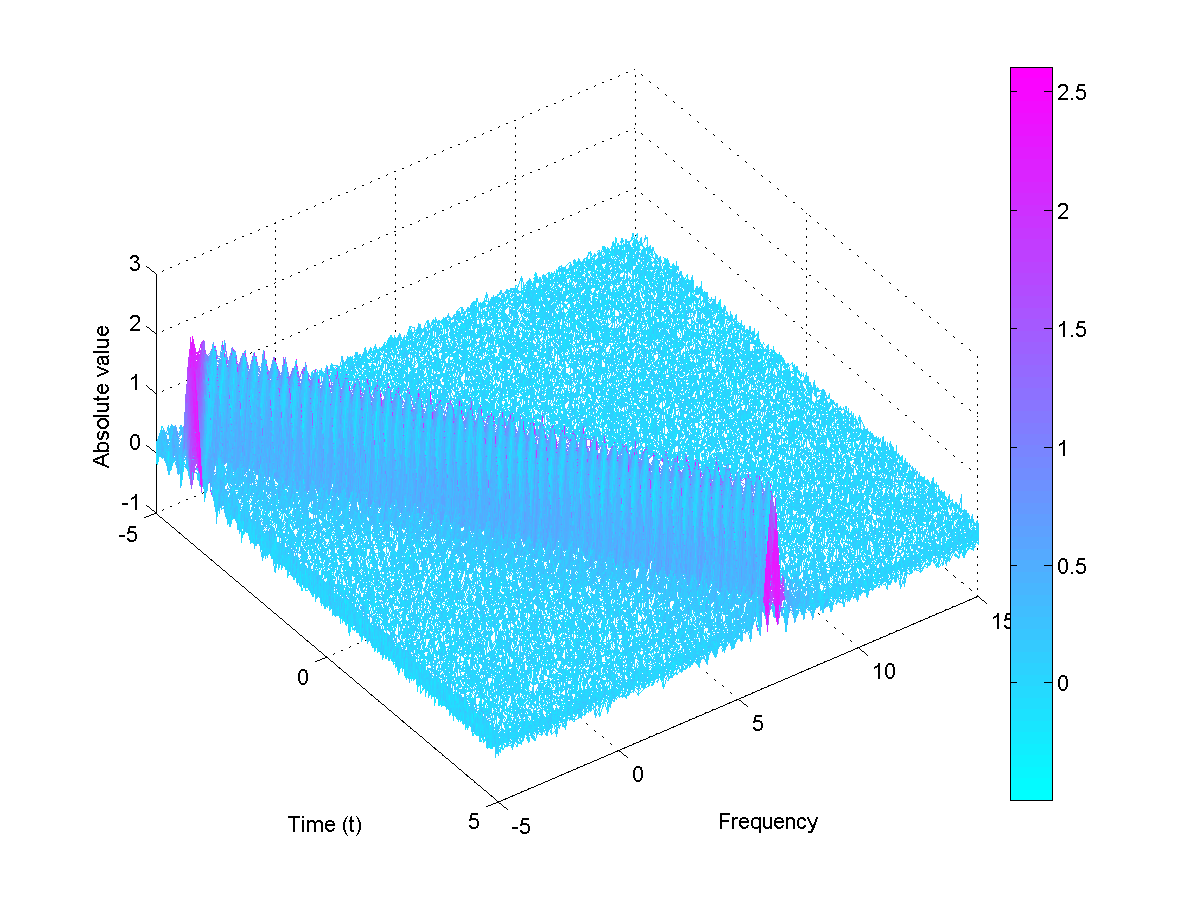}}
    \subfigure[Contour picture of new NFrAF of ${\bf x}(t)$ at $\theta=\frac{\pi}{4},k=1$]{\includegraphics[width=0.4\textwidth]{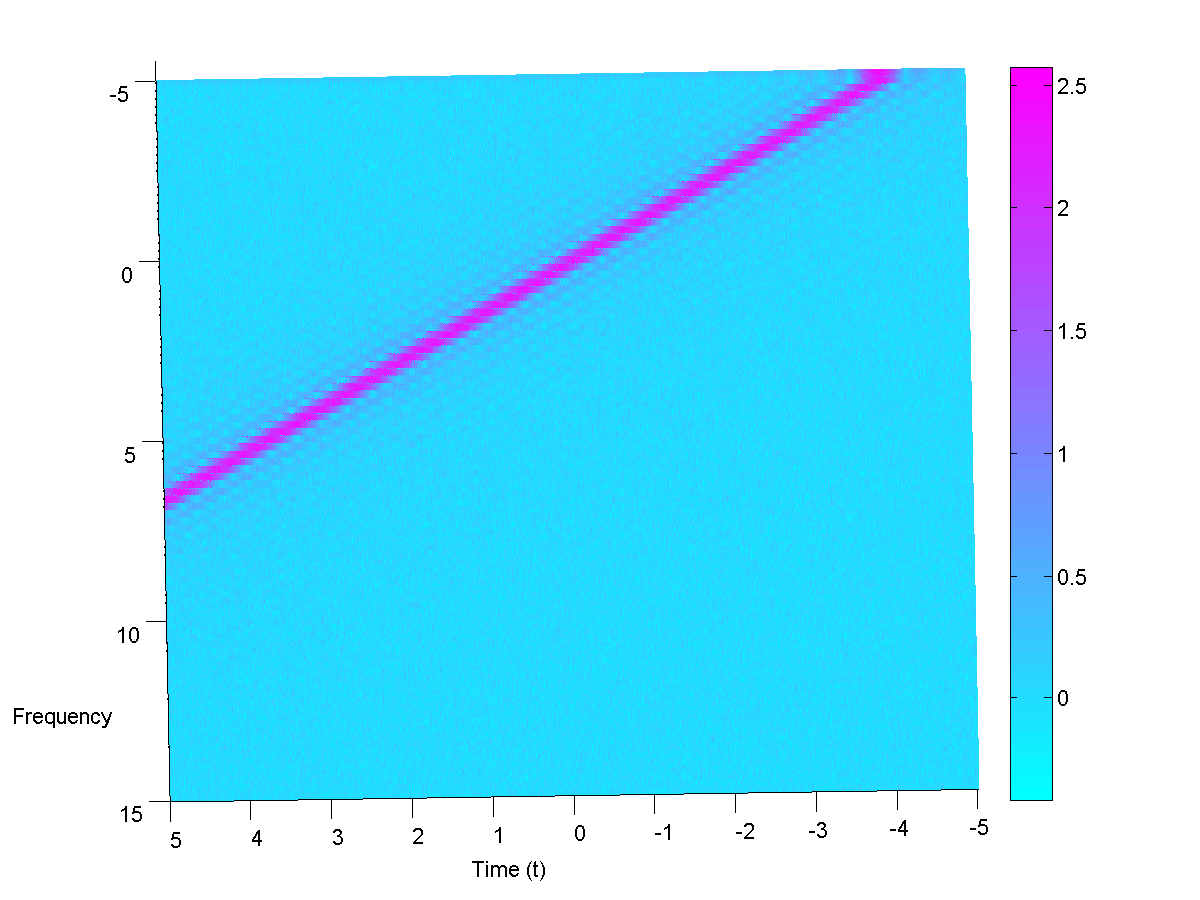}}
      \subfigure[The new NFrAF of ${\bf x}(t)$ at $\theta=\frac{\pi}{2}, k=1$]{\includegraphics[width=0.4\textwidth]{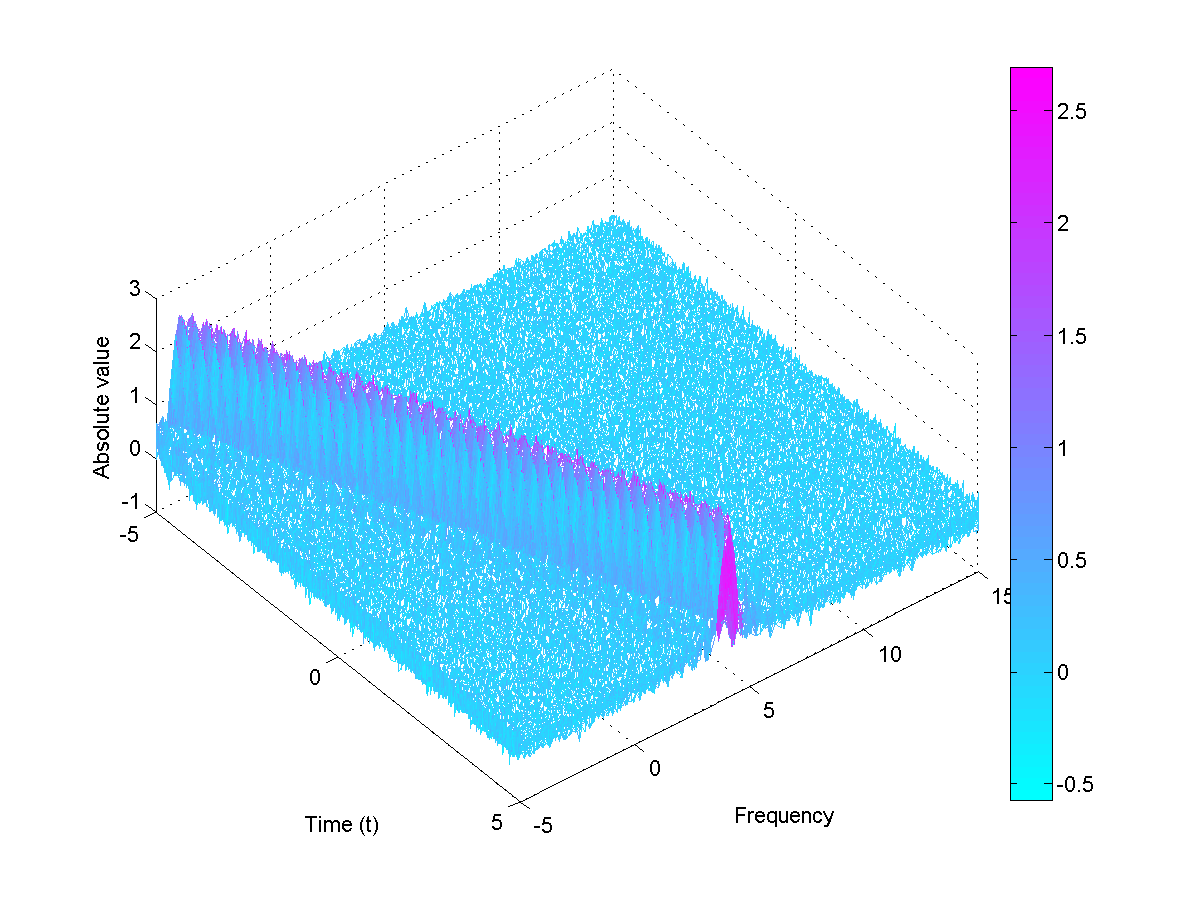}}
    \subfigure[Contour picture of new NFrAF of ${\bf x}(t)$ at $\theta=\frac{\pi}{2},k=1$]{\includegraphics[width=0.4\textwidth]{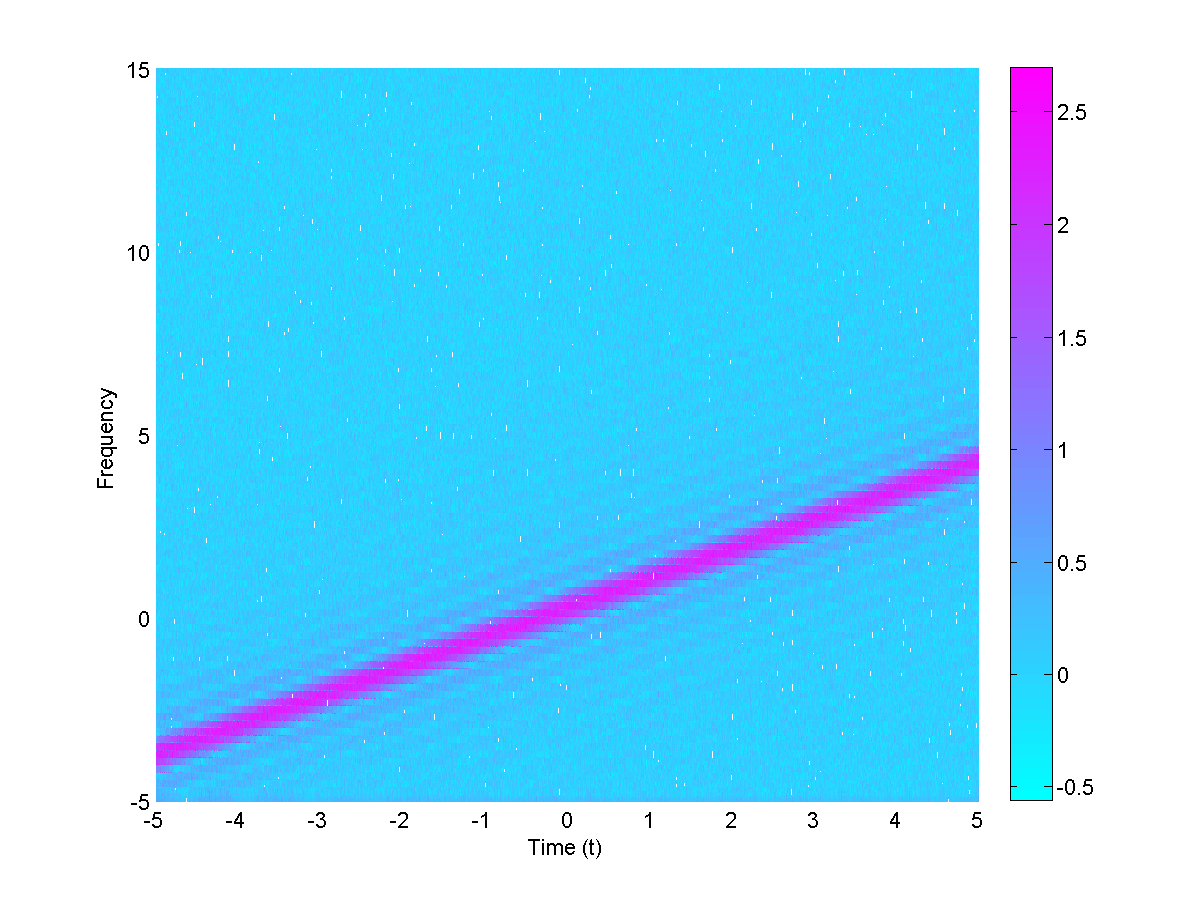}}
     \caption{Absolute value of the new NFrAF   of a mono-component signal ${\bf x}(t)=e^{i(0.2t+0.3t^2)}$ corresponding to particular choices of $\theta$ and $k$   with
$\Lambda_0=1,$  $T=10$ and SNR =10dB. }\label{fig1}
\end{framed}
\end{figure}

\begin{figure}[!htbp]
\begin{framed}
\centering
\subfigure[The new NFrAF of ${\bf x}(t)$ at $\theta=\frac{\pi}{4},k=2$]{\includegraphics[width=0.4\textwidth]{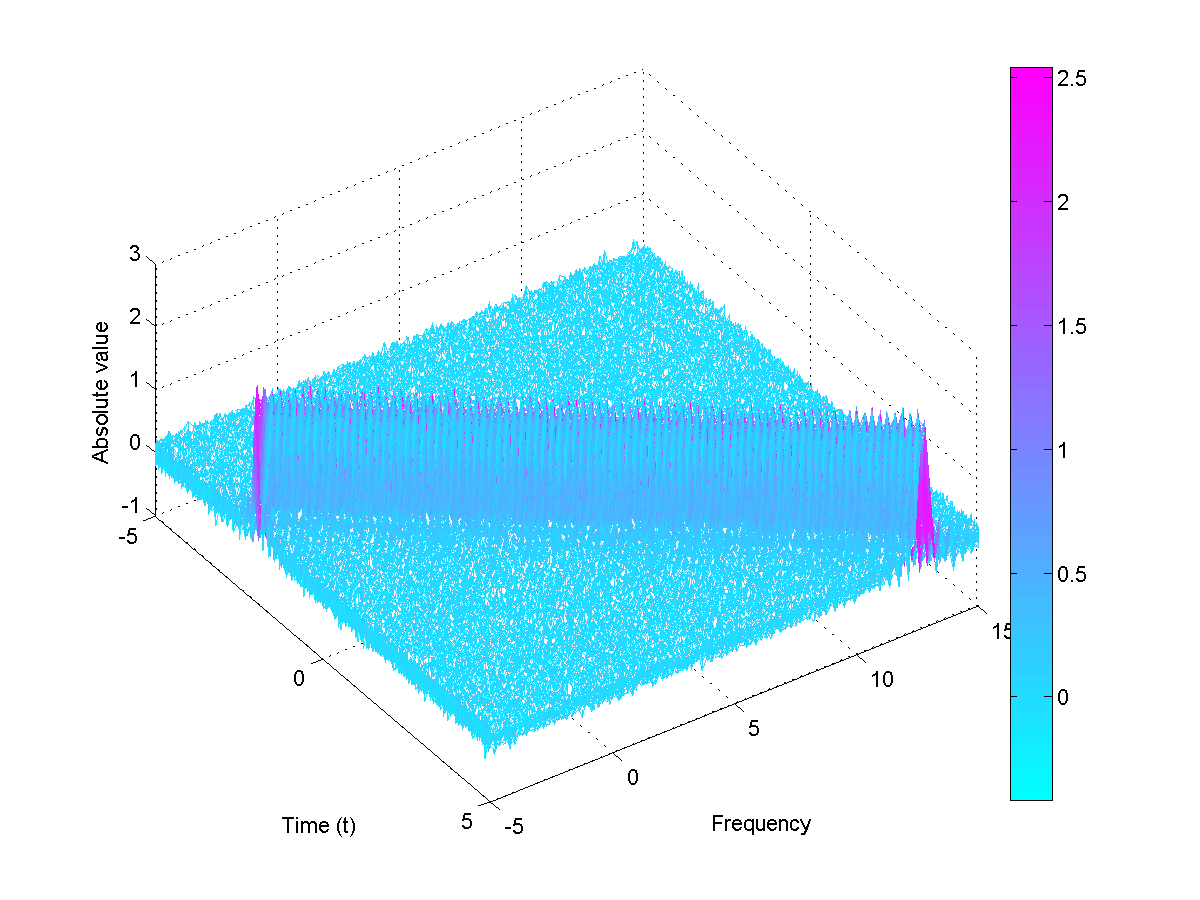}}
    \subfigure[Contour picture of new NFrAF of ${\bf x}(t)$ at $\theta=\frac{\pi}{4},k=2$]{\includegraphics[width=0.4\textwidth]{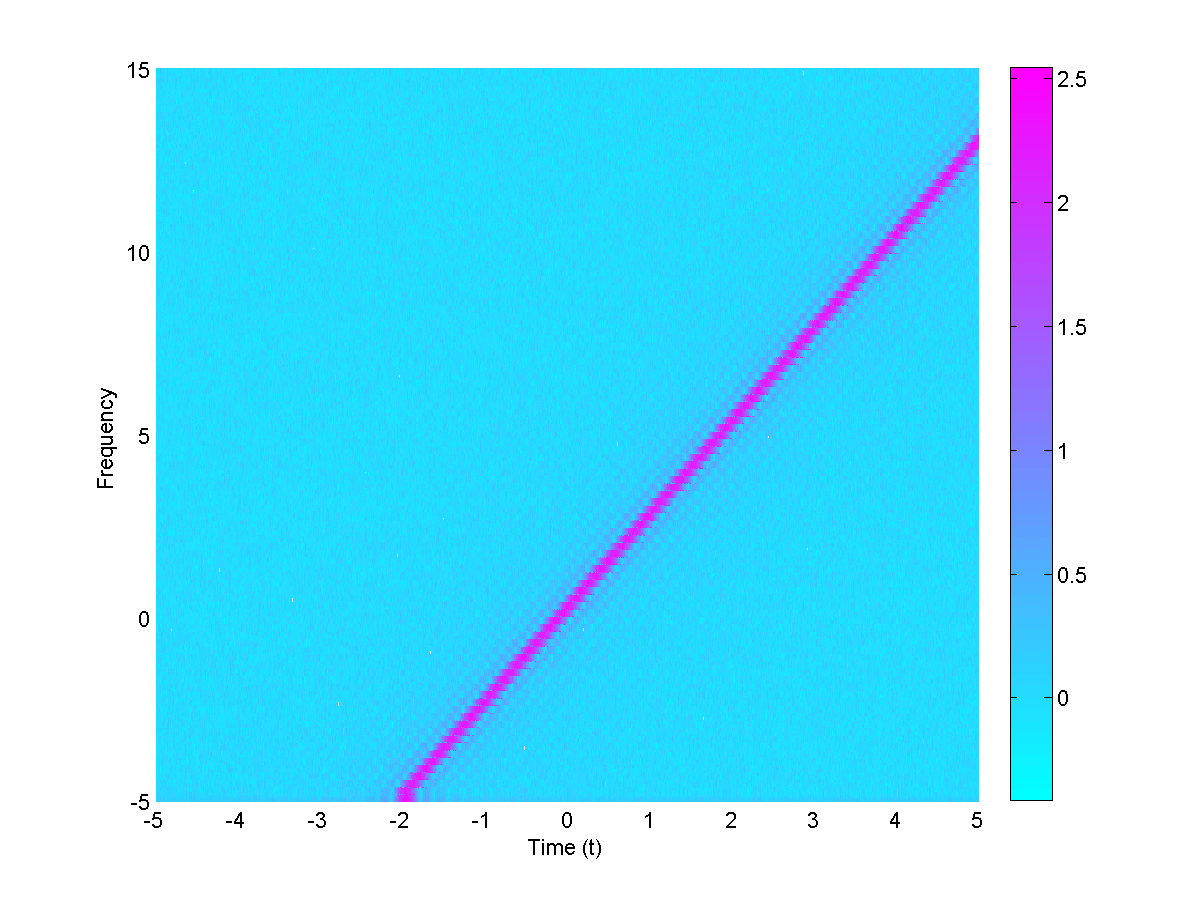}}
    \subfigure[The new NFrAF of ${\bf x}(t)$ at $\theta=\frac{\pi}{2},k=2$]{\includegraphics[width=0.4\textwidth]{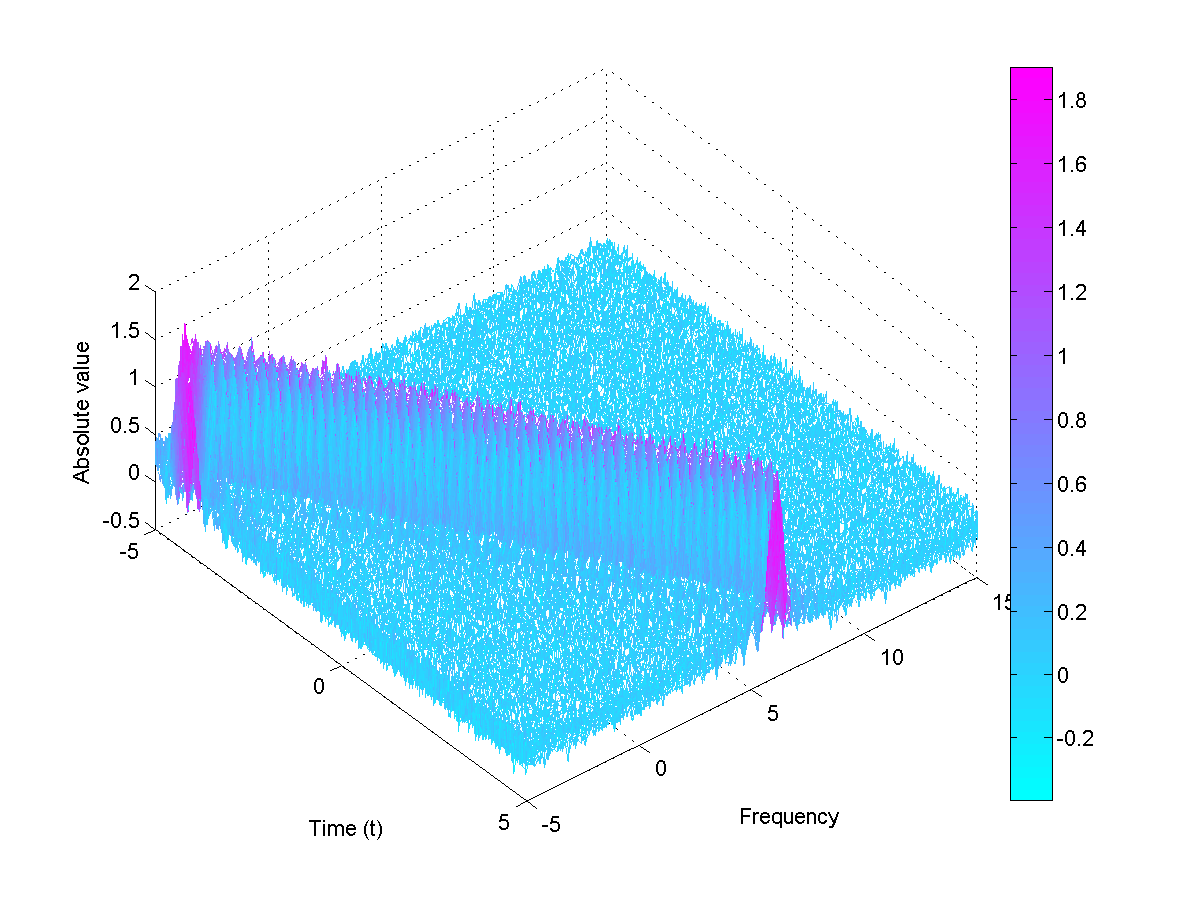}}
    \subfigure[Contour picture of new NFrAF of ${\bf x}(t)$ at $\theta=\frac{\pi}{2},k=2$]{\includegraphics[width=0.4\textwidth]{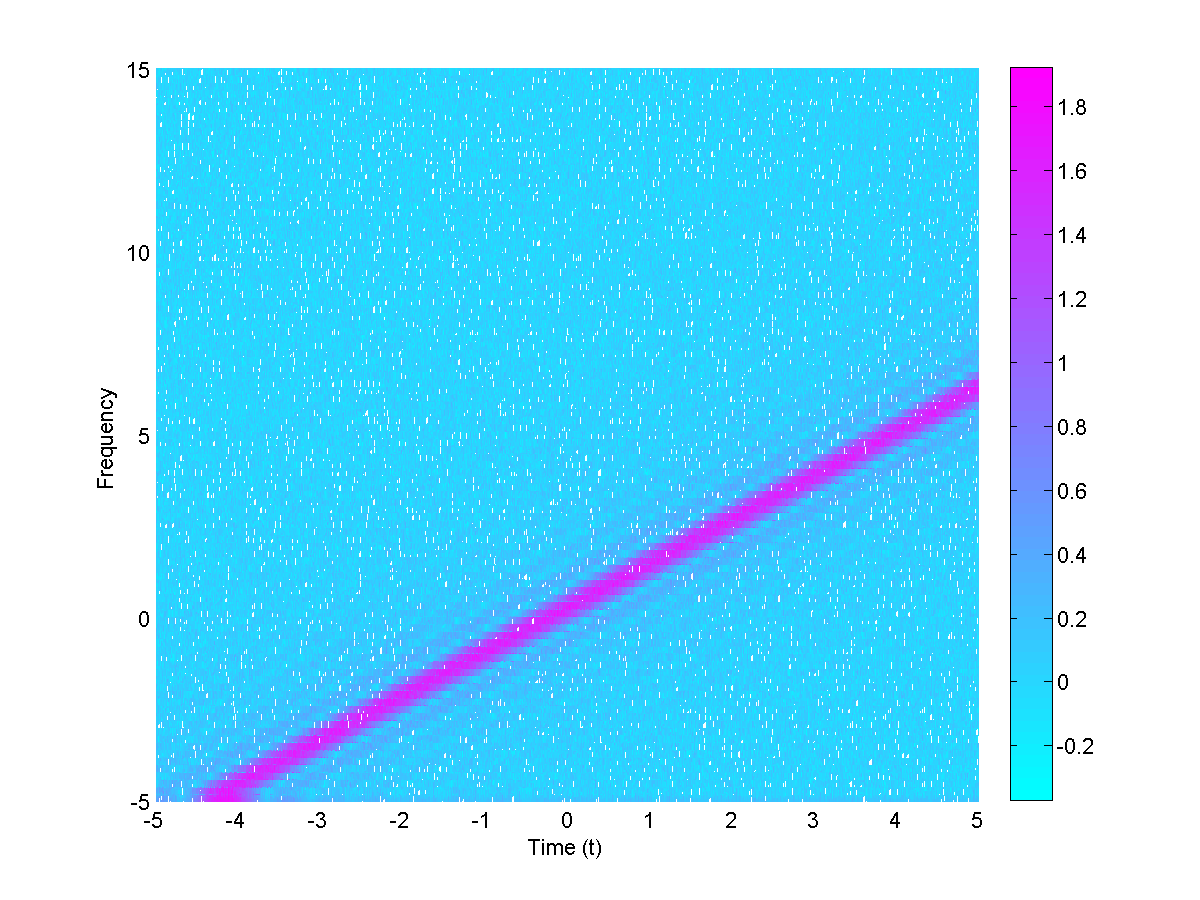}}
    \caption{Absolute value of the new NFrAF   of a mono-component signal ${\bf x}(t)=e^{i(0.2t+0.3t^2)}$ corresponding to particular choices of $\theta$ and $k$  with
$\Lambda_0=1,$  $T=10$ and SNR =10dB. }\label{fig2}
\end{framed}
\end{figure}

\begin{figure}[!htbp]
\begin{framed}
\centering
 \subfigure[The new NFrAF of ${\bf x}(t)$ at $\theta=\frac{\pi}{4},k=\frac{1}{2}$]{\includegraphics[width=0.4\textwidth]{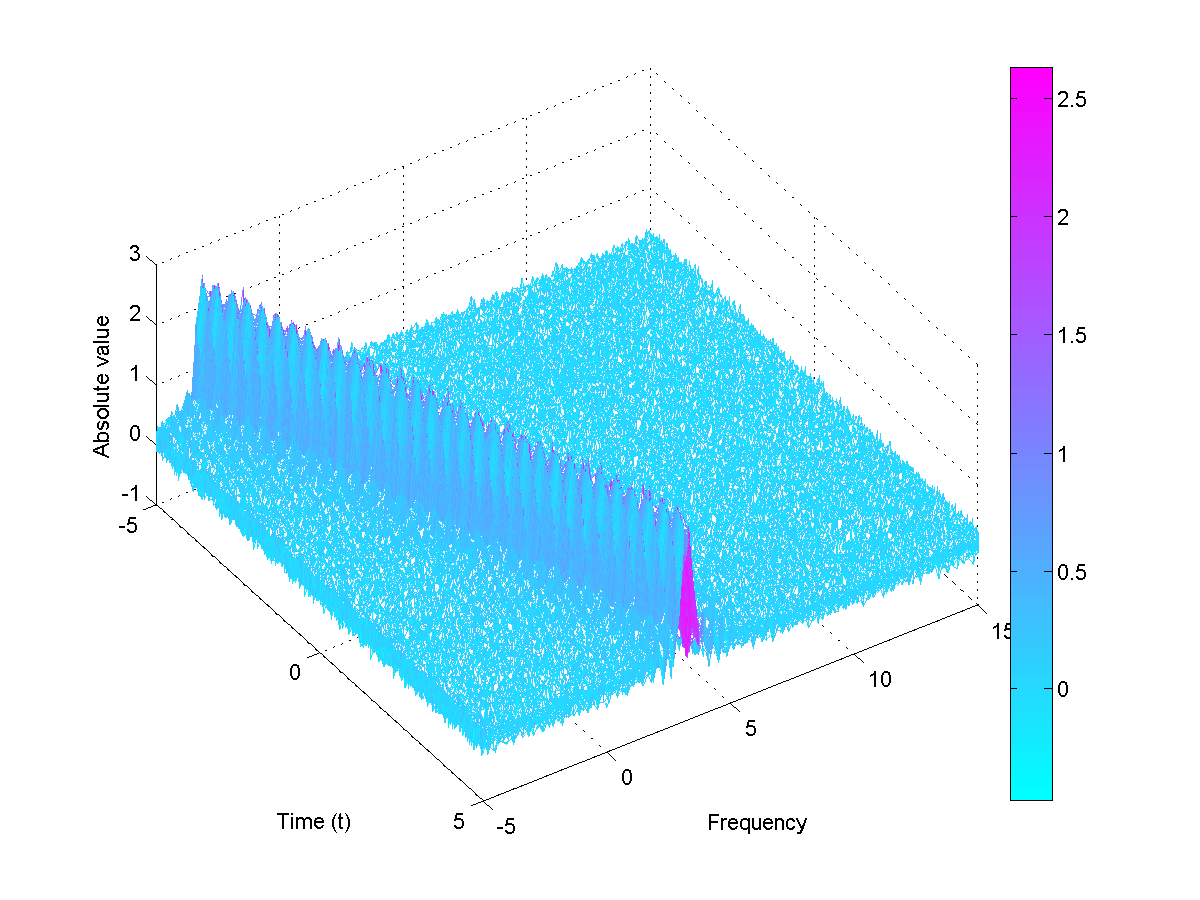}}
    \subfigure[Contour picture of new NFrAF of ${\bf x}(t)$ at $\theta=\frac{\pi}{4},k=\frac{1}{2}$]{\includegraphics[width=0.4\textwidth]{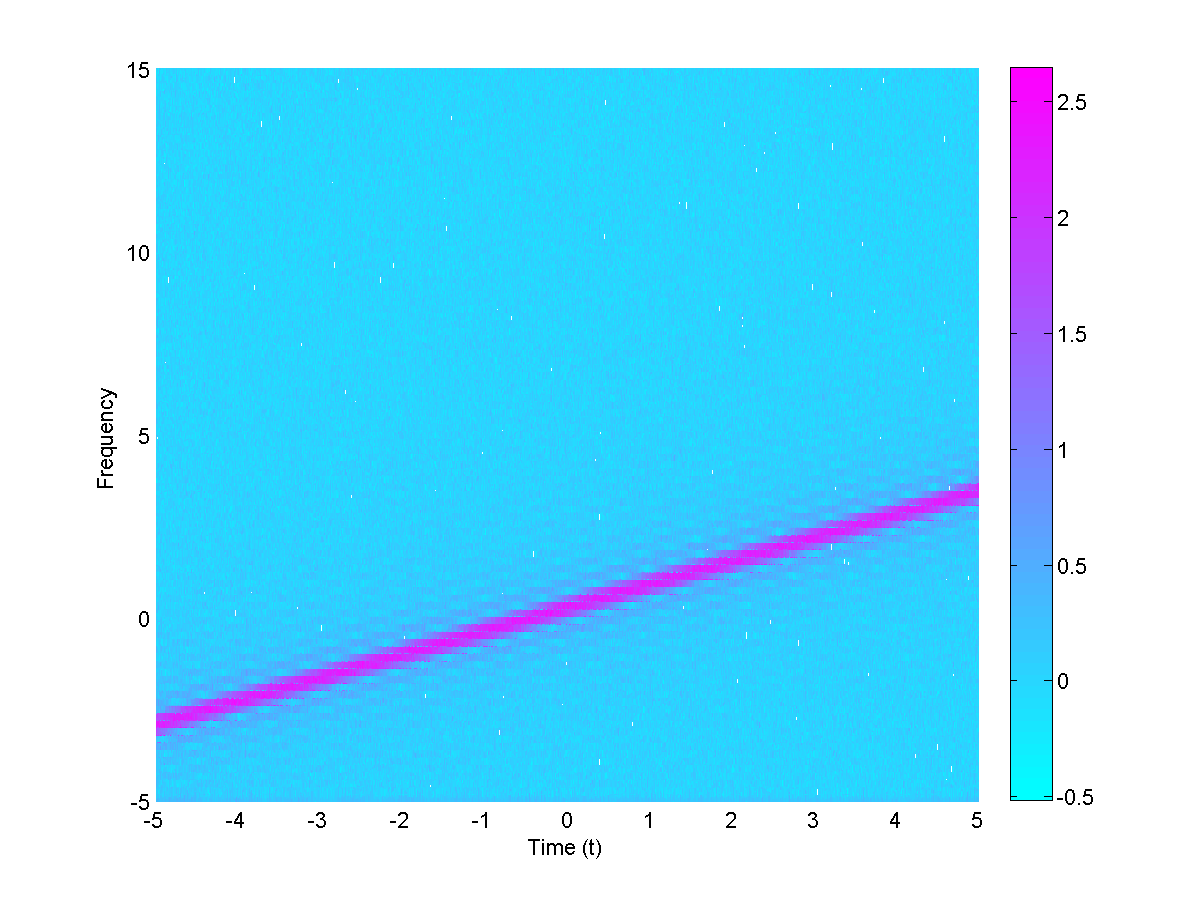}}
    \subfigure[The new NFrAF of ${\bf x}(t)$ at $\theta=\frac{\pi}{2},k=\frac{1}{2}$]{\includegraphics[width=0.4\textwidth]{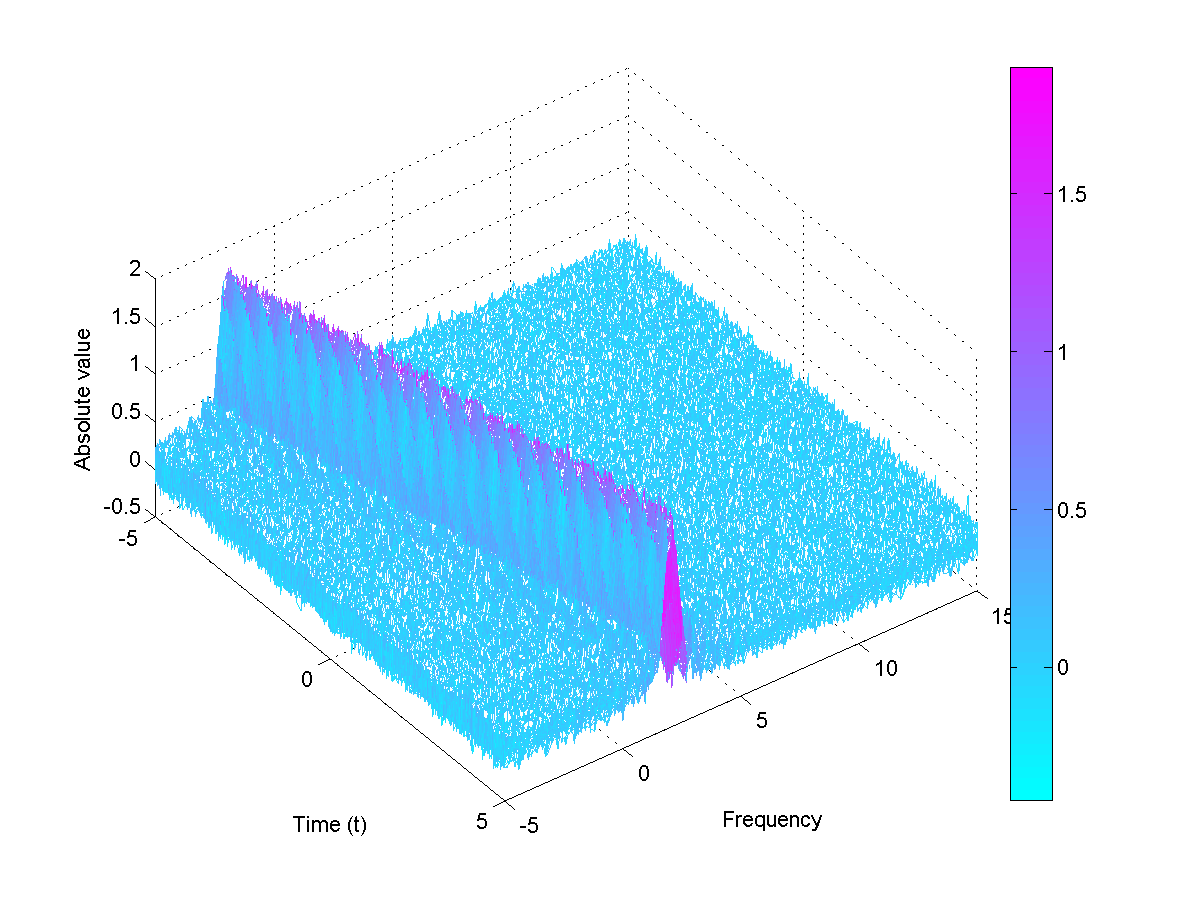}}
    \subfigure[Contour picture of new NFrAF of ${\bf x}(t)$ at $\theta=\frac{\pi}{2},k=\frac{1}{2}$]{\includegraphics[width=0.4\textwidth]{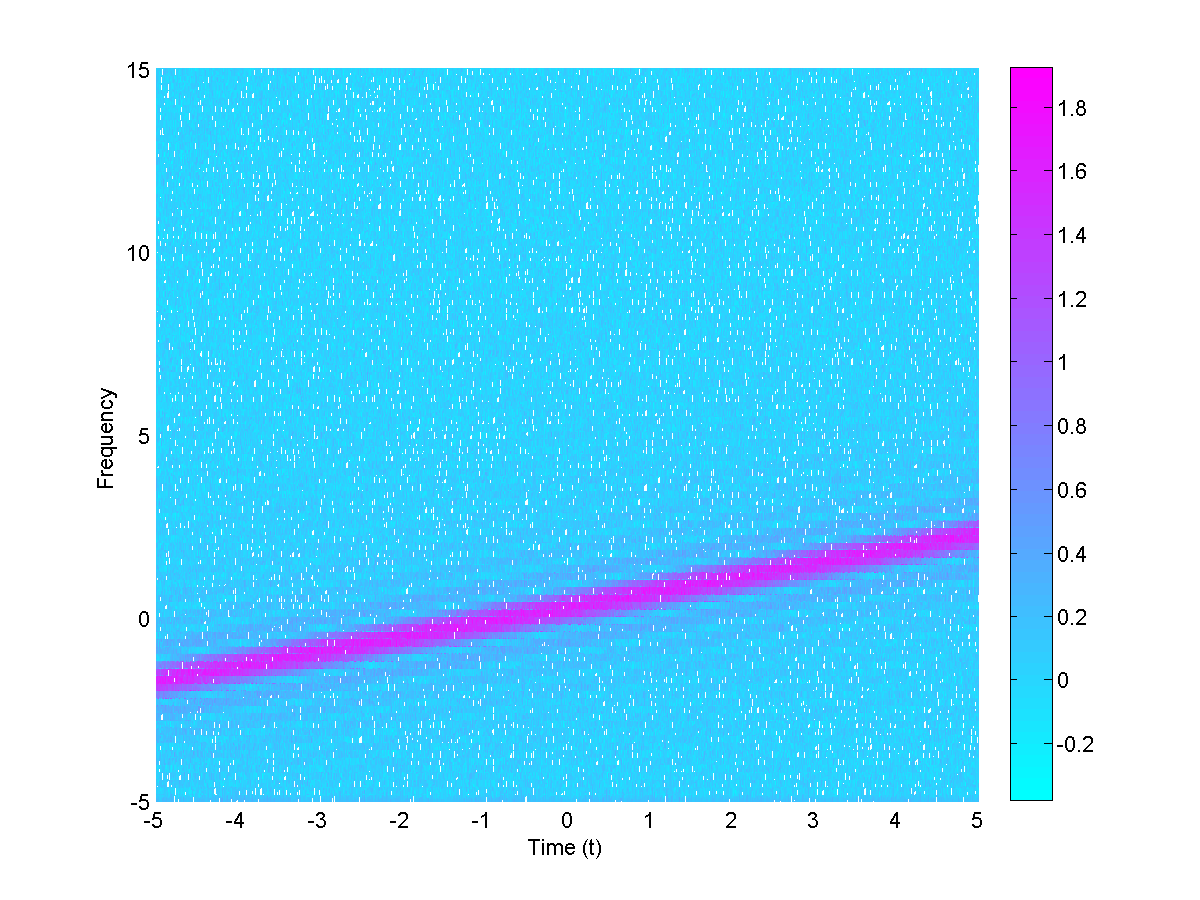}}
    \caption{Absolute value of the new NFrAF   of a mono-component signal ${\bf x}(t)=e^{i(0.2t+0.3t^2)}$ corresponding to particular choices of $\theta$ and $k$   with
$\Lambda_0=1,$  $T=10$ and SNR =10dB. }\label{fig3}
\end{framed}
\end{figure}

\item {\bf Multi-component LFM signal:} Every multi-component LFM signal ${\bf x}(t)$ can be written as a summation of $n$ single component LFM signals as
\begin{equation}\label{bi 1}
{\bf x}(t)=\sum_{j=1}^{n}{\bf x}_j(t), \quad {\bf x}_j(t)=\Lambda_je^{i(\mu_jt+\nu_jt^2)}, \quad -\frac{T}{2}\le t\le \frac{T}{2}.
\end{equation}
 Now using the non-linearity property (Table \ref{tab1}), the   NFrAF of multi-component LFM signal can be written as
 \begin{eqnarray}
\nonumber &&\mathcal  A^{\theta,k}_{\bf x}(\tau,u)=\sum_{j=1}^{n}\mathcal  A^{\theta,k}_{{\bf x}_j}(\tau,u)+\sum_{j_1\ne j_2=1}^{n}\mathcal  A^{\theta,k}_{{\bf x}_{j_1},{\bf x}_{j_2}}(\tau,u)
 \end{eqnarray}
 The first sum in last equation stands for the auto-terms of one-component signals, whereas the second represent the cross terms that are given by

\begin{eqnarray}
\nonumber &&\mathcal  A^{\theta,k}_{{\bf x}_{j_1},{\bf x}_{j_2}}(\tau,u)\\
\nonumber&&=|\Omega_\theta|^2\int_{-\frac{T}{2}}^{\frac{T}{2}}\left({\bf x}_{j_1}\underset{k\frac{\tau}{2}}{\otimes}{\bf x}_{j_2}^*\right)(t)e^{it(k\tau\cot\theta-u\csc\theta)}dt \\
\nonumber&&=\Lambda_{j_1}\Lambda^*_{j_2}|\Omega_\theta|^2\int_{-\frac{T}{2}}^{\frac{T}{2}}e^{i\left[\mu_{j_1}\left(t+k\frac{\tau}{2}\right)+\nu_{j_1}\left(t+k\frac{\tau}{2}\right)^2\right]}e^{-i\left[\mu_{j_2}\left(t-k\frac{\tau}{2}\right)+\nu_{j_2}\left(t-k\frac{\tau}{2}\right)^2\right]}e^{it(k\tau\cot\theta-u\csc\theta)}dt \\
\nonumber&&=\Lambda_{j_1}\Lambda^*_{j_2}|\Omega_\theta|^2e^{i\left(\frac{\nu_{j_1}-\nu_{j_2}}{4}k^2\tau^2+\frac{(\mu_{j_1}+\mu_{j_2})}{2}k\tau\right)}\int_{-\frac{T}{2}}^{\frac{T}{2}}e^{i[(\nu_{j_1}-\nu_{j_2})t^2]}e^{i\left[k(\nu_{j_1}+\nu_{j_2}+\cot\theta)\tau-u\csc\theta+(\mu_{j_1}-\mu_{j_2})\right]t}dt. \\
\end{eqnarray}

Hence the NFrAF of a multi-component signal ${\bf x}(t)=\sum_{j=1}^{n}{\bf x}_j(t)$ is given by
\begin{eqnarray*}
 \mathcal  A^{\theta,k}_{\bf x}(\tau,u)&=&\sum_{j=1}^{n}\mathcal  A^{\theta,k}_{{\bf x}_j}(\tau,u)+\sum_{j_1\ne j_2=1}^{n}\mathcal  A^{\theta}_{{\bf x}_{j_1},{\bf x}_{j_2}}(\tau,u)\\\\
 &=&\sum_{j=1}^{n}{|\Omega_\theta|^2|\Lambda_j|^2T}e^{ik\mu_j\tau} sinc\left\{\frac{T}{2}\left[k(2\nu_j+\cot\theta )\tau-u\csc\theta\right]\right\}\\\\
&& +\sum_{j_1\ne j_2=1}^{n}\Lambda_{j_1}\Lambda^*_{j_2}|\Omega_\theta|^2e^{i\left(\frac{\nu_{j_1}-\nu_{j_2}}{4}k^2\tau^2+\frac{(\mu_{j_1}+\mu_{j_2})}{2}k\tau\right)}\\
\qquad\qquad\qquad&&\times\int_{-\frac{T}{2}}^{\frac{T}{2}}e^{i[(\nu_{j_1}-\nu_{j_2})t^2]}e^{i\left[k(\nu_{j_1}+\nu_{j_2}+\cot\theta)\tau-u\csc\theta+(\mu_{j_1}-\mu_{j_2})\right]t}dt.\\
\end{eqnarray*}



\begin{figure}[!htbp]
\centering
\begin{framed}
\begin{minipage}[b]{0.48\linewidth}
	\centering
	\includegraphics[width=\linewidth]{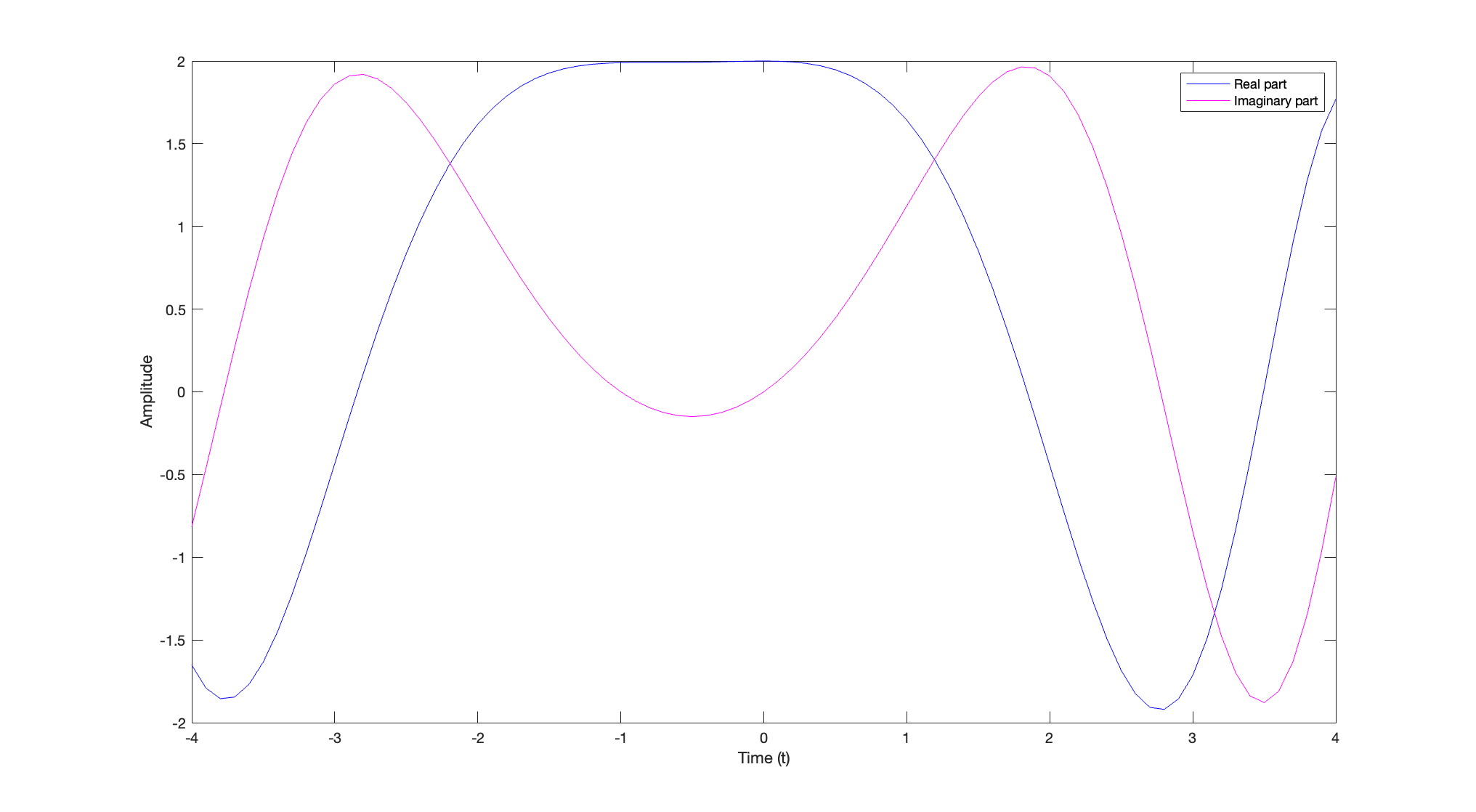}
	{\rm (a)  }
\end{minipage}
\hfill
\begin{minipage}[b]{0.48\linewidth}
	\centering
	\includegraphics[width=\linewidth]{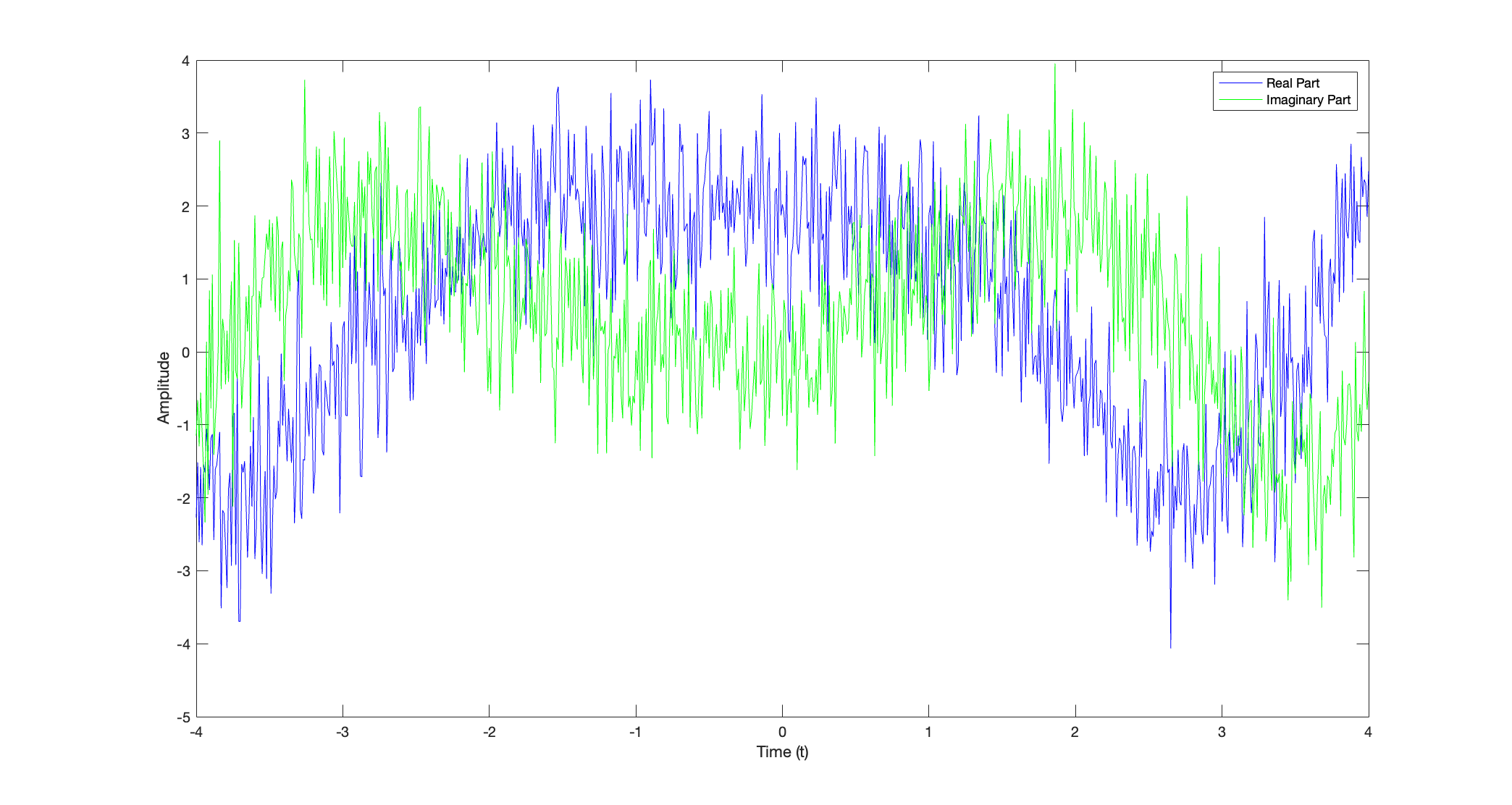}
	{\rm (b)}
\end{minipage}

\caption{\small  (a)Real and Imaginary parts of bi-component signal ${\bf x}(t)=e^{i(0.2t+0.3t^2)}+e^{i(0.4t+0.3t^2)}$ (b) Real and imaginary parts of the same signal corrupted by additive noise  at $5$dB SNR }
\label{S3F4}
\end{framed}
\end{figure}

 The multi-component LFM signal can still be recognized even though the existence of cross terms lowers detection performance because the auto-terms $\mathcal  A^{\theta,k}_{{\bf x}_j}$ can produce impulses that the cross terms $\mathcal  A^{\theta,k}_{{\bf x}_{j_1},{\bf x}_{j_2}},$ cannot. This implies that the innovative NFrAF can also be usefully applied in multi-component LFM signal detection.For the special case $\nu_1=\nu_2=....=\nu_n=\nu,$ we have
 \begin{eqnarray}\label{eqn multiwd}
\nonumber \mathcal  A^{\theta,k}_{\bf x}(\tau,u)\\
\nonumber&=&\sum_{j=1}^{n}\mathcal  A^{\theta}_{{\bf x}_j}(\tau,u)+\sum_{j_1\ne j_2=1}^{n}\mathcal  A^{\theta,k}_{{\bf x}_{j_1},{\bf x}_{j_2}}(\tau,u)\\
 \nonumber &=&\sum_{j=1}^{n}{|\Omega_\theta|^2|\Lambda_j|^2T}e^{ik\mu_j\tau} sinc\left\{\frac{T}{2}\left[k(2\nu+\cot\theta )\tau-u\csc\theta\right]\right\}\\
\nonumber && +\sum_{j_1\ne j_2=1}^{n}\Lambda_{j_1}\Lambda^*_{j_2}|\Omega_\theta|^2e^{i\left(\frac{\mu_{j_1}+\mu_{j_2}}{2}\right)k\tau}\\
\nonumber &&\qquad\qquad\qquad\qquad\times sinc\left\{\frac{T}{2}\left[k(2\nu +\cot\theta)\tau-u\csc\theta+(\mu_{j_1}-\mu_{j_2})\right]   \right\}. \\
\end{eqnarray}

\begin{figure}[!htbp]
\begin{framed}
\centering
     \subfigure[The new NFrAF of signal (\ref{bisig}) at $\theta=\frac{\pi}{4}$ and $k=1$]{\includegraphics[width=0.4\textwidth]{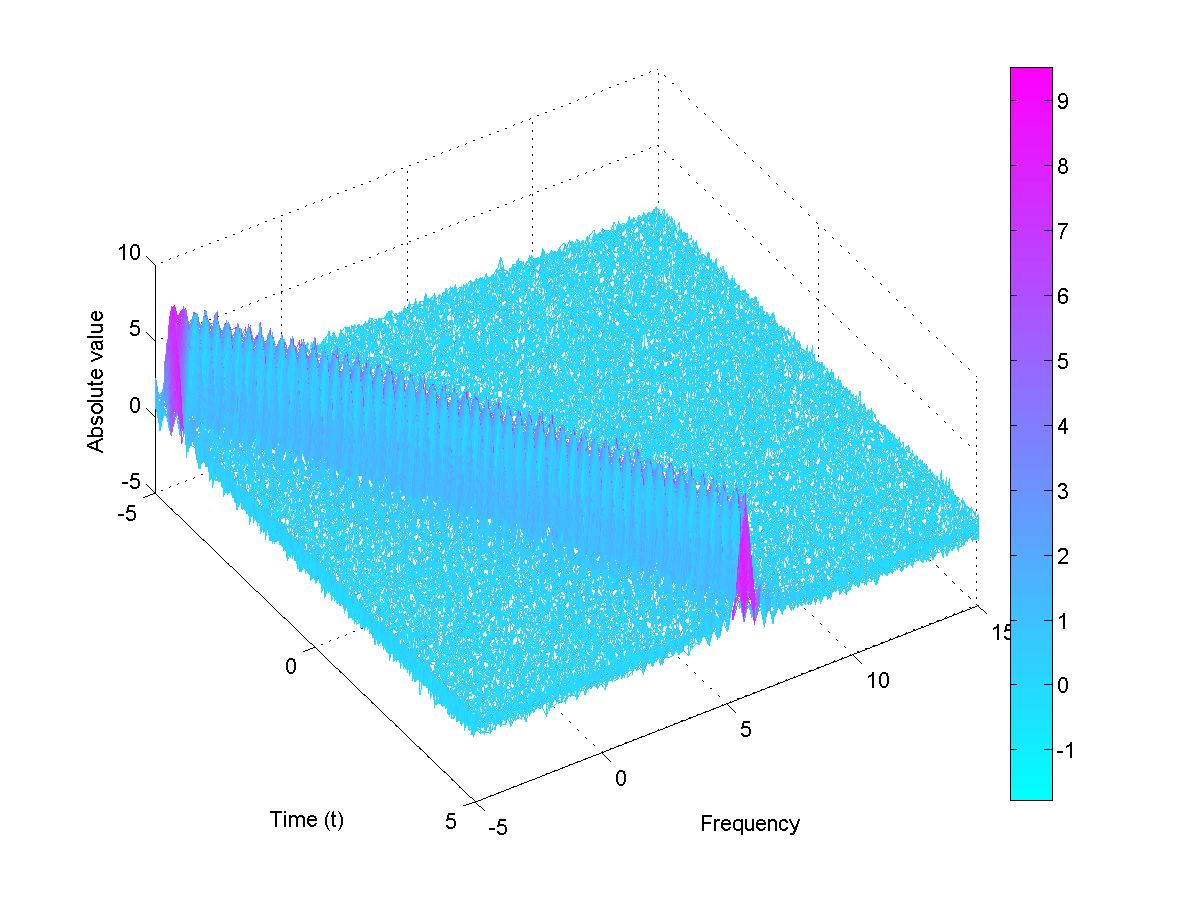}}
    \subfigure[Contour plot of signal (\ref{bisig}) at $\theta=\frac{\pi}{4}$ and $k=1$]{\includegraphics[width=0.4\textwidth]{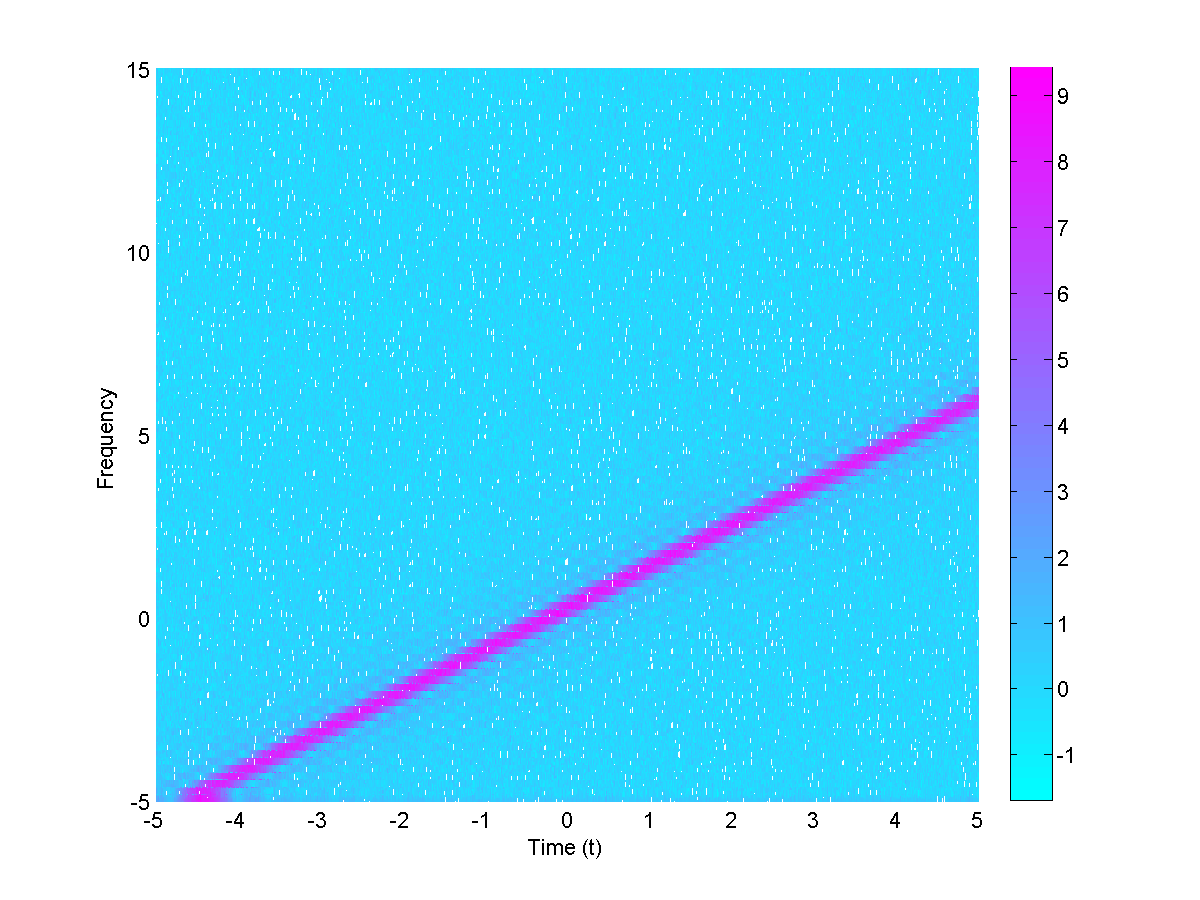}}
      \subfigure[The new NFrAF of signal (\ref{bisig}) at $\theta=\frac{\pi}{2}$ and $k=1$]{\includegraphics[width=0.4\textwidth]{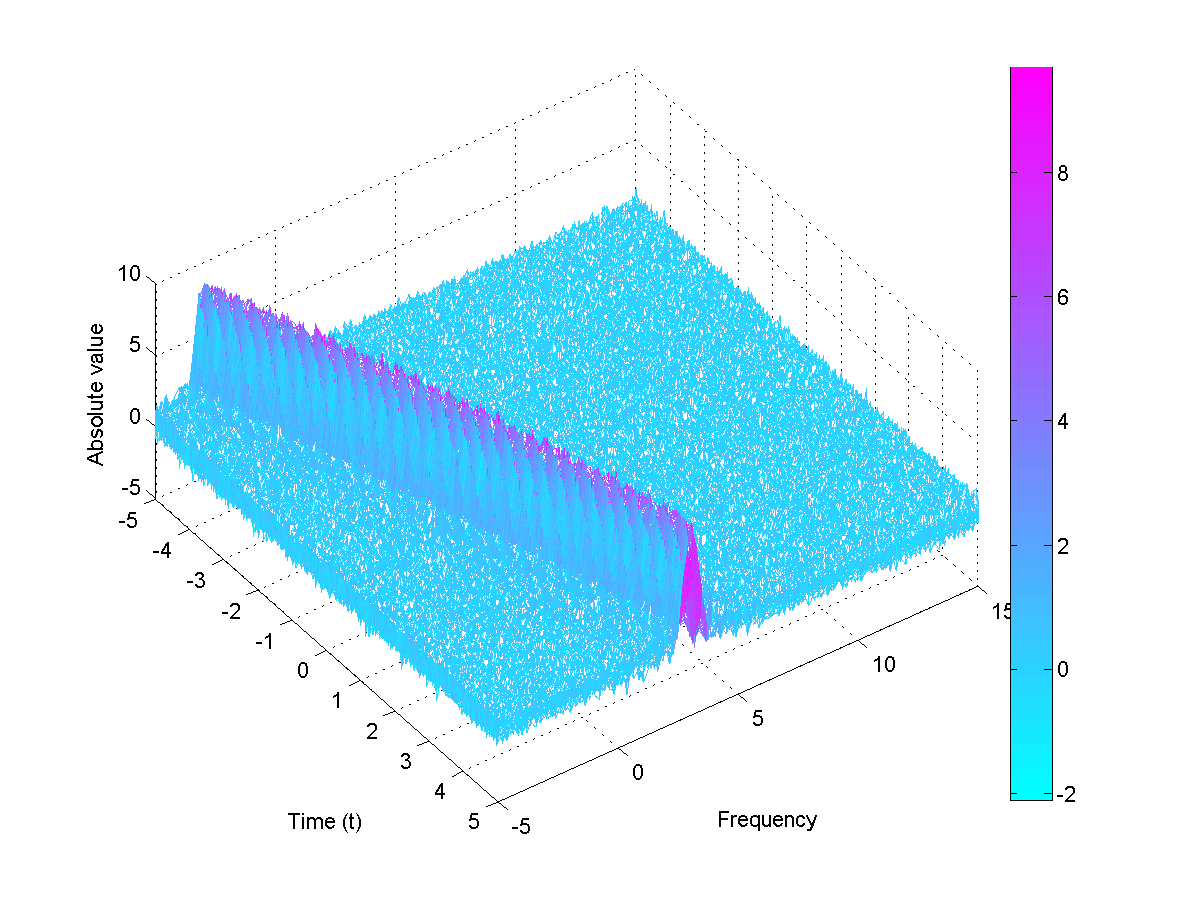}}
    \subfigure[Contour plot of signal (\ref{bisig}) at $\theta=\frac{\pi}{2}$ and $k=1$]{\includegraphics[width=0.4\textwidth]{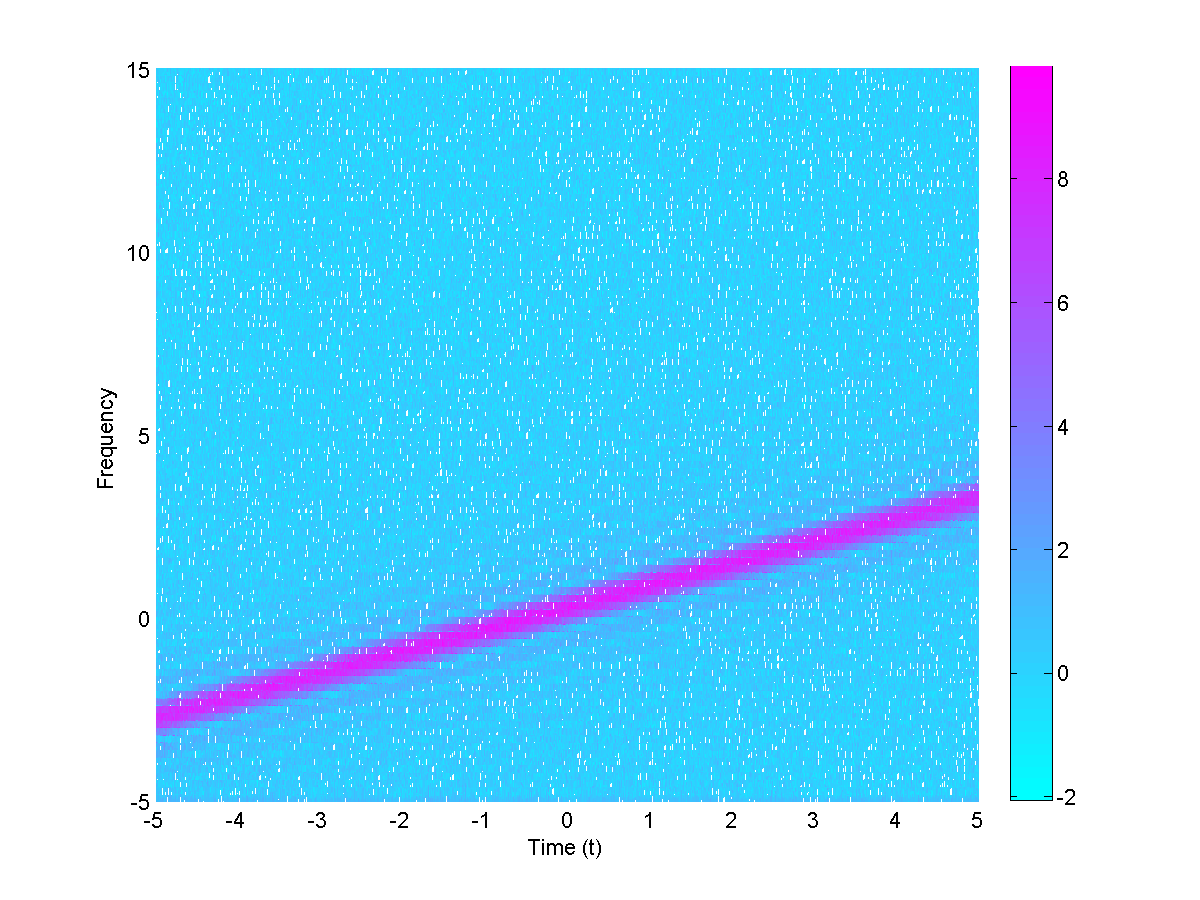}}
     \caption{Absolute value of the new NFrAF bi-component LFM signal (\ref{bisig})
with SNR =10dB db corresponding to particular choices of $\theta$ and $k$.  }\label{fig4}
\end{framed}
\end{figure}

\begin{figure}[!htbp]
\begin{framed}
\centering
\subfigure[The new NFrAF of signal (\ref{bisig}) at $\theta=\frac{\pi}{4}$ and $k=2$]{\includegraphics[width=0.4\textwidth]{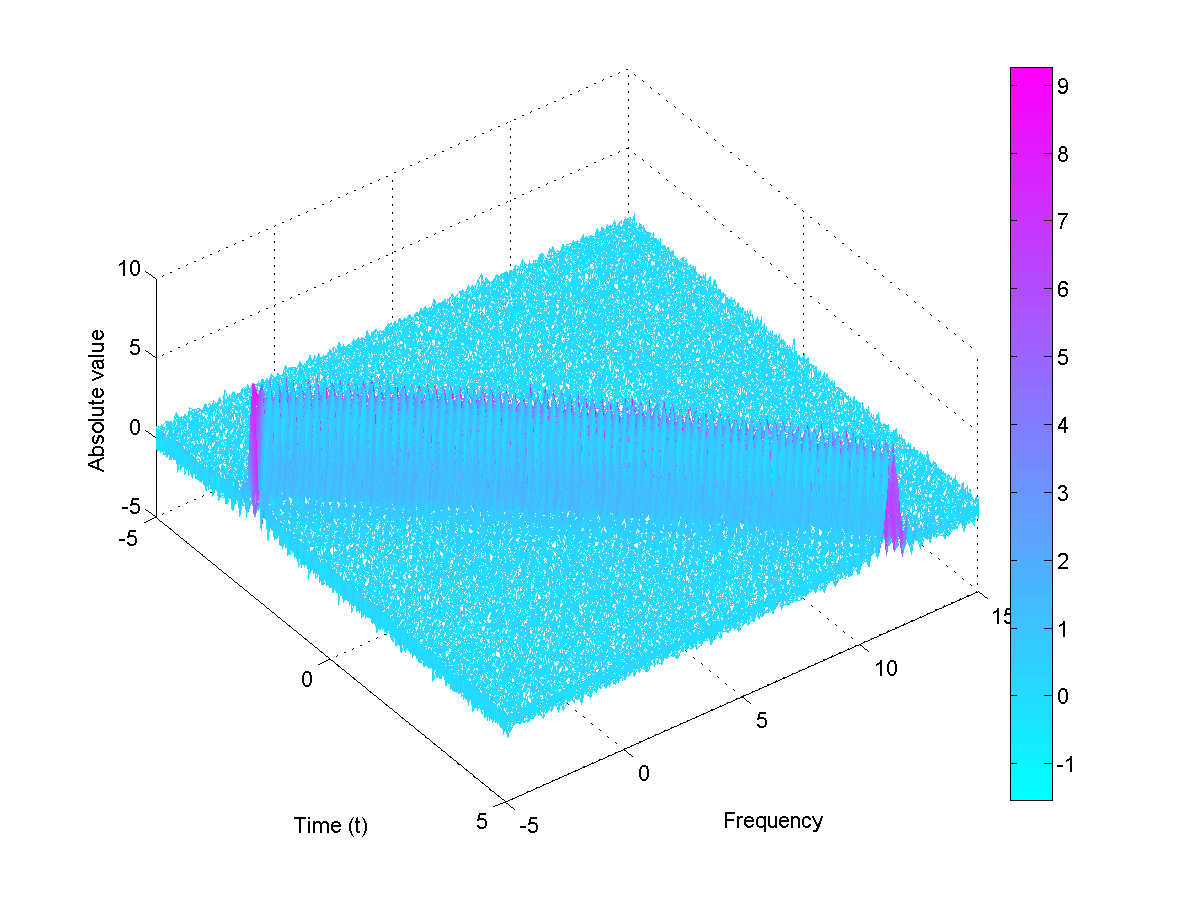}}
    \subfigure[Contour plot of signal (\ref{bisig}) at $\theta=\frac{\pi}{4}$ and $k=2$]{\includegraphics[width=0.4\textwidth]{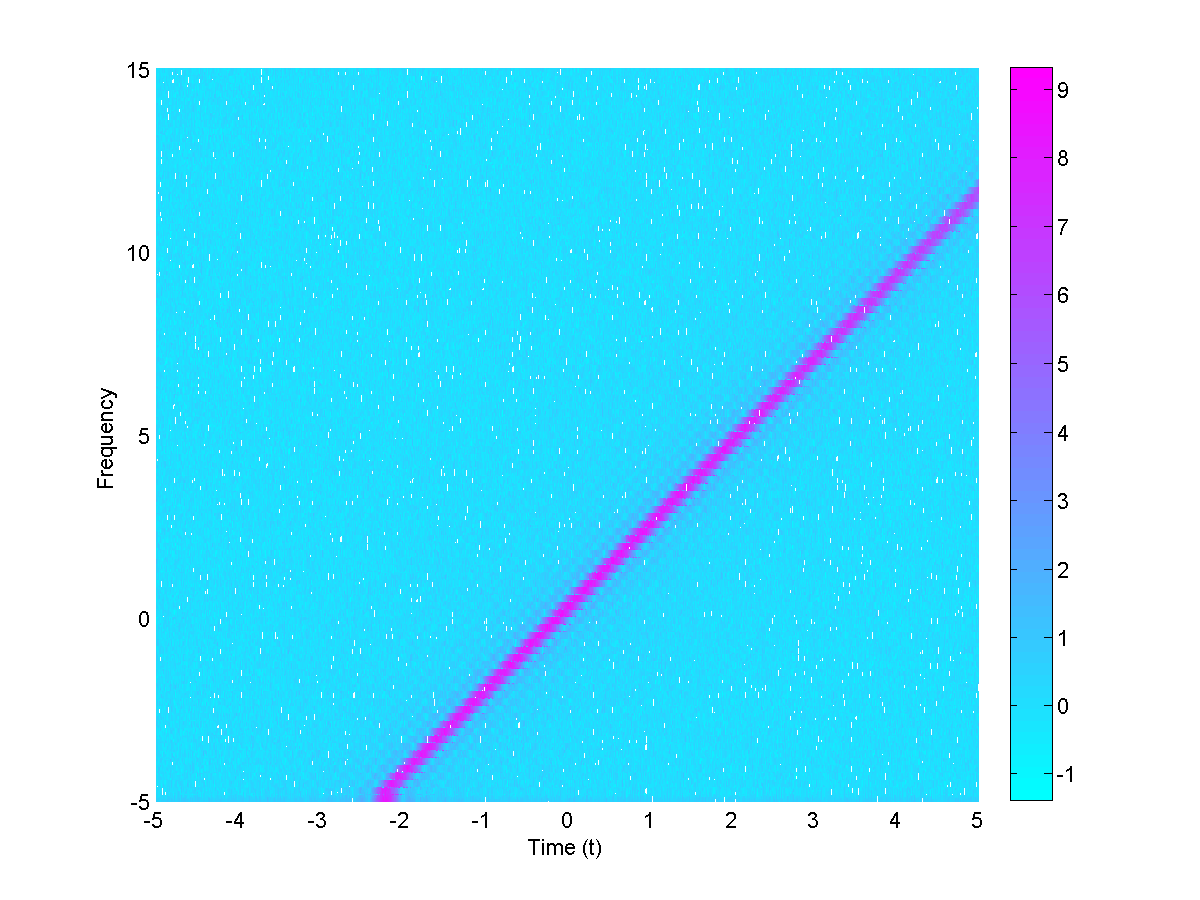}}
    \subfigure[The new NFrAF of signal (\ref{bisig}) at $\theta=\frac{\pi}{2}$ and $k=2$]{\includegraphics[width=0.4\textwidth]{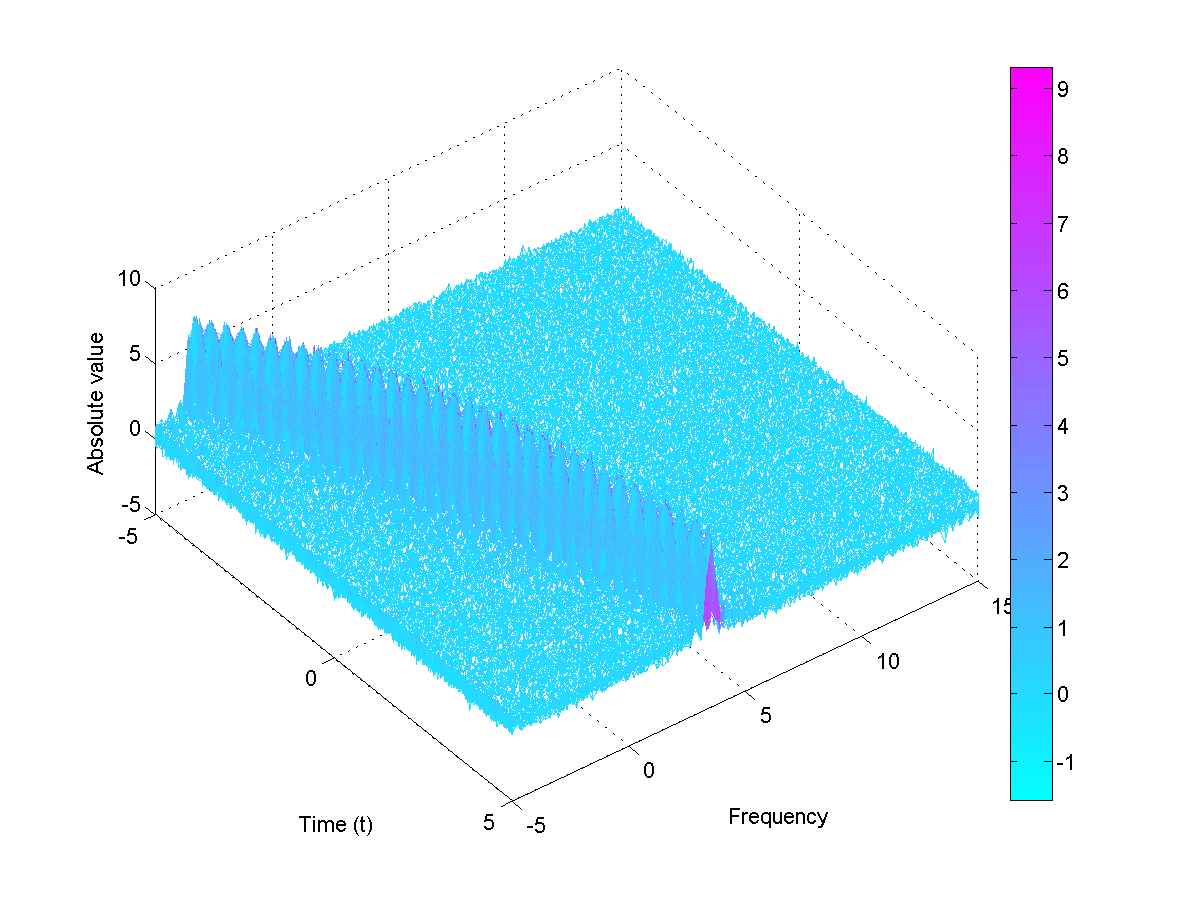}}
    \subfigure[Contour plot of signal (\ref{bisig}) at $\theta=\frac{\pi}{2}$ and $k=2$]{\includegraphics[width=0.4\textwidth]{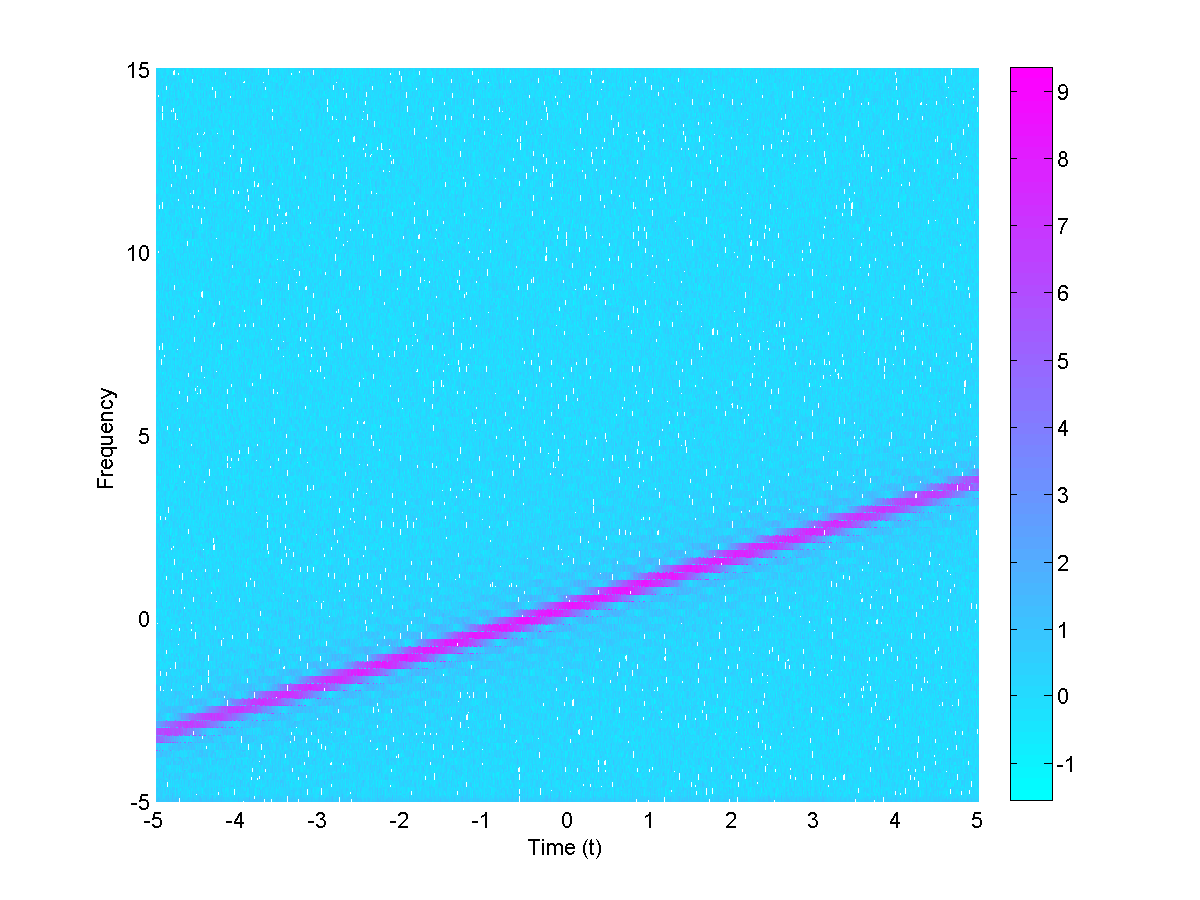}}
    \caption{Absolute value of the new NFrAF bi-component LFM signal (\ref{bisig})
with SNR =10dB db corresponding to particular choices of $\theta$ and $k$. }\label{fig5}
\end{framed}
\end{figure}

\begin{figure}[t]
\begin{framed}
\centering
 \subfigure[The new NFrAF of signal (\ref{bisig}) at $\theta=\frac{\pi}{4}$ and $k=\frac{1}{2}$]{\includegraphics[width=0.4\textwidth]{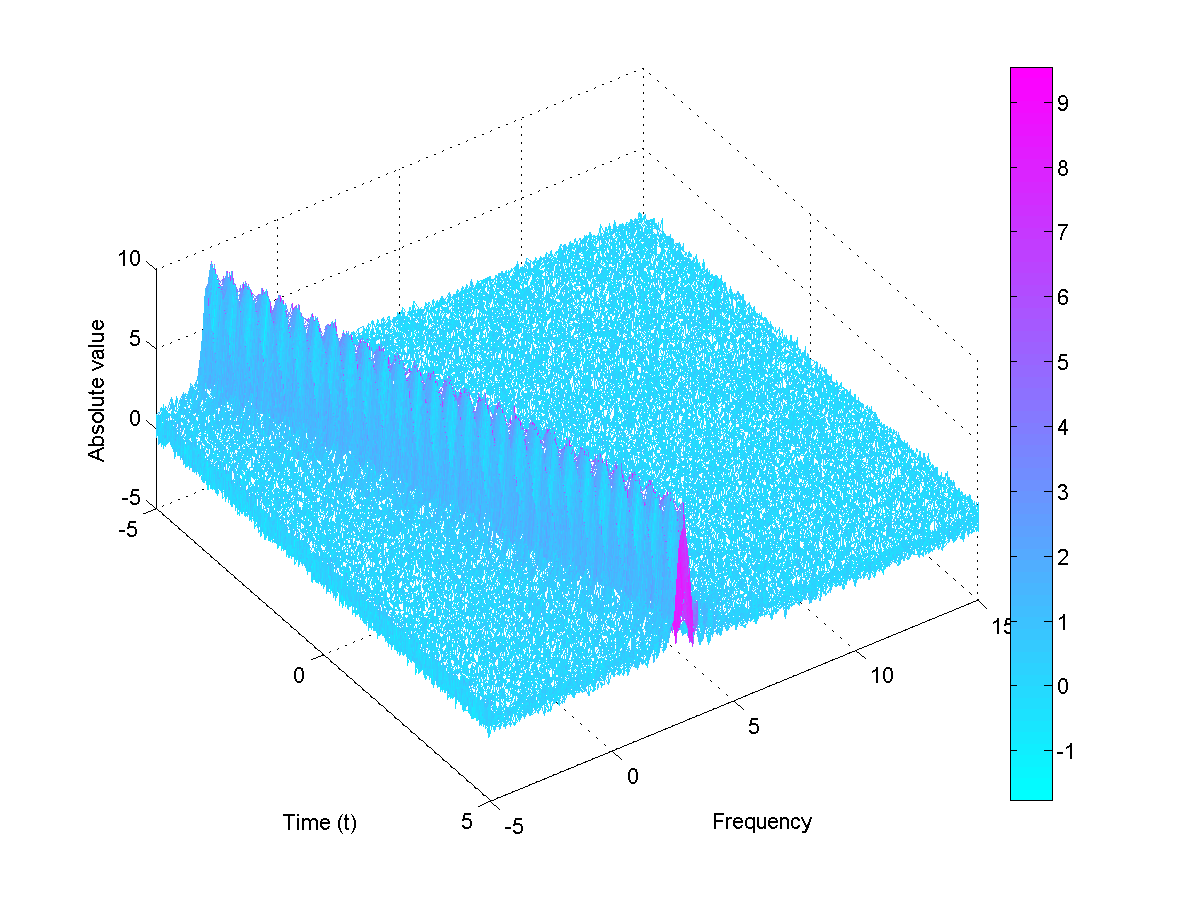}}
    \subfigure[Contour plot of signal (\ref{bisig}) at $\theta=\frac{\pi}{4}$ and $k=\frac{1}{2}$]{\includegraphics[width=0.4\textwidth]{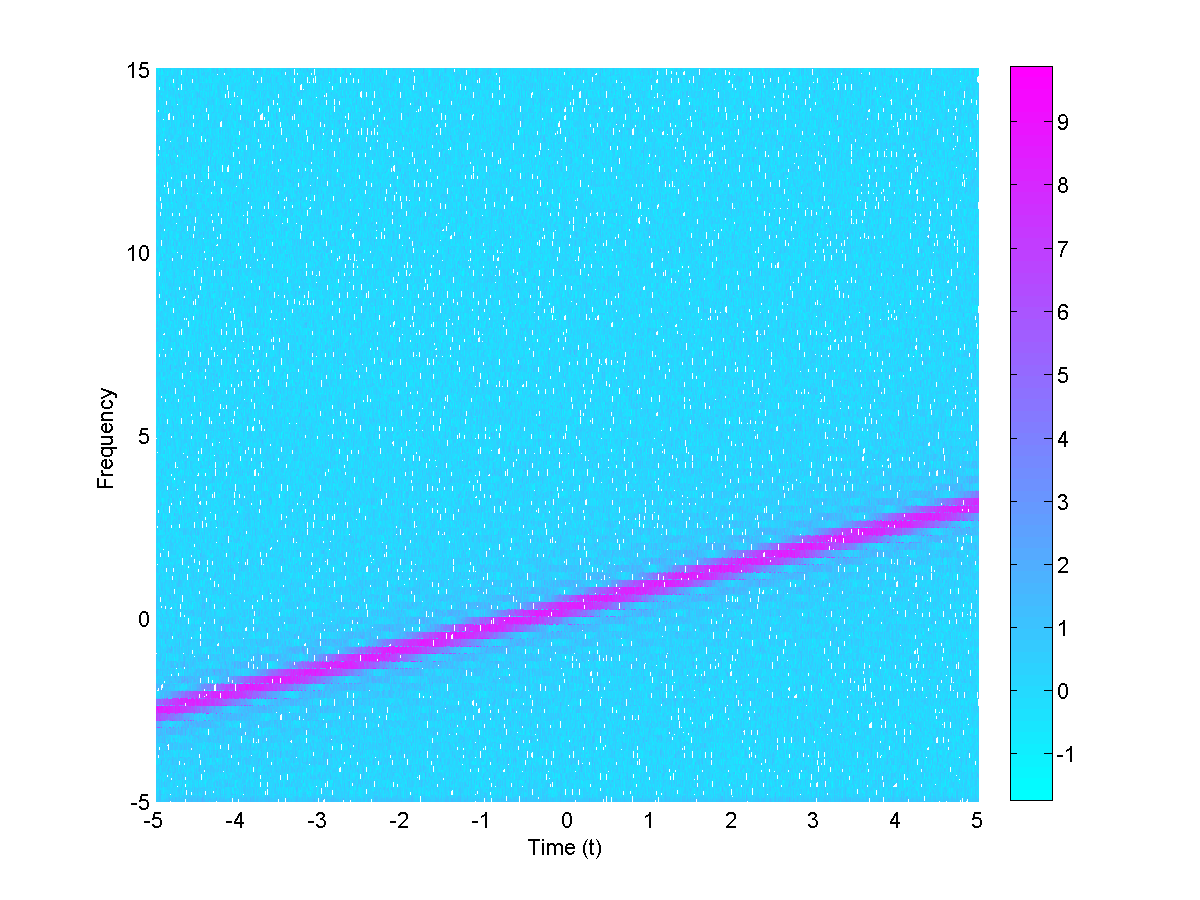}}
    \subfigure[The new NFrAF of signal (\ref{bisig}) at $\theta=\frac{\pi}{2}$ and $k=\frac{1}{2}$]{\includegraphics[width=0.4\textwidth]{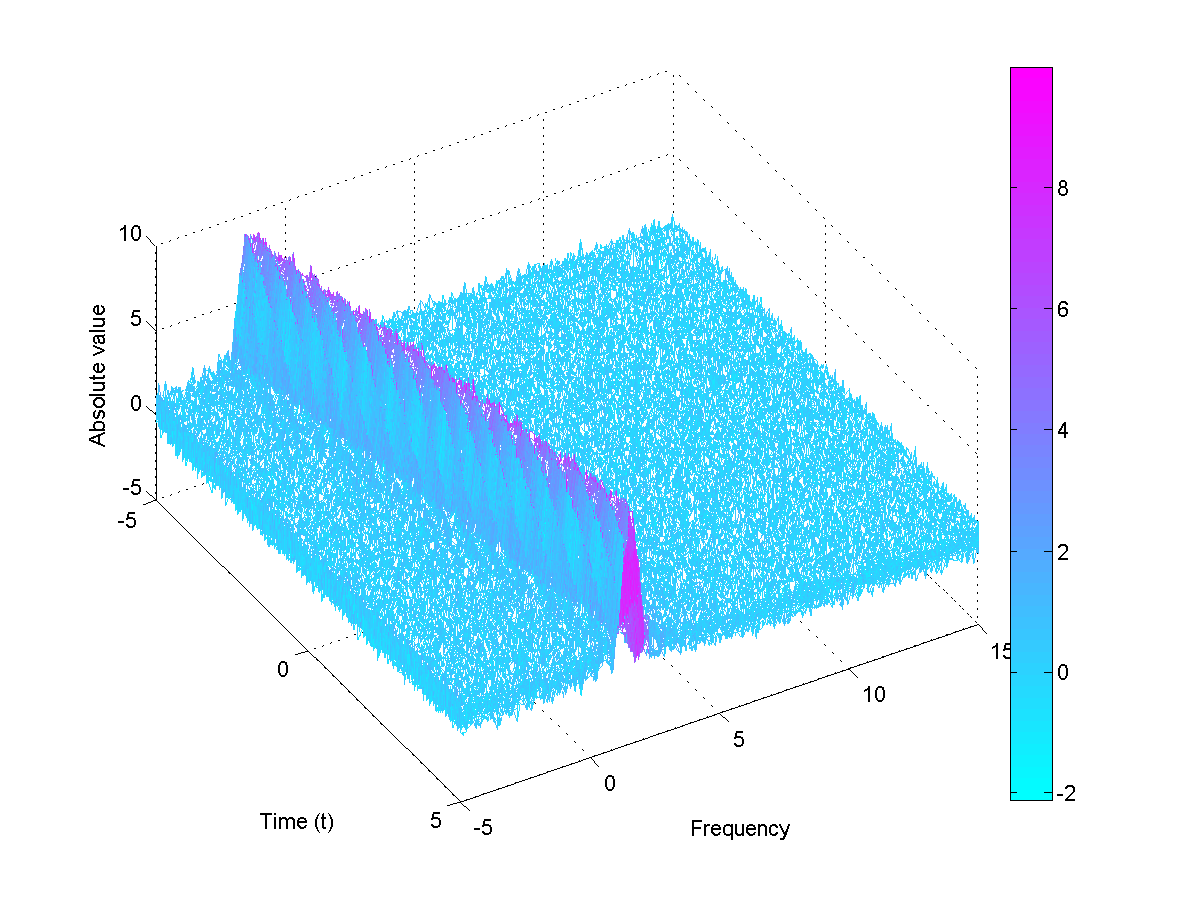}}
    \subfigure[Contour plot of signal (\ref{bisig}) at $\theta=\frac{\pi}{2}$ and $k=\frac{1}{2}$]{\includegraphics[width=0.4\textwidth]{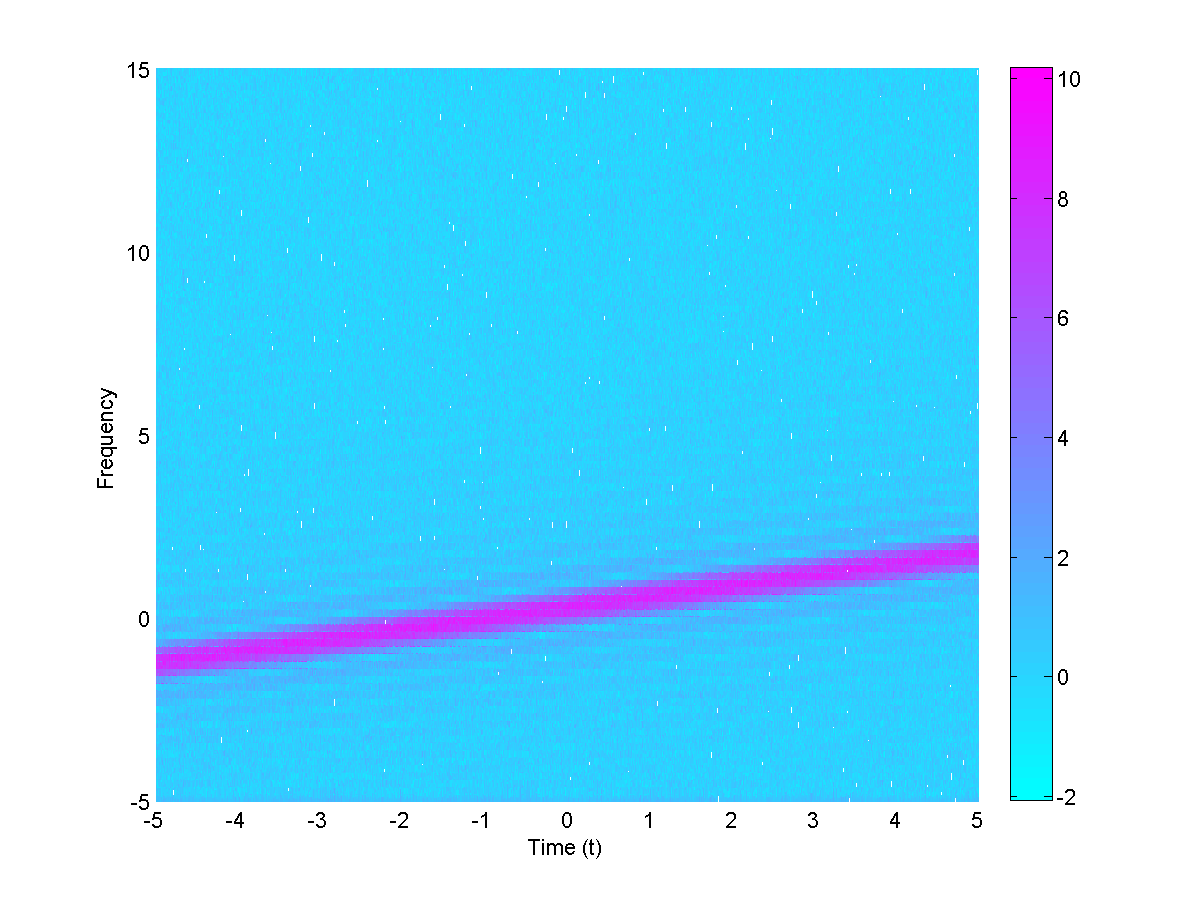}}
    \caption{Absolute value of the new NFrAF bi-component LFM signal (\ref{bisig})
with SNR =10dB db corresponding to particular choices of $\theta$ and $k$. }\label{fig6}
\end{framed}
\end{figure}

{\bf Numerical example.} Let us take a bi-component LFM signal  given by \begin{equation}\label{bisig}
 {\bf x}(t)=e^{i(0.2t+0.3t^2)}+e^{i(0.4t+0.3t^2)},\quad T=10
  \end{equation}
 then by the virtue of (\ref{eqn multiwd}) and choosing $\Lambda_j=1,\forall j,$ and $\theta=\frac{\pi}{4}$ , the  NFrAF of bi-component signal ${\bf x}(t)$ at scaling parameters $k=1,2,\frac{1}{2}$ are  given by
\begin{eqnarray}
\nonumber\mathcal  A^{\frac{\pi}{4},1}_{\bf x}(\tau,u)&=&\frac{5\sqrt{2}}{\pi}\left[e^{0.2i\tau}sinc(8\tau-5\sqrt{2}u)+e^{0.4i\tau}sinc(8\tau-5\sqrt{2}u)\right.\\
\label{bifrwdpi4k1}&&\qquad\left.+e^{0.3i\tau}sinc(8\tau-5\sqrt{2}u-1)+e^{0.3i\tau}sinc(8\tau-5\sqrt{2}u+1)\right]
\end{eqnarray}
\begin{eqnarray}
\nonumber\mathcal  A^{\frac{\pi}{4},2}_{\bf x}(\tau,u)&=&\frac{5\sqrt{2}}{\pi}\left[e^{0.4i\tau}sinc(16\tau-5\sqrt{2}u)+e^{0.8i\tau}sinc(16\tau-5\sqrt{2}u)\right.\\
\label{bifrwdpi4k2}&&\qquad\left.+e^{0.6i\tau}sinc(16\tau-5\sqrt{2}u-1)+e^{0.6i\tau}sinc(16\tau-5\sqrt{2}u+1)\right]
\end{eqnarray}
and
\begin{eqnarray}
\nonumber\mathcal  A^{\frac{\pi}{4},\frac{1}{2}}_{\bf x}(\tau,u)&=&\frac{5\sqrt{2}}{\pi}\left[e^{0.1i\tau}sinc(4\tau-5\sqrt{2}u)+e^{0.2i\tau}sinc(4\tau-5\sqrt{2}u)\right.\\
\label{bifrwdpi4k1/2}&&\qquad\left.+e^{0.15i\tau}sinc(4\tau-5\sqrt{2}u-1)+e^{0.15i\tau}sinc(4\tau-5\sqrt{2}u+1)\right]
\end{eqnarray}
respectively.\\
Similarly, under the same choices of $\Lambda_j=1,\forall j$  as above but taking $\theta=\frac{\pi}{2}$, the  NFrAF of bi-component signal ${\bf x}(t)$ at $k=1,2,\frac{1}{2}$ are obtained by
\begin{eqnarray}
\nonumber\mathcal  A^{\frac{\pi}{2},1}_{\bf x}(\tau,u)&=&\frac{5}{\pi}\left[e^{0.2i\tau}sinc(3\tau-5u)+e^{0.4i\tau}sinc(3\tau-5u)\right.\\
\label{bifrwdpi2k1}&&\qquad\left.+e^{0.3i\tau}sinc(3\tau-5u-1)+e^{0.3i\tau}sinc(3\tau-5u+1)\right],
\end{eqnarray}
\begin{eqnarray}
\nonumber\mathcal  A^{\frac{\pi}{2},1}_{\bf x}(\tau,u)&=&\frac{5}{\pi}\left[e^{0.4i\tau}sinc(6\tau-5u)+e^{0.8i\tau}sinc(6\tau-5u)\right.\\
\label{bifrwdpi2k2}&&\qquad\left.+e^{0.6i\tau}sinc(6\tau-5u-1)+e^{0.6i\tau}sinc(6\tau-5u+1)\right],
\end{eqnarray}
and
\begin{eqnarray}
\nonumber\mathcal  A^{\frac{\pi}{2},1}_{\bf x}(\tau,u)&=&\frac{5}{\pi}\left[e^{0.1i\tau}sinc(1.5\tau-5u)+e^{0.2i\tau}sinc(1.5\tau-5u)\right.\\
\label{bifrwdpi2k/2}&&\qquad\left.+e^{0.15i\tau}sinc(1.5\tau-5u-1)+e^{0.15i\tau}sinc(1.5\tau-5u+1)\right],
\end{eqnarray}
respectively.
Figures\ref{fig4}-Fig.\ref{fig6} show the detection and parameter estimation for the bi-component LFM signal (\ref{bisig}) with SNR=10dB by new NFrAF for $\theta=\frac{\pi}{4},\frac{\pi}{2}$ and $k=1,2,\frac{1}{2}$.\\
  Fig.\ref{fig4}(c)-(d)represents classical AF  and
Fig.\ref{fig5}(c)-(d) and Fig.\ref{fig6}(c)-(d) corresponds to SAF. Hence, we can see from Figs.\ref{fig4}–Figs.\ref{fig6} that the NFrAF benefits cross-term reduction while preserving perfect time–frequency resolution for suitable values of $k$ and $\theta$.

\end{itemize}



\subsection{CNN-Based Signal Classification Using Proposed NFrAF}\,\\


\begin{figure}[t]
\begin{framed}
\centering
\includegraphics[width=1.0\linewidth]{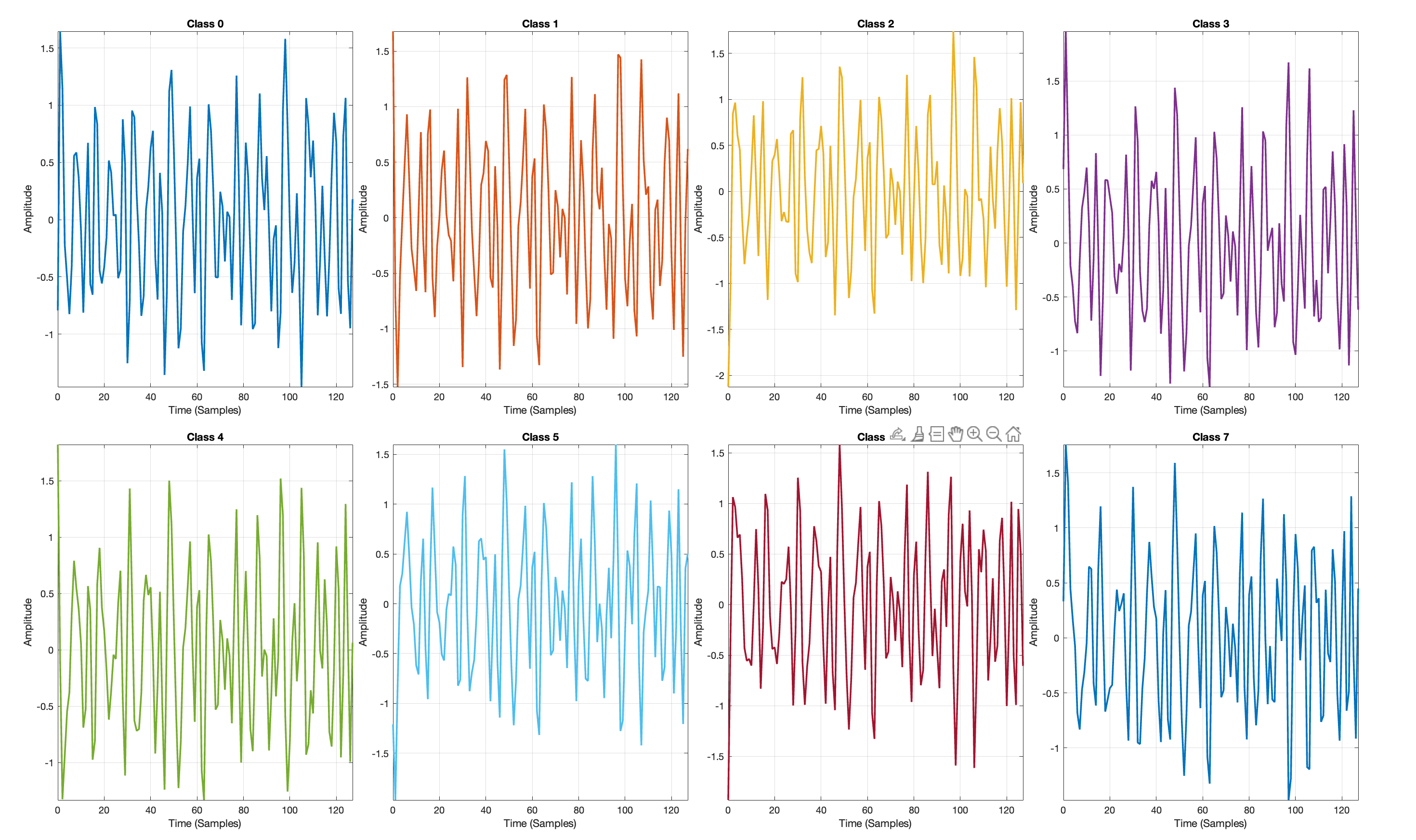}
\caption{Time-domain representations of the eight simulated signal classes.}
\label{fig:time_domain}
\end{framed}
\end{figure}


\begin{figure}[t]
\begin{framed}
\centering
\includegraphics[width=1.0\linewidth]{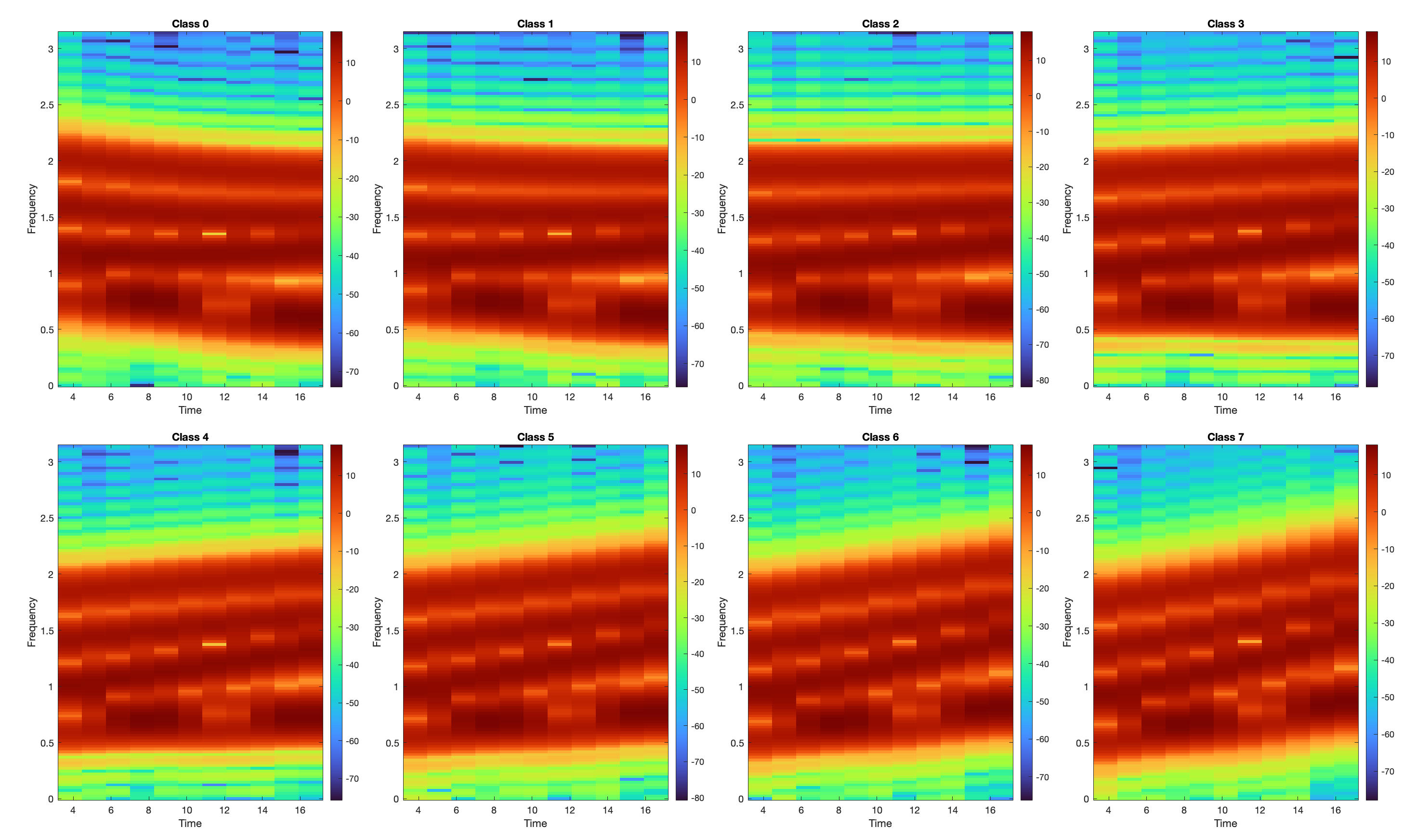}
\caption{STFT magnitude spectrograms of the eight signal classes using a Hann window of length $32$.}
\label{fig:spectrograms}
\end{framed}
\end{figure}


\begin{figure}[t]
\begin{framed}
\centering
\includegraphics[width=1.1\linewidth]{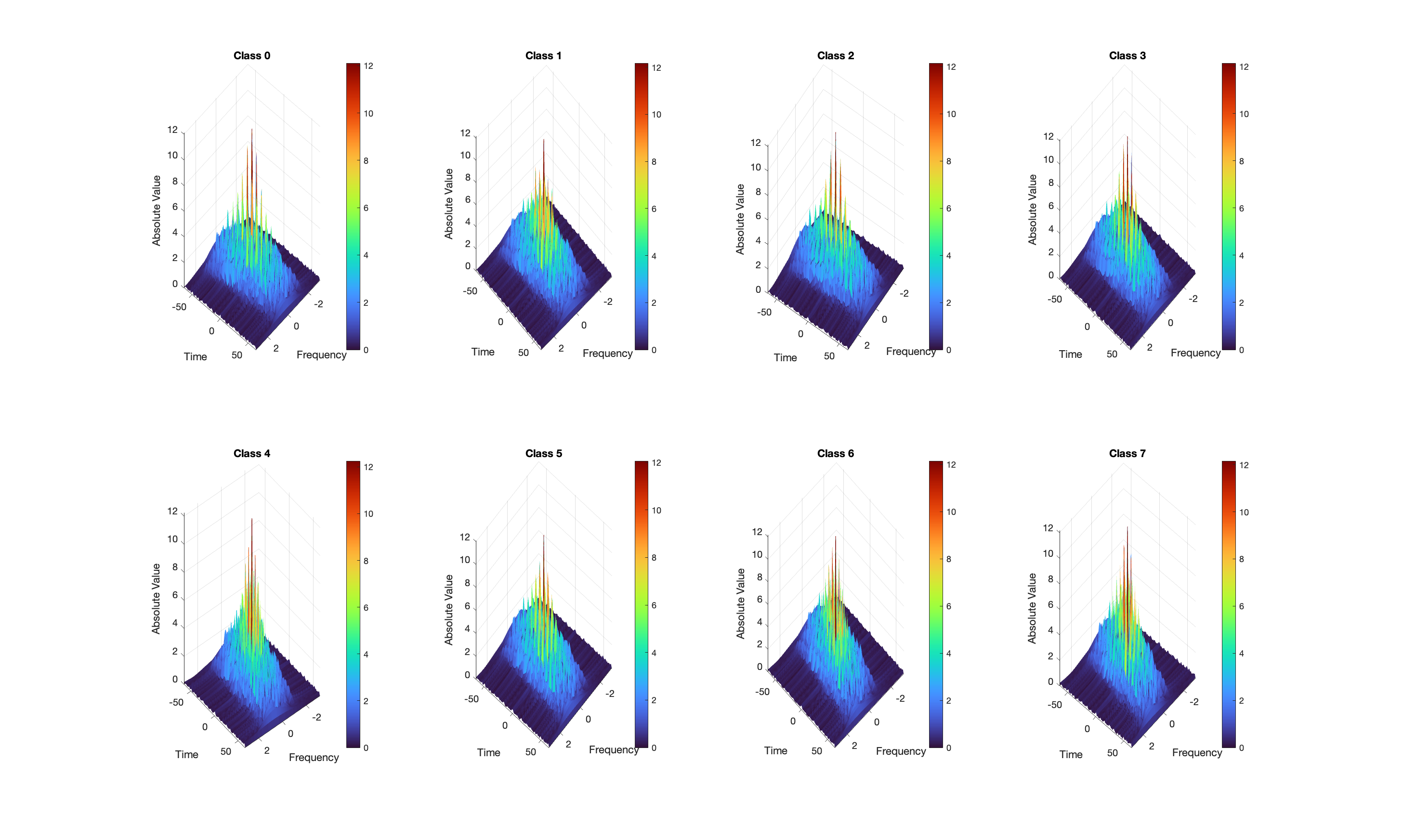}
\caption{Three-dimensional magnitude visualization of the classical ambiguity function (AF) for the eight signal classes.}
\label{fig:af_3d}
\end{framed}
\end{figure}


\begin{figure}[t]
\begin{framed}
\centering
\includegraphics[width=1.0\linewidth]{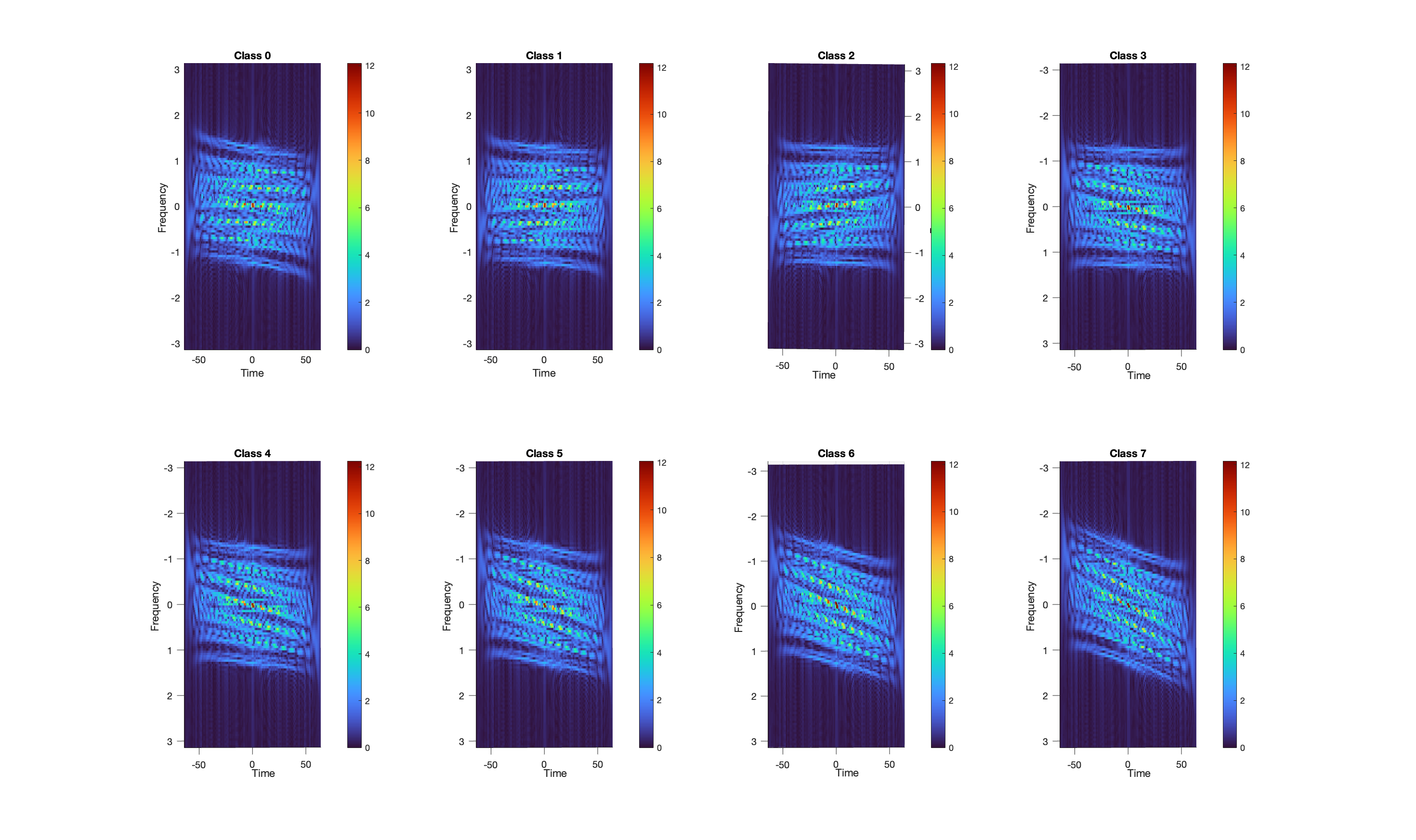}
\caption{Two-dimensional magnitude visualization of the classical ambiguity function (AF) for the eight signal classes.}
\label{fig:af_2d}
\end{framed}
\end{figure}


\begin{figure}[t]
\begin{framed}
\centering
\includegraphics[width=1.0\linewidth]{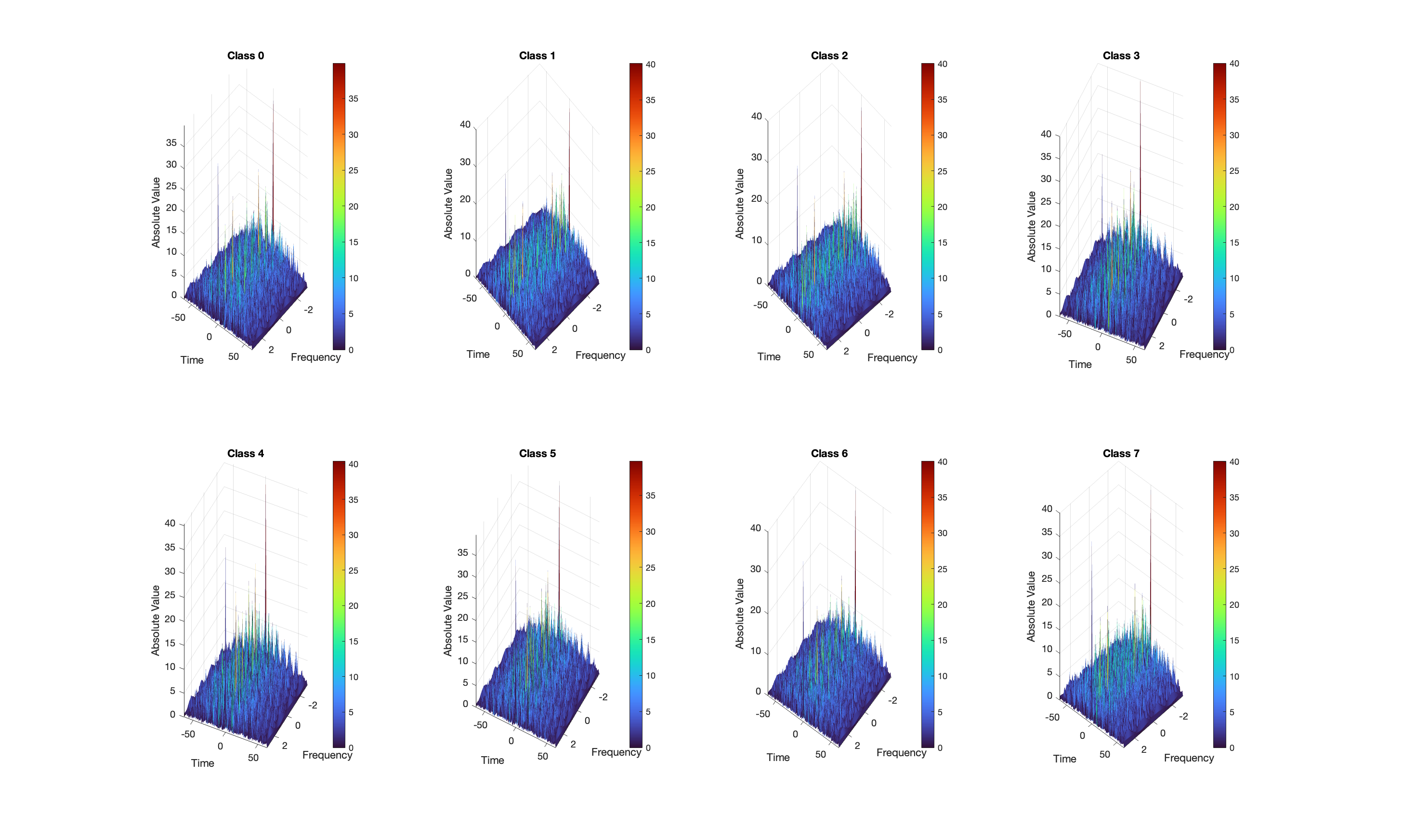}
\caption{Three-dimensional magnitude visualization of the proposed New Fractional Ambiguity Function (NFrAF) for the eight signal classes with $(\theta,k)=\left(\frac{\pi}{6},2\right)$.}
\label{fig:nfraf_3d}
\end{framed}
\end{figure}


\begin{figure}[t]
\begin{framed}
\centering
\includegraphics[width=1.0\linewidth]{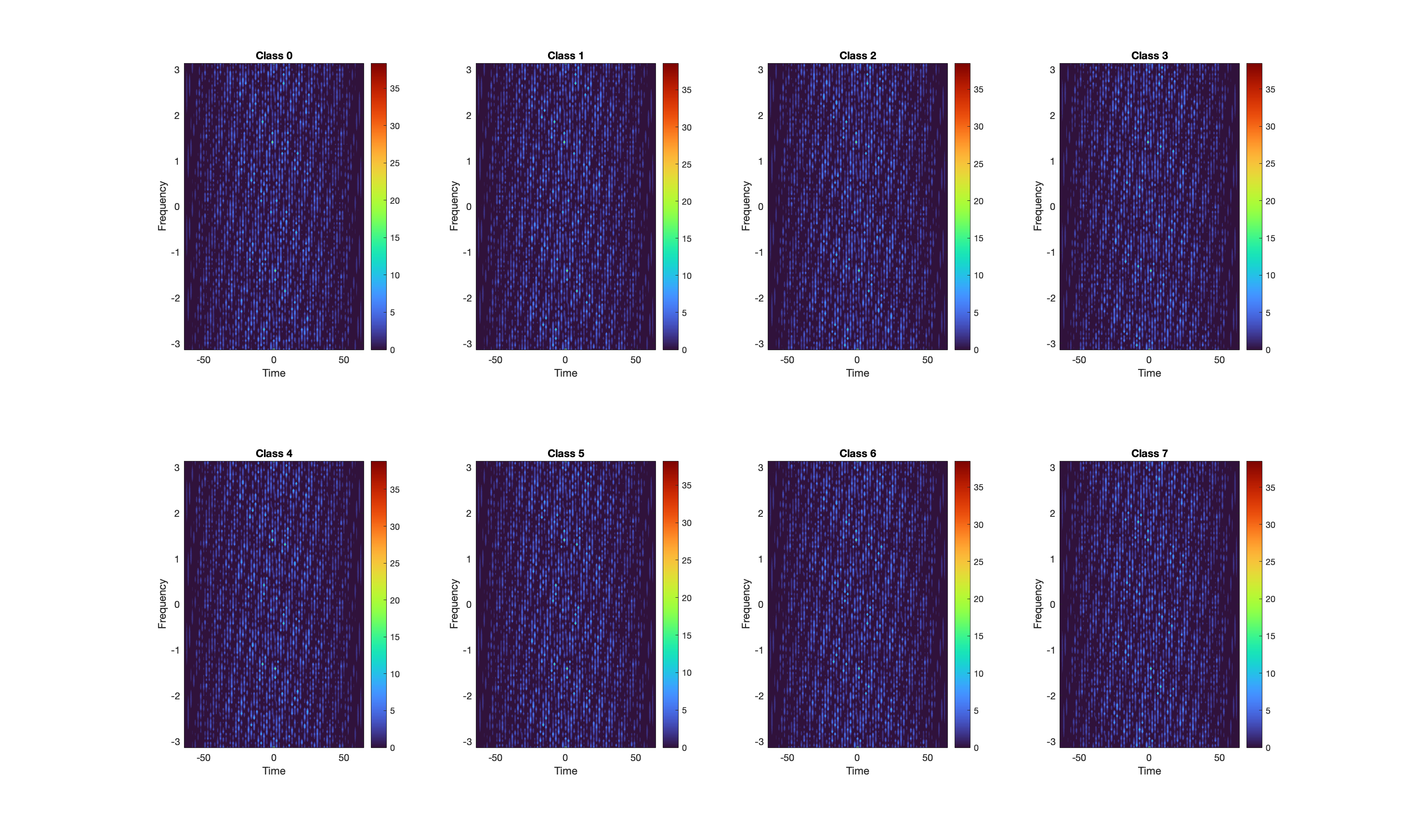}
\caption{Two-dimensional magnitude visualization of the proposed New Fractional Ambiguity Function (NFrAF) for the eight signal classes with $(\theta,k)=\left(\frac{\pi}{6},2\right)$.}
\label{fig:nfraf_2d}
\end{framed}
\end{figure}

To highlight the effectiveness of the NFrAF in signal classification tasks, we generated eight distinct classes of synthetic signals, resulting in a dataset comprising 2000 realizations. From this dataset, 1800 samples were designated for training and the remaining 200 samples were used for validation purposes. The validation set plays a crucial role in monitoring and preventing overfitting during the training process. For each simulated signal, the NFrAF, classical Ambiguity Function (AF), and the spectrogram were computed, ensuring that all three representations were derived from identical signal instances for a fair comparison. The spectrograms were computed using a Hann window of length 32 to capture local time-frequency structures. Prior to training, each representation—NFrAF, AF, and spectrogram—was individually normalized to have zero mean and unit variance, thereby standardizing the data and enhancing the performance of the learning model.

\subsubsection{Synthetic Signal Model}
\label{sec:synthetic-signal-model}

We consider a discrete-time complex-valued signal $x_c[n]$ of length $N=128$ samples,
associated with class $c \in \{1,\dots,8\}$ and time index $n \in \{0,\dots,N-1\}$.
Each signal is constructed as a superposition of weak class-independent background
components, a class-specific motif, and a class-dependent exponential chirp-rate
phase warp, followed by additive noise. Formally,
\begin{equation}
  x_c[n]
  =
  \bigl(
    x_{\mathrm{bg}}[n] +
    x_{\mathrm{st}}[n] +
    x_{\mathrm{motif},c}[n]
  \bigr)
  e^{j\psi_c[n]} + w[n],
  \qquad 0 \le n < N,
  \label{eq:general-xc}
\end{equation}
where $x_{\mathrm{bg}}[n]$ denotes a faint linear-chirp background,
$x_{\mathrm{st}}[n]$ is a faint stationary sinusoid,
$x_{\mathrm{motif},c}[n]$ encodes the class-dependent structure,
$\psi_c[n]$ is an exponential chirp-rate phase warp specific to class $c$,
and $w[n] \sim \mathcal{CN}(0,\sigma^2)$ is circularly symmetric complex-valued
additive white Gaussian noise (with $\sigma^2 = 10^{-8}$ in all experiments).

\paragraph{Background and stationary components.}
The class-independent components are defined as
\begin{align}
  x_{\mathrm{bg}}[n]
  &= A_{\mathrm{bg}}\,
     \exp\!\Bigl(
       j\bigl(2\pi(f_{\mathrm{bg}} n + \tfrac{1}{2} d_{\mathrm{bg}} n^2) + \phi\bigr)
     \Bigr),
  \label{eq:xbg} \\
  x_{\mathrm{st}}[n]
  &= A_{\mathrm{st}}\,
     \exp\!\bigl(
       j(2\pi f_{\mathrm{st}} n + \phi)
     \bigr),
  \label{eq:xst}
\end{align}
where $A_{\mathrm{bg}}=0.45$ and $A_{\mathrm{st}}=0.40$ are fixed amplitudes,
$f_{\mathrm{bg}}$ and $f_{\mathrm{st}}$ are base frequencies,
$d_{\mathrm{bg}}$ is a small chirp rate, and
$\phi \sim \mathcal{U}[0,2\pi)$ is a single global random phase shared by all
components within one signal realization.

The parameters $f_{\mathrm{bg}}, d_{\mathrm{bg}}, f_{\mathrm{st}}$ are drawn from
narrow normal distributions around class-independent anchor values and are then
linearly interpolated toward these anchors using a similarity-strength parameter
$\lambda \in [0,1]$:
\begin{equation}
  \tilde{v} = (1-\lambda)v + \lambda a,
  \label{eq:shrink}
\end{equation}
where $v$ denotes a randomly drawn parameter and $a$ its anchor value.
Larger $\lambda$ enforces stronger similarity across classes;
we use $\lambda = 0.85$ throughout.

\paragraph{Class-specific motifs.}
The motif for class $c$ is constructed as a small sum of complex exponentials,
\begin{equation}
  x_{\mathrm{motif},c}[n]
  =
  \sum_{k=1}^{K_c}
    a_{c,k}
    \exp\!\Bigl(
      j\bigl(2\pi(f_{c,k} n + \tfrac{1}{2} d_{c,k} n^2) + \phi\bigr)
    \Bigr),
  \label{eq:motif}
\end{equation}
where $a_{c,k}$ are amplitudes, $f_{c,k}$ base frequencies,
$d_{c,k}$ chirp rates (zero for stationary tones), and $K_c$ the number of
components in class $c$.

The eight classes are organized into four nearly indistinguishable pairs:
\begin{itemize}
  \item Classes $1,2$: double linear-chirp motifs with closely spaced chirp rates,
        where class~2 additionally includes a weak amplitude modulation.
  \item Classes $3,4$: near-harmonic stacks around a fundamental frequency,
        with class~3 containing an additional weak harmonic.
  \item Classes $5,6$: a single-tone signal versus a very closely spaced
        two-tone mixture.
  \item Classes $7,8$: one or two Gaussian-windowed tones with slightly different
        temporal centers and spreads.
\end{itemize}
Within each pair, the time--frequency signatures of the motifs are deliberately
made highly overlapping so that spectrograms and classical ambiguity functions
appear visually similar across classes.

\paragraph{Exponential chirp-rate phase warp.}
The key class-defining mechanism is the exponential chirp-rate phase warp
\begin{equation}
  \psi_c[n]
  =
  \begin{cases}
    2\pi f_0 (n-n_0), & \beta_c = 0, \\[0.5em]
    2\pi f_0 \dfrac{\exp(\beta_c(n-n_0)) - 1}{\beta_c}, & \beta_c \neq 0,
  \end{cases}
  \label{eq:exp-warp}
\end{equation}
where $f_0$ is a small reference frequency,
$n_0=(N-1)/2$ is the temporal center,
and $\beta_c$ is a class-dependent exponential chirp-rate parameter.
The values $\{\beta_c\}_{c=1}^8$ are chosen on an evenly spaced grid over
$[-\beta_{\max}, \beta_{\max}]$.

This construction ensures continuity with linear phase evolution when
$\beta_c \to 0$, while inducing an exponential instantaneous-frequency law
for $\beta_c \neq 0$. The range of $\beta_c$ is kept sufficiently small so
that neither the time-domain waveforms nor the spectrograms differ
substantially across classes, while affine time--frequency representations
such as the NFrAF remain highly sensitive to these chirp-rate variations.

\paragraph{Dataset and representations.}
For each class, $250$ independent realizations are generated, yielding a
balanced dataset of $2000$ signals. For every signal, we compute three
real-valued time--frequency representations: a normalized STFT magnitude
spectrogram, the classical ambiguity function, and the NFrAF with $(\theta, k)=\left(\frac{\pi}{6},2\right).$ Each representation is
independently normalized to $[0,1]$ on a per-sample basis before
classification.

\paragraph{Implementation details.}
The synthetic signals defined above are generated using a reproducible
implementation that closely follows the mathematical model. Minor
implementation details, including exact parameter values and class-wise
constants, are provided in Appendix~A for completeness.

\begin{table}[t]
\centering
\small
\setlength{\tabcolsep}{6pt}
\caption{Key parameters used in synthetic signal generation.}
\label{tab:gen_params}
\begin{tabular}{ll}
\hline
\textbf{Parameter} & \textbf{Description} \\
\hline
$N$ & Number of samples per signal ($128$) \\
$C$ & Number of classes ($8$) \\
$N_c$ & Number of realizations per class ($250$) \\
$\sigma^2$ & Variance of complex Gaussian noise ($10^{-8}$) \\
$A_{\mathrm{bg}}$ & Amplitude of linear-chirp background component \\
$A_{\mathrm{st}}$ & Amplitude of stationary sinusoidal component \\
$f_{\mathrm{bg}}$ & Base frequency of background chirp \\
$d_{\mathrm{bg}}$ & Chirp rate of background component \\
$f_{\mathrm{st}}$ & Frequency of stationary sinusoid \\
$\lambda$ & Similarity-strength parameter controlling inter-class overlap \\
$K_c$ & Number of motif components in class $c$ \\
$f_{c,k}$ & Base frequency of $k$-th motif component (class $c$) \\
$d_{c,k}$ & Chirp rate of $k$-th motif component \\
$\beta_{\max}$ & Maximum magnitude of exponential chirp-rate parameter \\
$\{\beta_c\}$ & Class-dependent exponential warp parameters \\
$\theta$ & Shear angle used in NFrAF computation ($\pi/6$) \\
$k$ & Scale factor used in NFrAF computation ($2$) \\
\hline
\end{tabular}
\end{table}

\noindent
Table~\ref{tab:gen_params} summarizes the key parameters governing the synthetic
signal generation process. The parameters are chosen to control signal length,
class balance, noise level, and the relative contribution of background,
motif, and affine-warp components in a systematic and reproducible manner.
The number of samples per signal ($N=128$) and the number of realizations per
class ($N_c=250$) are fixed across all experiments to ensure balanced class
representation and consistent input dimensionality for the learning models.

The amplitudes $A_{\mathrm{bg}}$ and $A_{\mathrm{st}}$ define weak, class-independent
background components that are common to all signals and are deliberately kept
small relative to the motif energy. The associated frequencies
$f_{\mathrm{bg}}$, $f_{\mathrm{st}}$ and chirp rate $d_{\mathrm{bg}}$ are drawn from
narrow distributions around anchor values and then shrunk toward those anchors
using the similarity-strength parameter $\lambda$. This mechanism enforces strong
overlap of background statistics across classes, preventing trivial
discrimination based on global spectral differences.

The class-specific motif parameters $K_c$, $f_{c,k}$, and $d_{c,k}$ determine the
number and structure of the sinusoidal or chirp components used to construct
each class. These parameters vary only subtly across class pairs, ensuring that
the resulting time–frequency signatures remain highly similar under
conventional representations such as spectrograms or classical ambiguity
functions.

The exponential chirp-rate warp parameters $\{\beta_c\}$, bounded by
$\beta_{\max}$, constitute the primary source of class discrimination in the
dataset. The range of $\beta_c$ is selected to be sufficiently small so that the
time-domain waveforms and STFT spectrograms remain visually similar across
classes, while affine-aware time–frequency representations are able to capture
the resulting differences. Finally, the parameters $\theta$ and $k$ specify the
shear angle and scale factor used in computing the NFrAF representation and are
held fixed across all experiments to ensure a fair comparison with other
representations.

\noindent
The values of $\{\beta_c\}$ are chosen on an evenly spaced grid over
$[-\beta_{\max},\beta_{\max}]$, while all other parameters are either
class-independent or weakly perturbed around common anchor values.
This design ensures that class discrimination is driven primarily by
the exponential chirp-rate warp rather than by trivial spectral or
energy differences.

\begin{table}[t]
\centering
\small
\setlength{\tabcolsep}{6pt}
\caption{Class-wise motif structures and defining parameters used in the synthetic signal model.}
\label{tab:class_params}
\begin{tabular}{c p{5.2cm} p{6.2cm}}
\hline
\textbf{Class} & \textbf{Motif components} & \textbf{Distinguishing characteristics} \\
\hline
1 &
Two linear chirp components &
Closely spaced chirp rates; no amplitude modulation \\
\hline
2 &
Two linear chirp components &
Similar to Class~1, with additional weak amplitude modulation \\
\hline
3 &
Near-harmonic stack &
Fundamental and higher harmonics; includes an additional weak harmonic component \\
\hline
4 &
Near-harmonic stack &
Similar to Class~3 but without the extra harmonic \\
\hline
5 &
Single stationary tone &
Purely harmonic tone with fixed frequency \\
\hline
6 &
Two closely spaced tones &
Very small frequency separation between tones, producing near-overlapping spectra \\
\hline
7 &
Gaussian-windowed tone(s) &
One or two time-localized tones with specific temporal centers \\
\hline
8 &
Gaussian-windowed tone(s) &
Similar to Class~7 with slightly shifted temporal centers or spreads \\
\hline
\end{tabular}
\end{table}

\noindent
Within each class pair, the motif structures are deliberately designed to
produce highly overlapping time--frequency signatures under conventional
spectrogram and classical ambiguity-function analysis, ensuring that class
discrimination relies primarily on the exponential chirp-rate warp.

\noindent
Table~\ref{tab:class_params} summarizes the class-wise motif structures used in
the synthetic signal generator. The eight classes are organized into four
pairs, where each pair is designed to be nearly indistinguishable under
conventional time--frequency analysis. Within each pair, the constituent
components differ only by subtle structural variations, such as the presence
or absence of a weak harmonic, a slight amplitude modulation, or a small
temporal shift of a localized component.

For Classes~1 and~2, both motifs consist of two linear chirps with closely
spaced chirp rates; Class~2 includes an additional weak amplitude modulation,
which produces only marginal visual differences in the spectrogram domain.
Classes~3 and~4 are near-harmonic stacks centered around a fundamental
frequency, differing only by the inclusion of a weak higher-order harmonic in
Class~3. Classes~5 and~6 contrast a single stationary tone with a two-tone
mixture having an extremely small frequency separation, resulting in almost
identical spectral envelopes. Finally, Classes~7 and~8 consist of one or two
Gaussian-windowed tones whose temporal centers and spreads differ slightly,
leading to overlapping time--frequency support.

Across all class pairs, the background and stationary components are identical
and the motif amplitudes are chosen to be weak relative to the overall signal
energy. Consequently, neither time-domain waveforms nor STFT-based
spectrograms provide reliable class-specific cues. Instead, class identity is
primarily encoded through the exponential chirp-rate phase warp described in
Section~\ref{sec:synthetic-signal-model}, which becomes explicitly aligned with
the affine structure of the NFrAF representation.

\paragraph{Qualitative visualization of the synthetic signals.}
Figure~\ref{fig:time_domain} illustrates representative time-domain realizations
of the eight synthetic signal classes. As intended by the design of the signal
model, the waveforms appear visually similar across classes, with no obvious
class-specific cues observable directly in the temporal domain. This confirms
that class identity is not trivially encoded in amplitude or waveform shape.

Figure~\ref{fig:spectrograms} shows the corresponding STFT magnitude spectrograms
computed using a Hann window of length $32$. Despite minor local variations,
the spectrograms of different classes exhibit strongly overlapping
time--frequency support and highly similar energy distributions. In particular,
signals belonging to paired classes (e.g., Classes~1--2, 3--4, etc.) are
visually difficult to distinguish, highlighting the limited discriminative
power of conventional short-time Fourier analysis for this dataset.

The classical ambiguity function (AF) representations are shown in
Figures~\ref{fig:af_3d} and~\ref{fig:af_2d}. While the AF reveals additional
structure related to chirp content and signal auto-correlation, the magnitude
patterns remain broadly similar across classes. The exponential chirp-rate
differences introduced by the class-dependent warp do not produce clearly
separable signatures in the classical AF, particularly when visualized in
two dimensions.

In contrast, Figures~\ref{fig:nfraf_3d} and~\ref{fig:nfraf_2d} display the
magnitude of the proposed New Fractional Ambiguity Function (NFrAF) computed
with shear angle $\theta=\pi/6$ and scale factor $k=2$. Here, the effect of the
exponential chirp-rate warp becomes markedly more pronounced. The NFrAF
representation aligns with the affine structure induced by the exponential
phase warping, resulting in sharper, more class-specific concentration
patterns. Signals that are nearly indistinguishable in the time domain,
spectrogram, and classical AF domains exhibit visibly distinct signatures
in the NFrAF domain, providing qualitative evidence for the enhanced
discriminative capability of affine-aware time--frequency representations.

\subsubsection{Deep Neural Network Architecture}
\label{sec:cnn-architecture}

We employ convolutional neural networks (CNNs) operating on the two-dimensional
time frequency images derived from the synthetic signals described in
Section~\ref{sec:synthetic-signal-model}. To ensure a strictly fair comparison
across representations, the \emph{same} CNN architecture is used for all three
inputs: the spectrogram, the classical ambiguity function (AF), and the NFrAF.
All representations are treated as single-channel images and processed by a
shared network architecture (denoted \texttt{SharedCNN}), which is trained
independently for each representation using identical hyperparameters and
identical train/validation/test split indices.

Let $X \in \mathbb{R}^{H \times W}$ denote a scalar-valued time--frequency
representation. Since the raw output dimensions may differ across transforms,
an adaptive average pooling layer is first applied to resize each input to a
fixed spatial resolution of $128 \times 128$, yielding a standardized input
tensor of size $1 \times 128 \times 128$.

The standardized input is then processed by a compact CNN trunk composed of
three $3 \times 3$ convolutional layers with ReLU activations. The first two
convolutional layers map the input from one channel to 16 and then 32 feature
maps, respectively, followed by a $2 \times 2$ max-pooling operation that
reduces the spatial resolution. A third $3 \times 3$ convolution increases the
channel dimension to 64 feature maps. The resulting feature maps are compressed
using adaptive average pooling to a fixed spatial size of $4 \times 4$, then
flattened and passed through a fully connected layer with 128 hidden units and
ReLU activation. A final linear layer produces eight output logits,
corresponding to the eight synthetic classes. The softmax nonlinearity is
implicitly handled by the cross-entropy loss during training. The overall
architecture is illustrated in Fig.~\ref{fig:block_cnn_arch}.

All networks are trained using the Adam optimizer with a learning rate of
$10^{-3}$ and a batch size of $32$ for $12$ epochs. The dataset is split into
$70\%$ training, $15\%$ validation, and $15\%$ test sets using a stratified
sampling strategy. To ensure a fair comparison, the same split indices are
reused across all representations.

\begin{figure}[t!]
\centering
\begin{framed}
\includegraphics[width=0.9\linewidth]{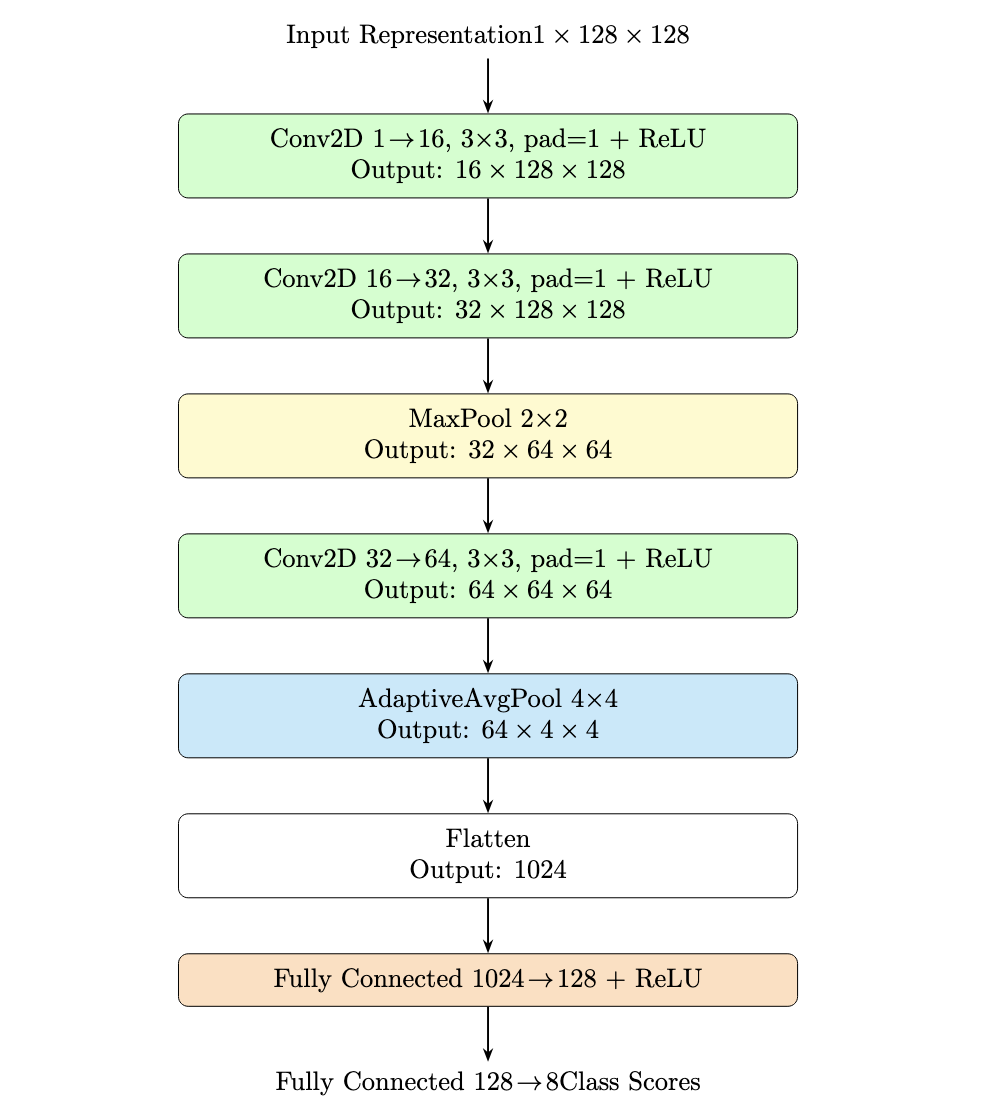}









\caption{Block-wise architecture of the shared CNN used for Spectrogram, AF, and NFrAF inputs.}
\label{fig:block_cnn_arch}
\end{framed}
\end{figure}

\subsubsection{Classification Results}
\label{sec:classification-results}

We evaluate the shared CNN architecture described in Subsubsection~\ref{sec:cnn-architecture} on the chirp regime of the synthetic dataset using a stratified test set comprising $15\%$ of the signals from each of the eight classes. Performance is reported in terms of overall test accuracy and micro-averaged receiver operating characteristic (ROC) area under the curve (AUC), computed in a one-vs-rest manner across all classes. For fairness, the same train/validation/test split indices and the same CNN architecture are used independently for spectrogram, AF, and NFrAF inputs.

In the chirp regime, class identity is encoded through very small differences in chirp slope, making the time--frequency patterns of different classes highly similar in the spectrogram domain. As a result, the spectrogram-based CNN exhibits the lowest classification performance among the three representations, with noticeable confusion between classes that differ only in subtle chirp-rate variations. The AF-based CNN improves upon the spectrogram, reflecting the benefit of quadratic time--frequency representations for resolving multi-chirp structure, but still suffers from moderate cross-term interference.

The best performance is obtained using the proposed NFrAF representation. The NFrAF-based CNN achieves the highest test accuracy and ROC--AUC among all three methods in the chirp regime, indicating that even in scenarios where the signal is not purely affine-warped, the affine-sensitive representation provides more stable and discriminative features than the spectrogram and classical AF.

\begin{figure}[h!]
\begin{framed}
    \centering
    \includegraphics[width=0.9\linewidth]{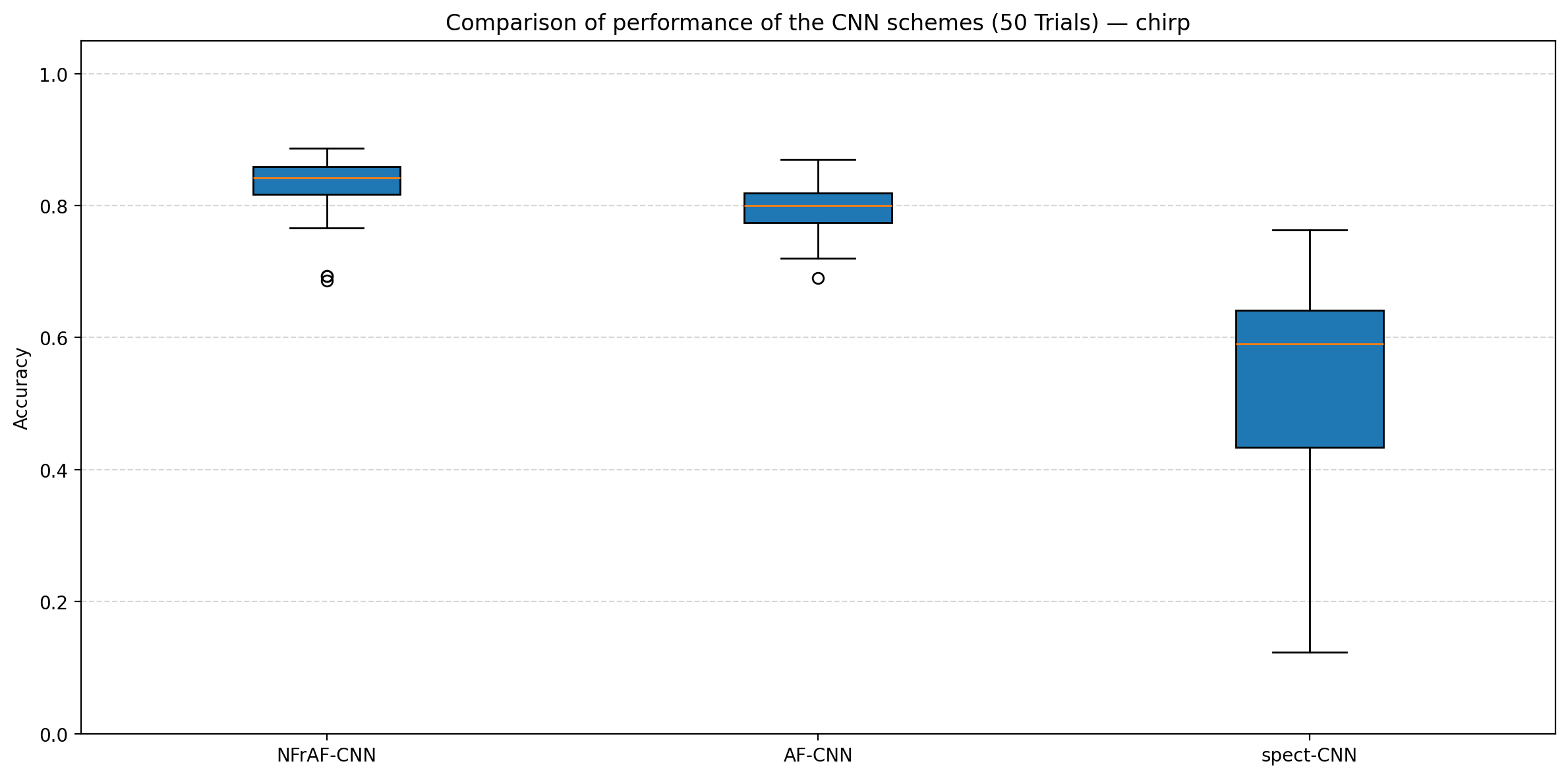}
    \caption{Box plot of test accuracy over 50 independent trials for the chirp regime using spectrogram, AF, and NFrAF representations.}
    \label{fig:boxplot_chirp}
    \end{framed}
\end{figure}

To assess robustness, we repeat the entire training and evaluation process over 50 independent trials with different random seeds affecting signal generation and network initialization. The distribution of test accuracies across trials is summarized in {\bf Figure~\ref{fig:boxplot_chirp}}. The box plot shows that the NFrAF-based CNN achieves a higher median accuracy and a tighter interquartile range than both spectrogram- and AF-based systems, demonstrating superior consistency and reduced sensitivity to random initialization.

{\bf Figure~\ref{fig:confmats_chirp}} shows the confusion matrices obtained in a representative trial for the three representations. The spectrogram-based classifier exhibits widespread off-diagonal errors, particularly between classes with very similar chirp slopes. The AF-based classifier reduces these confusions but still shows noticeable ambiguity caused by cross-term artifacts. In contrast, the NFrAF-based classifier produces a much more diagonal confusion matrix, indicating more reliable separation of classes.

\begin{figure}[!htbp]
\begin{framed}
\centering
\subfigure[Spectrogram]{\includegraphics[width=0.48\textwidth]{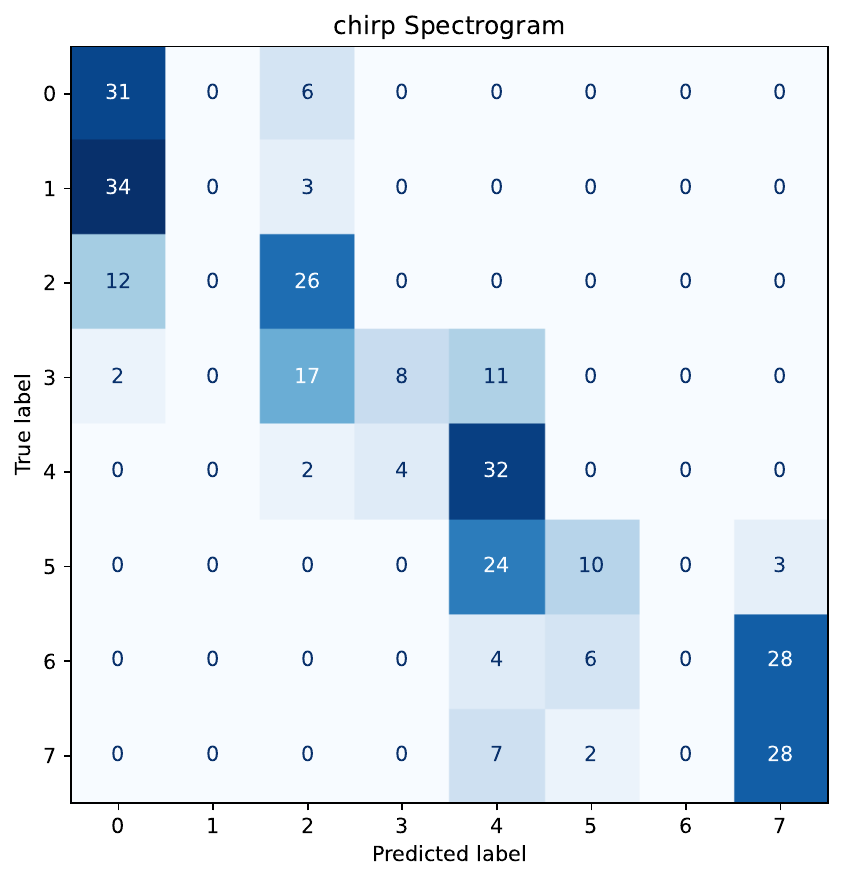}}
    \subfigure[Classical AF]{\includegraphics[width=0.48\textwidth]{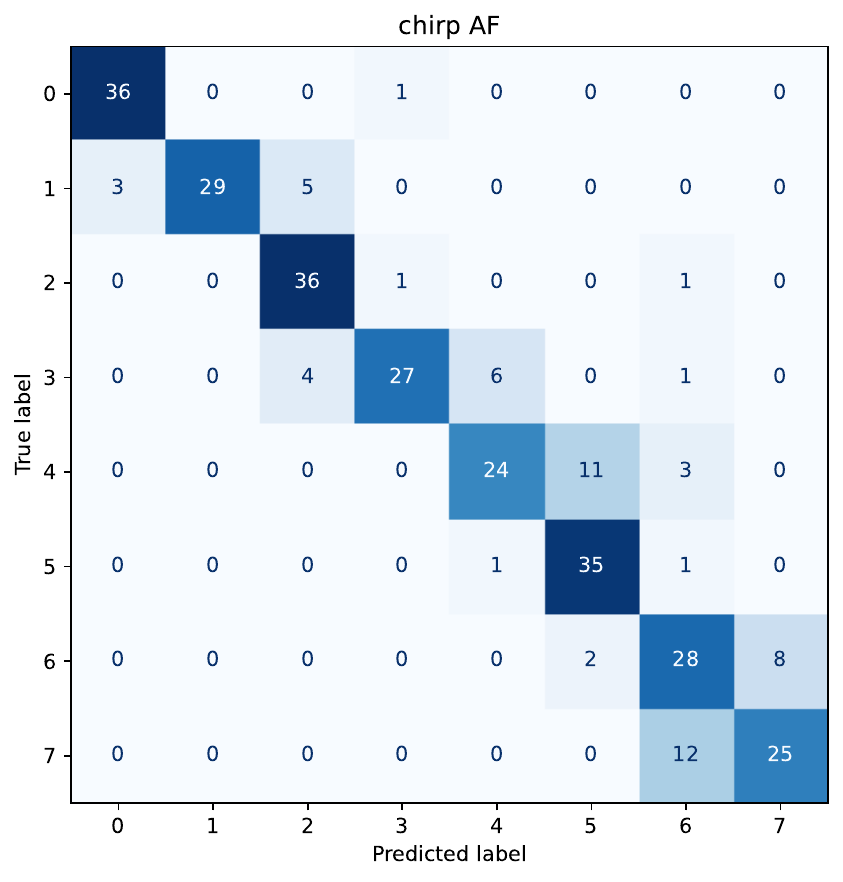}}
    \subfigure[NFrAF]{\includegraphics[width=0.48\textwidth]{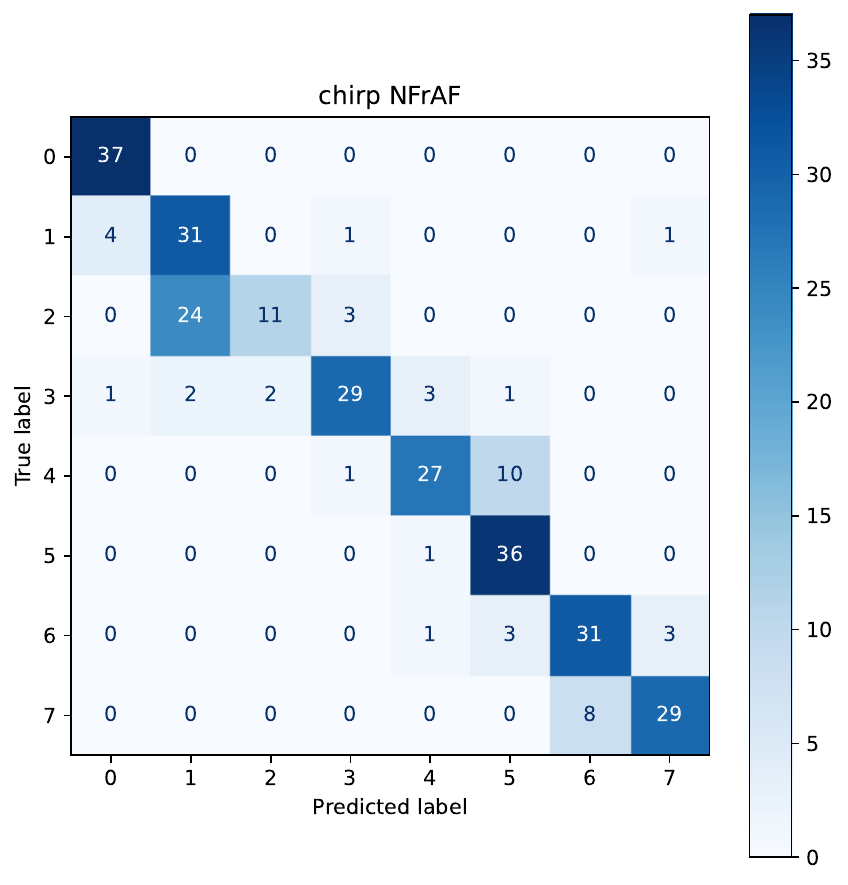}}
    \caption{Confusion matrices for the chirp regime using (a) spectrogram, (b) AF, and (c) NFrAF representations.}
  \label{fig:confmats_chirp}

\end{framed}
\end{figure}

\begin{figure}[t]
\begin{framed}
\centering
\includegraphics[width=0.7\linewidth]{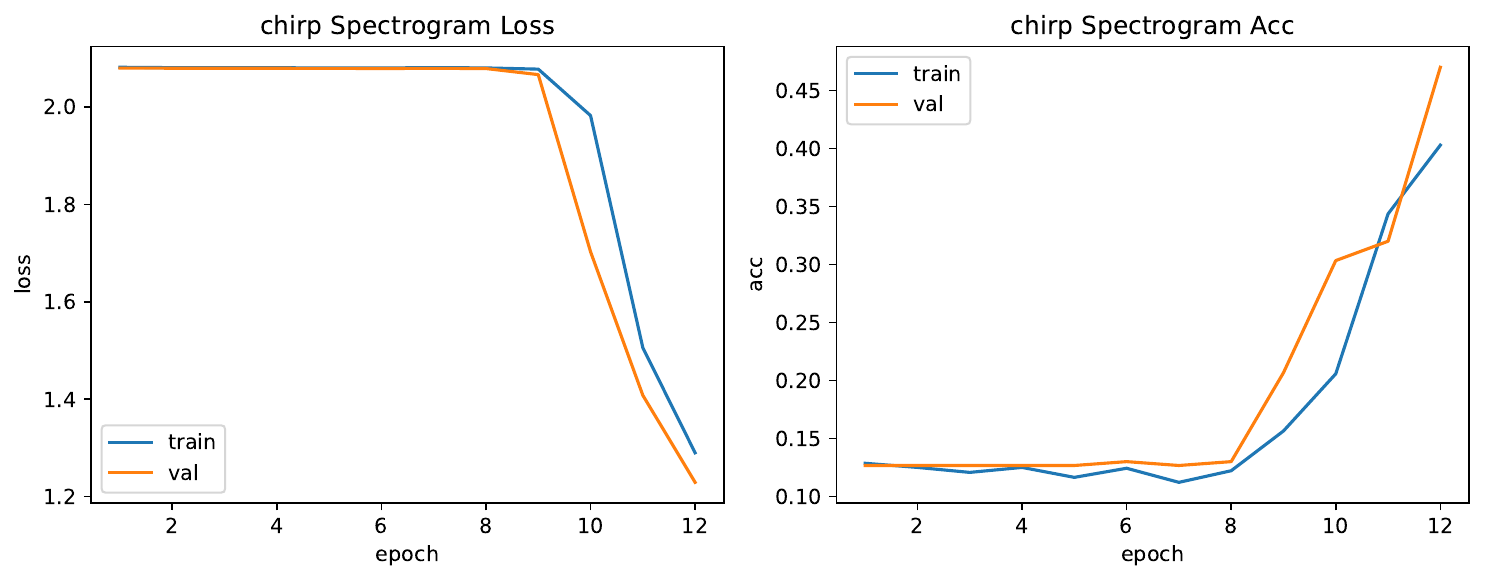}
\caption{Training and validation loss and accuracy curves for the chirp regime using spectrogram input.}
\label{fig:curves_chirp_spec}
\end{framed}
\end{figure}

\begin{figure}[t]
\begin{framed}
\centering
\includegraphics[width=0.7\linewidth]{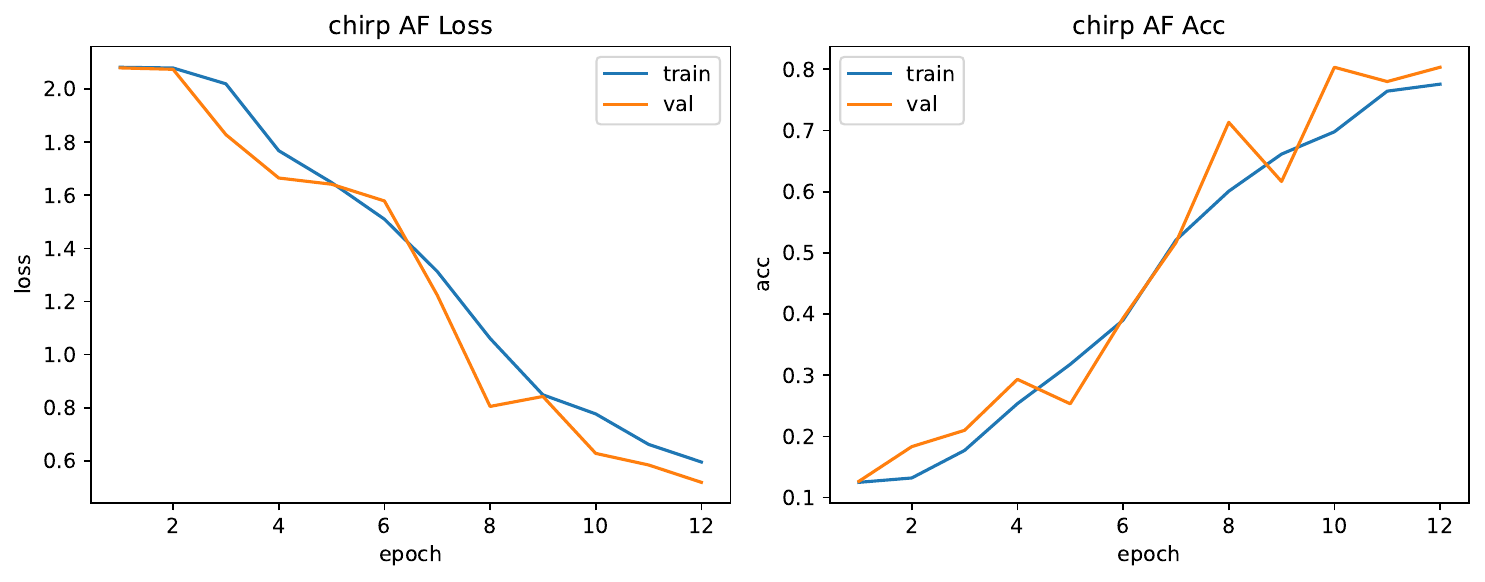}
\caption{Training and validation loss and accuracy curves for the chirp regime using AF input.}
\label{fig:curves_chirp_af}
\end{framed}
\end{figure}

\begin{figure}[t]
\begin{framed}
\centering
\includegraphics[width=0.7\linewidth]{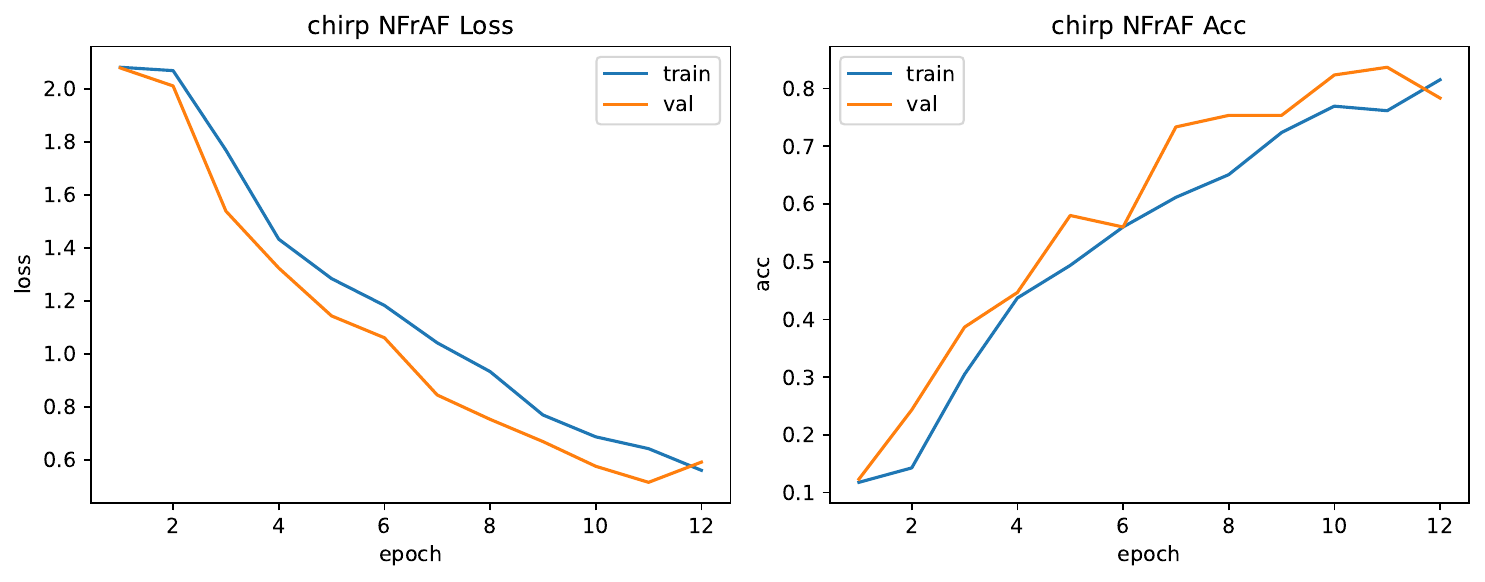}
\caption{Training and validation loss and accuracy curves for the chirp regime using NFrAF input.}
\label{fig:curves_chirp_nfraf}
\end{framed}
\end{figure}

Further insight into the behavior of the three representations is provided by the training and validation loss and accuracy curves shown in Figures~\ref{fig:curves_chirp_spec}--\ref{fig:curves_chirp_nfraf}. For the spectrogram-based CNN (Fig.~\ref{fig:curves_chirp_spec}), training accuracy increases slowly and saturates at a relatively low level, while the validation accuracy remains consistently lower and exhibits noticeable fluctuations. This behavior indicates that the spectrogram representation does not provide sufficiently discriminative features for separating classes that differ only in subtle chirp-rate variations, leading to both underfitting and unstable generalization.

The AF-based CNN (Fig.~\ref{fig:curves_chirp_af}) demonstrates faster convergence and improved training accuracy compared to the spectrogram, reflecting the advantage of quadratic time--frequency representations in resolving multi-chirp structure. However, the validation curves still show moderate oscillations and a persistent gap between training and validation performance, suggesting sensitivity to cross-term artifacts and limited robustness across trials.

In contrast, the NFrAF-based CNN (Fig.~\ref{fig:curves_chirp_nfraf}) exhibits the most stable and consistent learning behavior. Both training and validation losses decrease smoothly, and the corresponding accuracy curves converge rapidly with a small generalization gap. The reduced variance between training and validation performance indicates that the affine-sensitive representation yields more structured and class-consistent features, enabling the network to learn discriminative patterns more reliably. This stable convergence behavior is consistent with the higher median accuracy and reduced variability observed in the box-plot analysis, and further supports the effectiveness of NFrAF for chirp-based classification.

\section{Conclusion}\label{sec 5}

In order to improve cross-term suppression and time-frequency resolution, this work proposed a new fractional ambiguity function (NFrAF), which is a generalization of the classical ambiguity function with additional fractional parameters. Additionally, we showed that the NFrAF significantly outperforms the spectrogram and classical ambiguity function in signal classification tasks on simulated datasets, making it an extremely useful input representation for CNN-based machine-learning frameworks.

\section*{Declarations}
\begin{itemize}
\item  Availability of data and materials: The data is provided on the request to the authors.
\item Competing interests: The authors have no competing interests.
  \item Use of AI tools declaration:
The authors declare they have not used Artificial Intelligence (AI) tools in the creation of this article.
\item Funding: No funding was received for this work
\item Author's contribution: All the authors equally contributed towards this work.

\end{itemize}

\section*{{\bf Appendix A}}

{\bf Proof of Eq.(\ref{stftrel}).}
Taking $z=t+k\frac{\tau}{2}$, (\ref{eqn def cbqpwd}) yields
\begin{eqnarray*}
\mathcal  A^{\theta,k}_{\bf x}\left(\tau,u\right)
&=&{|\Omega_\theta|^2}\int_{\mathbb R}{\bf x}\left(z\right){\bf x}^*\left(z-k\tau\right)e^{i
\left(z-k\frac{\tau}{2}\right)\left(k\tau\cot\theta-u\csc\theta\right)}dz
\end{eqnarray*}
and now
\begin{eqnarray*}
&&\mathcal  A^{\theta,k}_{\bf x}\left(\frac{\tau}{k},u{\sin\theta}+\tau{\cos\theta}\right)\\
&&={|\Omega_\theta|^2}\int_{\mathbb R}{\bf x}\left(z\right){\bf x}^*\left(z-\tau\right)e^{i
\left(z-\frac{\tau}{2}\right)\left[{\tau}\cot\theta-\left(u{\sin\theta}+\tau{\cos\theta}\right)\csc\theta\right]}dz\\
&&={|\Omega_\theta|^2}\int_{\mathbb R}{\bf x}\left(z\right){\bf x}^*\left(z-\tau\right)e^{-iu\left(z-\frac{\tau}{2}\right)}dz\\
&&={|\Omega_\theta|^2}e^{iu\frac{\tau}{2}}\int_{\mathbb R}{\bf x}\left(z\right){\bf x}^*\left(z-\tau\right)e^{-iuz}dz\\
&&={|\Omega_\theta|^2}e^{iu\frac{\tau}{2}}\mathcal  V_{\bf h}[{\bf x}](\tau,u)
\end{eqnarray*}
where $\mathcal  V_{\bf h}[{\bf x}](\tau,u) $ denotes the classical STFT with respect to
the window function ${\bf h}(t)={\bf x}(t).$\\


{\bf Proof of Conjugation Properties.}
By the definition of NFrAF, we have
\begin{eqnarray*}
\nonumber\overline{\mathcal  A^{\theta}_{{\bf x}}(\tau,u)}&=&|\Omega_\theta|^2\int_{\mathbb R}\left({\bf x}^*\underset{k\frac{\tau}{2}}{\otimes}{\bf x}\right)(t)e^{-it(k\tau\cot\theta-u\csc\theta)}dt\\
\nonumber&=&|\Omega_\theta|^2\int_{\mathbb R}\left({\bf x}\underset{-k\frac{\tau}{2}}{\otimes}{\bf x}^*\right)(t)e^{it[k(-\tau)\cot\theta-(-u)\csc\theta]}dt\\
\nonumber&=&\mathcal  A^{\theta,k}_{{\bf x}}(-\tau,-u).
\end{eqnarray*}
And
\begin{eqnarray*}
\nonumber\mathcal  A^{\theta}_{P{\bf x}(t)}(t,u)&=&|\Omega_\theta|^2\int_{\mathbb R}\left(P{\bf x}\underset{k\frac{\tau}{2}}{\otimes}P{\bf x}^*\right)(t)e^{it(k\tau\cot\theta-u\csc\theta)}dt\\
\label{c2}&=&|\Omega_\theta|^2\int_{\mathbb R}\left({\bf x}\underset{-k\frac{\tau}{2}}{\otimes}{\bf x}^*\right)(-t)e^{it(k\tau\cot\theta-u\csc\theta)}dt\\
 &=&|\Omega_\theta|^2\displaystyle\int_{\mathbb R}\left({\bf x}\underset{k\frac{-\tau}{2}}{\otimes}{\bf x}^*\right)(-t)e^{i(-t)[k(-\tau)\cot\theta-(-u)\csc\theta]}dt\\
&=&\mathcal  A^{\theta,k}_{{\bf x}}(-\tau,-u).
\end{eqnarray*}

Hence, completes the proof.\\

\vspace{.1in}

{\bf Proof of Time shift Property.}
 From the definition of the NFrAF, we obtain
\begin{eqnarray*}
&&\mathcal  A^{\theta,k}_{{\bf x}(t-t_0)}(t,u)\\
&&=|\Omega_\theta|^2\int_{\mathbb R}\left({\bf x}\underset{k\frac{\tau}{2}}{\otimes}{\bf x}^*\right)(t-t_0)e^{it(k\tau\cot\theta-u\csc\theta)}dt\\
&&=|\Omega_\theta|^2e^{it_0(k\tau\cot\theta-u\csc\theta)}\int_{\mathbb R}\left({\bf x}\underset{k\frac{\tau}{2}}{\otimes}{\bf x}^*\right)(t-t_0)e^{i(t-t_0)(k\tau\cot\theta-u\csc\theta)}dt\\
&&=e^{it_0(k\tau\cot\theta-u\csc\theta)}\mathcal  A^{\theta,k}_{\bf x}\left(\tau,u\right).
\end{eqnarray*}

This completes the proof.\\
\vspace{.1in}

{\bf Proof of Frequency Shift Property.}
Using Definition \ref{def swdolct}, we have
\begin{eqnarray*}
&&\mathcal  A^{\theta,k}_{\tilde {\bf x}(t)}(\tau,u)\\
&&=|\Omega_\theta|^2\int_{\mathbb R}\left(\tilde{\bf x}\underset{k\frac{\tau}{2}}{\otimes}\tilde{\bf x}^*\right)(t)e^{it(k\tau\cot\theta-u\csc\theta)}dt\\
&&=|\Omega_\theta|^2\int_{\mathbb R}e^{iu_0\left(t+k\frac{\tau}{2}\right)}\left({\bf x}\underset{k\frac{\tau}{2}}{\otimes}{\bf x}^*\right)(t)e^{-iu_0\left(t-k\frac{\tau}{2}\right)}e^{it(k\tau\cot\theta-u\csc\theta)}dt\\
&&=|\Omega_\theta|^2\int_{\mathbb R}\left({\bf x}\underset{k\frac{\tau}{2}}{\otimes}{\bf x}^*\right)(t)e^{it(k\tau\cot\theta-u\csc\theta)+iu_0k\tau}dt\\
&&=e^{iu_0k\tau}|\Omega_\theta|^2\int_{\mathbb R}\left({\bf x}\underset{k\frac{\tau}{2}}{\otimes}{\bf x}^*\right)(t)e^{it(k\tau\cot\theta-u\csc\theta)}dt\\
&&=e^{iu_0k\tau}\mathcal  A^{\theta,k}_{\bf x}\left(\tau,u\right).
\end{eqnarray*}
Which completes the proof.\\

\vspace{.1in}

{\bf Proof of Scaling Property.}
Using the definition of  NFrAF, we have
\begin{eqnarray*}
&&\mathcal  A^{\theta,k}_{\tilde {\bf x}}(t,u)\\
&&=|\Omega_\theta|^2\int_{\mathbb R}\left(\hat{\bf x}\underset{k\frac{\tau}{2}}{\otimes}\hat{\bf x}^*\right)(t)e^{it(k\tau\cot\theta-u\csc\theta)}dt\\
&&=\lambda|\Omega_\theta|^2\int_{\mathbb R}\left({\bf x}\underset{\lambda k\frac{\tau}{2}}{\otimes}{\bf x}^*\right)(\lambda t)e^{it(k\tau\cot\theta-u\csc\theta)}dt\\
&&\buildrel\rm \lambda t=t^\prime \over={\lambda |\Omega_\theta|^2}\int_{\mathbb R}\left({\bf x}\underset{  k\frac{\lambda\tau}{2}}{\otimes}{\bf x}^*\right)( t^\prime)e^{i\frac{t^\prime}{\lambda}(k\tau\cot\theta-u\csc\theta)}\frac{dt^\prime}{\lambda}\\
&&=|\Omega_\theta|^2\int_{\mathbb R}\left({\bf x}\underset{ k\frac{\lambda\tau}{2}}{\otimes}{\bf x}^*\right)( t^\prime)e^{i{t^\prime}(k\frac{\tau}{\lambda}\cot\theta-u\frac{\csc\theta}{\lambda})}{dt^\prime}\\
&&=|\Omega_\theta|^2\int_{\mathbb R}\left({\bf x}\underset{ k\frac{\lambda\tau}{2}}{\otimes}{\bf x}^*\right)( t^\prime)e^{i{t^\prime}\left[k(\lambda \tau)\frac{\cot\theta}{\lambda^2}-u\frac{\csc\theta}{\lambda}\right]}dt^\prime\\
&&=|\Omega_\theta|^2\int_{\mathbb R}\left({\bf x}\underset{ k\frac{\lambda\tau}{2}}{\otimes}{\bf x}^*\right)( t^\prime)e^{i{t^\prime}\left[k(\lambda \tau){\cot\phi}-u\left(\frac{\csc\theta}{\lambda \csc\phi}\right)\csc\phi\right]}dt^\prime\\
&&=\mathcal  A^{\phi,k}_ {\bf x}\left(\lambda \tau,u\frac{\csc\theta}{\lambda \csc\phi}\right),\quad \mbox{where} \quad \phi=arc\cot\left(\frac{\cot\theta}{\lambda^2} \right).
\end{eqnarray*}
Hence completes the proof.\\

\vspace{.1in}

{\bf Proof of Inverse Property.}
The NFrAF can be expressed in terms of Fourier transform as:

\begin{eqnarray*}
\mathcal  A^{\theta,k}_{{\bf x}}(\tau,u)
&=&|\Omega_\theta|^2\int_{\mathbb R}\left({\bf x}\underset{k\frac{\tau}{2}}{\otimes}{\bf x}^*\right)(t)e^{it(k\tau\cot\theta-u\csc\theta)}dt\\
&=&|\Omega_\theta|^2\mathcal F\left[\left({\bf x}\underset{k\frac{\tau}{2}}{\otimes}{\bf x}^*\right)(t)e^{i\tau kt\cot\theta} \right](u\csc\theta)
\end{eqnarray*}
Now taking inverse Fourier transform above equation yields
\begin{equation}
\label{rel ft}\left({\bf x}\underset{k\frac{\tau}{2}}{\otimes}{\bf x}^*\right)(t)e^{i\tau kt\cot\theta}=\frac{1}{\csc\theta}\int_{\mathbb R}\mathcal  A^{\theta,k}_{ {\bf x}}(\tau,u)e^{i\tau u\csc\theta}du.
\end{equation}
Changing ${k\frac{\tau}{2}}$ by $t$,it yields
\begin{eqnarray*}
\left({\bf x}\underset{t}{\otimes}{\bf x}^*\right)(t)e^{2it^2\cot\theta}=\frac{1}{\csc\theta}\int_{\mathbb R}\mathcal  A^{\theta,k}_{ {\bf x}}(\tau,u)e^{i\tau u\csc\theta}du.
\end{eqnarray*}
Equivalently
\begin{eqnarray*}
{\bf x}(2t){\bf x}^*(0)e^{2it^2\cot\theta}=\frac{1}{\csc\theta}\int_{\mathbb R}\mathcal  A^{\theta,k}_{ {\bf x}}(\tau,u)e^{i\tau u\csc\theta}du.
\end{eqnarray*}
Further simplifying we get the desired result as
\begin{eqnarray*}
{\bf x}(t)=\frac{e^{-i\frac{t^2}{2}\cot\theta}}{{\bf x}^*(0)\csc\theta}\int_{\mathbb R}\mathcal  A^{\theta,k}_{ {\bf x}}\left(\tau,u\right)e^{i\tau u\csc\theta}du.
\end{eqnarray*}
Hence completes the proof.\\

\vspace{.1in}

{\bf Proof of Marginal properties.}
From Definition \ref{def swdolct}, we have
\begin{eqnarray*}
&&\int_{\mathbb R}\mathcal  A^{\theta,k}_{\bf x}(\tau,u)du\\
&&=|\Omega_\theta|^2\int_{\mathbb R}\int_{\mathbb R}\left({\bf x}\underset{k\frac{\tau}{2}}{\otimes}{\bf x}^*\right)(t)e^{it(k\tau\cot\theta-u\csc\theta)}dt du\\
&&=\displaystyle\int_{\mathbb R}\left({\bf x}\underset{k\frac{\tau}{2}}{\otimes}{\bf x}^*\right)(t)e^{ikt\tau\cot\theta}\left\{|\Omega_\theta|^2\int_{\mathbb R}e^{-itu\csc\theta}du\right\} dt\\
&&=\displaystyle\int_{\mathbb R}\left({\bf x}\underset{k\frac{\tau}{2}}{\otimes}{\bf x}^*\right)(t)e^{ikt\tau\cot\theta}\delta(t)dt\\
&&=\left({\bf x}\underset{k\frac{\tau}{2}}{\otimes}{\bf x}^*\right)(0).
\end{eqnarray*}
which proves the first part.\\
Again we have from Definition \ref{def swdolct}
\begin{eqnarray*}
&&\int_{\mathbb R}\mathcal  A^{\theta,k}_{\bf x}(\tau,u)d\tau\\
&&=|\Omega_\theta|^2\int_{\mathbb R}\int_{\mathbb R}\left({\bf x}\underset{k\frac{\tau}{2}}{\otimes}{\bf x}^*\right)(t)e^{it(k\tau\cot\theta-u\csc\theta)}dt d\tau\\
&&\buildrel\rm \bm{k\frac{\tau}{2}=x-t} \over=\frac{2|\Omega_\theta|^2}{k}\int_{\mathbb R}\int_{\mathbb R}\left({\bf x}\underset{x-t}{\otimes}{\bf x}^*\right)(t)e^{it[2(x-t)\cot\theta-u\csc\theta]}dx dt.\\
&&=\frac{2|\Omega_\theta|^2}{k}\int_{\mathbb R}\int_{\mathbb R}{\bf x}(x){\bf x}^*\left(2t-x\right)e^{it[2(x-t)\cot\theta-u\csc\theta]}dx dt.\\
&&\buildrel\rm \bm{2t=x+y}\over=\frac{|\Omega_\theta|^2}{k}\int_{\mathbb R}\int_{\mathbb R}{\bf x}(x){\bf x}^*\left(y\right)e^{i\left(\frac{x+y}{2}\right)\left[2\left(x-\frac{x+y}{2}\right)\cot\theta-u\csc\theta\right]}dx dy\\
&&=\frac{|\Omega_\theta|^2}{k}\int_{\mathbb R}\int_{\mathbb R}{\bf x}(x){\bf x}^*\left(y\right)e^{i\left[\frac{(x^2-y^2)}{2}\cot\theta-\frac{(x+y)}{2}u\csc\theta\right]}dx dy\\
&&=\frac{|\Omega_\theta|^2}{k}\int_{\mathbb R}{\bf x}(x)e^{\frac{i}{2}\left[\left(x^2+\left(\frac{u}{2}\right)^2\right)\cot\theta-2x\left(\frac{u}{2}\right)\csc\theta\right]}dx\\
&&\qquad\qquad\times \int_{\mathbb R}{\bf x}^*\left(y\right)e^{\frac{-i}{2}\left[\left(y^2+\left(\frac{-u}{2}\right)^2\right)\cot\theta-2y\left(\frac{-u}{2}\right)\csc\theta\right]} dy\\
&&=\frac{1}{k}\int_{\mathbb R}{\bf x}(x)\Omega_\theta e^{\frac{i}{2}\left[\left(x^2+\left(\frac{u}{2}\right)^2\right)\cot\theta-2x\left(\frac{u}{2}\right)\csc\theta\right]}dx\\
&&\qquad\qquad\times \left[\int_{\mathbb R}{\bf x}(y) \Omega_\theta e^{\frac{i}{2}\left[\left(y^2+\left(\frac{-u}{2}\right)^2\right)\cot\theta-2y\left(\frac{-u}{2}\right)\csc\theta\right]}dy\right]^*\\
&&=\frac{1}{k}\int_{\mathbb R}{\bf x}(x)
\mathcal K_{\theta}\left(x,\frac{u}{2}\right)dx\left[\int_{\mathbb R}{\bf x}(y)\mathcal K_{\theta}\left(y,\frac{-u}{2}\right)dy\right]^*\\
&&=\frac{1}{k}\mathcal F^{\theta}[{\bf x}]\left(\frac{u}{2}\right)\mathcal F^{\theta^*}[{\bf x}]\left(\frac{-u}{2}\right).
\end{eqnarray*}
Hence completes the proof.\\

\vspace{.1in}


{\bf Proof of Moyal Formula.}
From Definition \ref{def swdolct}, we have
\begin{eqnarray*}
&&\int_{\mathbb R}\int_{\mathbb R}\mathcal  A^{\theta,k}_ {{\bf x}_1}(\tau,u)\left[\mathcal  A^{\theta,k}_ {{\bf x}_2}(\tau,u)\right]^*d\tau du\\\\
&&=\left(|\Omega_\theta|^2\right)^2\int_{\mathbb R}\int_{\mathbb R}\int_{\mathbb R}\int_{\mathbb R}\left({\bf x}_1\underset{k\frac{\tau}{2}}{\otimes}{\bf x}_1^*\right)(t)\left({\bf x}_2^*\underset{k\frac{\tau}{2}}{\otimes}{\bf x}_2\right)(t^\prime)\\\\
&&\qquad\qquad\qquad\qquad\times e^{it(k\tau\cot\theta-u\csc\theta)}e^{-it^\prime(k\tau\cot\theta-u\csc\theta)}dt d\tau dt^\prime du\\\\
&&=|\Omega_\theta|^2\int_{\mathbb R}\int_{\mathbb R}\int_{\mathbb R}\left({\bf x}_1\underset{k\frac{\tau}{2}}{\otimes}{\bf x}_1^*\right)(t)\left({\bf x}_2^*\underset{k\frac{\tau}{2}}{\otimes}{\bf x}_2\right)(t^\prime)\\\\
&&\qquad\qquad\qquad\qquad\times e^{i(t-t^\prime)k\tau\cot\theta} \left(\frac{\csc\theta}{2\pi}\int_{\mathbb R}e^{i(t-t^\prime)u\csc\theta}du\right)dt d\tau dt^\prime\\\\
&&=|\Omega_\theta|^2\int_{\mathbb R}\int_{\mathbb R}\int_{\mathbb R}\left({\bf x}_1\underset{k\frac{\tau}{2}}{\otimes}{\bf x}_1^*\right)(t^\prime)\left({\bf x}_2^*\underset{k\frac{\tau}{2}}{\otimes}{\bf x}_2\right)(t)e^{i(t-t^\prime)k\tau\cot\theta} \delta(t-t^\prime)dt d\tau dt^\prime\\\\
&&=|\Omega_\theta|^2\int_{\mathbb R}\int_{\mathbb R}\left({\bf x}_1\underset{k\frac{\tau}{2}}{\otimes}{\bf x}_1^*\right)(t)\left({\bf x}_2^*\underset{k\frac{\tau}{2}}{\otimes}{\bf x}_2\right)(t)d\tau dt\\
\end{eqnarray*}
By setting $k\frac{\tau}{2}=x-t,$ we have
\begin{eqnarray*}
\int_{\mathbb R}\int_{\mathbb R}\mathcal  A^{\theta,k}_ {{\bf x}_1}(\tau,u)\left[\mathcal  A^{\theta,k}_ {{\bf x}_2}(\tau,u)\right]^*d\tau du
&=&\frac{|\Omega_\theta|^2}{k}\int_{\mathbb R}\int_{\mathbb R}\left({\bf x}_1\underset{x-t}{\otimes}{\bf x}_1^*\right)(t)\left({\bf x}_2^*\underset{x-t}{\otimes}{\bf x}_2\right)(t)dxdt\\
\end{eqnarray*}
Now making use of tensor product and taking $2t-x=y,$ we obtain

\begin{eqnarray*}
\int_{\mathbb R}\int_{\mathbb R}\mathcal  A^{\theta,k}_ {{\bf x}_1}(t,u)\left[\mathcal  A^{\theta,k}_ {{\bf x}_2}(t,u)\right]^*d\tau du
&=&\frac{|\Omega_\theta|^2}{k}\int_{\mathbb R}\int_{\mathbb R}{\bf x}_1(x){\bf x}_1^*(y){\bf x}_2^*(x){\bf x}_2(y)dxdy\\
&=&\frac{|\Omega_\theta|^2}{k}\left(\int_{\mathbb R}{\bf x}_1(x){\bf x}_2^*(x)dx\right)\left(\int_{\mathbb R}{\bf x}_1^*(y){\bf x}_2(y)dy\right)\\
&=&\frac{|\Omega_\theta|^2}{k}\left|\langle {\bf x}_1,{\bf x}_2\rangle\right|^2.
\end{eqnarray*}
Which completes the proof.

{\bf{References}}


\begin{enumerate}

{\small {

\bibitem{af1}
L.~Cohen,
\newblock Time--frequency distributions---A review,
\newblock {\em Proc. IEEE}, 77(7):941--981, 1989.
\bibitem{sm2}
R.~Alaifari,
``Multiwindow approaches for direct and stable STFT phase,''
\newblock {\em SIMA J. Appl. Math.}, 85(5):2376--2398, 2025..
\bibitem{af2}
I.~Daubechies,
\newblock {\em Ten Lectures on Wavelets},
\newblock SIAM, Philadelphia, PA, 1992.

\bibitem{sim}
C.~E. Heil and D.~F. Walnut,
\newblock Continuous and discrete wavelet transforms,
\newblock {\em SIAM Review}, 31(4):628--666, 1989.

\bibitem{af3}
H.~M. Ozaktas, M.~A. Kutay, and Z.~Zalevsky,
\newblock {\em The Fractional Fourier Transform with Applications in Optics and Signal Processing},
\newblock Wiley, New York, 2001.

\bibitem{af4}
R.~Tao, B.~Deng, and Y.~Wang,
\newblock {\em Fractional Fourier Transform and Its Applications},
\newblock Tsinghua University Press, Beijing, 2009.

\bibitem{fr1}
V.~Namias,
\newblock The fractional order Fourier transform and its application to quantum mechanics,
\newblock {\em J. Inst. Math. Appl.}, 25:241--265, 1980.

\bibitem{af6}
A.~C. McBride and F.~H. Kerr,
\newblock On Namias's fractional Fourier transforms,
\newblock {\em SIMA J. Appl. Math.}, 39(2):159--175, 1987.

\bibitem{af7}
L.~B. Almeida,
\newblock The fractional Fourier transform and time--frequency representations,
\newblock {\em IEEE Trans. Signal Process.}, 42(11):3084--3091, 1994.

\bibitem{af8}
B.~Santhanam and J.~H. McClellan,
\newblock The discrete rotational Fourier transform,
\newblock {\em IEEE Trans. Signal Process.}, 44(4):994--998, 1996.

\bibitem{af9}
H.~M. Ozaktas, O.~Arikan, M.~A. Kutay, and G.~Bozdagi,
\newblock Digital computation of the fractional Fourier transform,
\newblock {\em IEEE Trans. Signal Process.}, 44(9):2141--2150, 1996.

\bibitem{af10}
S.-C. Pei and J.-J. Ding,
\newblock Closed-form discrete fractional and affine Fourier transforms,
\newblock {\em IEEE Trans. Signal Process.}, 48(5):1338--1353, 2000.

\bibitem{fr3}
M.~J. Bastiaans and A.~J. van Leest,
\newblock From the rectangular to the quincunx Gabor lattice via fractional Fourier transformation,
\newblock {\em IEEE Signal Process. Lett.}, 5(8):203--205, 1998.

\bibitem{fr5}
C.~Candan, M.~A. Kutay, and H.~M. Ozaktas,
\newblock The discrete fractional Fourier transform,
\newblock {\em IEEE Trans. Signal Process.}, 48(5):1329--1337, 2000.

\bibitem{fr6}
G.~Cariolaro, T.~Erseghe, P.~Kraniauskas, and N.~Laurenti,
\newblock Multiplicity of fractional Fourier transforms and their relationships,
\newblock {\em IEEE Trans. Signal Process.}, 48(1):227--241, 2000.

\bibitem{fr13}
T.~Erseghe, P.~Kraniauskas, and G.~Cariolaro,
\newblock Unified fractional Fourier transform and sampling theorem,
\newblock {\em IEEE Trans. Signal Process.}, 47(12):3419--3423, 1999.

\bibitem{fr16}
F.~Hlawatsch and G.~F. Boudreaux-Bartels,
\newblock Linear and quadratic time--frequency signal representations,
\newblock {\em IEEE Signal Process. Mag.}, 9(2):21--67, 1992.

\bibitem{fr24}
D.~Mendlovic and H.~M. Ozaktas,
\newblock Fractional Fourier transforms and their optical implementation,
\newblock {\em J. Opt. Soc. Am. A}, 10(9):1875--1881, 1993.

\bibitem{fr27}
D.~Mustard,
\newblock The fractional Fourier transform and the Wigner distribution,
\newblock {\em J. Austral. Math. Soc. Ser. B}, 38(2):209--219, 1996.

\bibitem{fr221}
R.~Tao, Y.~Li, and Y.~Wang,
\newblock Short-time fractional Fourier transform and its applications,
\newblock {\em IEEE Trans. Signal Process.}, 58(5):2568--2579, 2010.

\bibitem{fr222a}
S.-C. Pei and J.-J. Ding,
\newblock Relations between fractional operations and time--frequency distributions and their applications,
\newblock {\em IEEE Trans. Signal Process.}, 49(8):1638--1655, 2001.

\bibitem{fr222b}
A.~I. Zayed,
\newblock A new perspective on the two-dimensional fractional Fourier transform and its relationship with the Wigner distribution,
\newblock {\em J. Fourier Anal. Appl.}, 25(2):460--487, 2019.

\bibitem{fr222c}
D.~Mustard,
\newblock The fractional Fourier transform and the Wigner distribution,
\newblock {\em J. Austral. Math. Soc. Ser. B}, 38(2):209--219, 1996.

\bibitem{fr222d}
D.~Mendlovic, H.~M. Ozaktas, and A.~W. Lohmann,
\newblock Graded-index fibers, Wigner-distribution functions, and the fractional Fourier transform,
\newblock {\em Appl. Opt.}, 33(26):6188--6193, 1994.

\bibitem{new26}
J.~A. Johnston,
\newblock Wigner distribution and FM radar signal design,
\newblock {\em IEE Proc. F: Radar Signal Process.}, 136(2):81--88, 1989.

\bibitem{new27}
M.~S. Wang, A.~K. Chan, and C.~K. Chui,
\newblock Linear frequency-modulated signal detection using the Radon--ambiguity transform,
\newblock {\em IEEE Trans. Signal Process.}, 46(3):571--586, 1998.

\bibitem{new28}
L.~Auslander and R.~Tolimieri,
\newblock Radar ambiguity functions and group theory,
\newblock {\em SIAM J. Math. Anal.}, 16(3):577--601, 1985.

\bibitem{new29}
G.~Kutyniok,
\newblock Ambiguity functions, Wigner distributions, and Cohen’s class for LCA groups,
\newblock {\em J. Math. Anal. Appl.}, 277(2):589--608, 2003.

\bibitem{d12}
D.~Urynbassarova, A.~Urynbassarova, and E.~Al-Hussam,
\newblock The Wigner--Ville distribution based on the offset linear canonical transform domain,
\newblock in {\em Proc. 2nd Int. Conf. Modelling, Simulation and Applied Mathematics}, 2017.

\bibitem{zhswd}
Z.~C. Zhang, X.~Jiang, S.~Z. Qiang, A.~Sun, Z.~Y. Liang, X.~Shi, and A.~Y. Wu,
\newblock Scaled Wigner distribution using fractional instantaneous autocorrelation,
\newblock {\em Optik}, 237:166691, 2021.

\bibitem{OWNb}
M.~Y. Bhat and A.~H. Dar,
\newblock Convolution and correlation theorems for the Wigner--Ville distribution associated with the quaternion offset linear canonical transform,
\newblock {\em Signal Image Video Process.}, 2023.

\bibitem{new30}
Z.~Y. Zhang and M.~Levoy,
\newblock Wigner distributions and how they relate to the light field,
\newblock in {\em Proc. IEEE Int. Conf. Computational Photography}, pp.~1--10, 2009.

\bibitem{fra19}
G.~Kutyniok,
\newblock Ambiguity functions, Wigner distributions, and Cohen’s class for LCA groups,
\newblock {\em J. Math. Anal. Appl.}, 277(2):589--608, 2003.

\bibitem{fra20}
R.~G. Shenoy and T.~W. Parks,
\newblock Wide-band ambiguity functions and affine Wigner distributions,
\newblock {\em Signal Process.}, 41(3):339--363, 1995.

\bibitem{fra21}
H.~T. Li, P.~M. Djurić, and M.~M. Djuric,
\newblock MMSE estimation of nonlinear parameters of multiple linear and quadratic chirps,
\newblock {\em IEEE Trans. Signal Process.}, 46(3):796--801, 1998.

\bibitem{scale1}
A.~H. Dar and M.~Y. Bhat,
\newblock Scaled ambiguity function and scaled Wigner distribution for linear canonical transform signals,
\newblock {\em Optik}, 267:169678, 2022.

\bibitem{scale3}
M.~Y. Bhat and A.~H. Dar,
\newblock Quadratic-phase scaled Wigner distribution: Convolution and correlation,
\newblock {\em Signal Image Video Process.}, 2023.

\bibitem{fraa1}
T.~W. Che, B.~Z. Li, and T.~Z. Xu,
\newblock The ambiguity function associated with the linear canonical transform,
\newblock {\em EURASIP J. Adv. Signal Process.}, 2012:138, 2012.




\bibitem{zhang}
Z. C. Zhang,
``New Wigner distribution and ambiguity function associated with the linear canonical transform,''
\emph{Optik}, 127 (2016), pp.~4995--5012.

\bibitem{d11}
D. Urynbassarova, B. Z. Li, and R. Tao,
``The Wigner--Ville distribution in the linear canonical transform domain,''
\emph{IAENG Int. J. Appl. Math.}, 46(4) (2016), pp.~559--563.

\bibitem{owncar}
M. Y. Bhat and A. H. Dar,
``Wigner--Ville distribution and ambiguity function of QPFT signals,''
\emph{Ann. Univ. Craiova Math. Comput. Sci. Ser.}, 50(2) (2023), pp.~1--14.

\bibitem{fra26}
X.-G. Xia,
``Discrete chirp-Fourier transform and its application to chirp rate estimation,''
\emph{IEEE Trans. Signal Process.}, 48(11) (2006), pp.~3122--3133.

\bibitem{fra28}
M. Wang, A. K. Chan, and C. K. Chui,
``Linear frequency-modulated signal detection using the Radon--ambiguity transform,''
\emph{IEEE Trans. Signal Process.}, 46(3) (1998), pp.~571--587.

\bibitem{fra24}
C. Zhe, W. Hongyu, and Q. Tianshuang,
``Research on ambiguity function associated with the fractional Fourier transform,''
\emph{Signal Process.}, 19(6) (2003), pp.~499--502 (in Chinese).

\bibitem{fr222b}
V. B. Shakhmurov and A. I. Zayed,
``Fractional Wigner distribution and ambiguity functions,''
\emph{Fract. Calc. Appl. Anal.}, 6(4) (2003), pp.~473--490.

\bibitem{fr222a1}
A. H. Dar, M. Zayed, and M. Y. Bhat,
``Convolution-based fractional Wigner distribution and ambiguity function: Theory and applications,''
\emph{J. Pseudo-Differ. Oper. Appl.}, 15 (2024), Art.~76.

\bibitem{siam1}
C. F. Higham and D. J. Higham,
``Deep learning: An introduction for applied mathematicians,''
\emph{SIAM Rev.}, 61(4) (2019), pp.~860--891.

\bibitem{sim}
Z. Fang, H. Feng, S. Huang, and D.-X. Zhou,
``Theory of deep convolutional neural networks II: Spherical analysis,''
\emph{Neural Netw.}, 131 (2020), pp.~154--162.

\bibitem{siam2}
Y. Yang, H. Feng, and D.-X. Zhou,
``On the rates of convergence for learning with convolutional neural networks,''
\emph{SIAM J. Math. Data Sci.}, 7(4) (2025), Art.~M1652945.

\bibitem{CNN1}
G. Lekkas, E. Vrochidou, and G. A. Papakostas,
``Time--frequency transformations for enhanced biomedical signal classification with convolutional neural networks,''
\emph{BioMedInformatics}, 5(1) (2025), Art.~7.

\bibitem{CNN2}
M. Parlak,
``Use cases for time--frequency image representations and deep learning techniques for improved signal classification,''
\emph{arXiv preprint}, arXiv:2302.11093, 2023.

\bibitem{cnn1}
H. Lee, Y. Largman, P. Pham, and A. Y. Ng,
``Unsupervised feature learning for audio classification using convolutional deep belief networks,''
in \emph{Advances in Neural Information Processing Systems}, vol.~22, 2009, pp.~1096--1104.

\bibitem{cnn2}
M. Espi, M. Fujimoto, K. Kinoshita, and T. Nakatani,
``Exploiting spectro-temporal locality in deep learning acoustic event detection,''
\emph{EURASIP J. Audio Speech Music Process.}, 2015, Art.~26.

\bibitem{cnn3}
Y. Costa, L. Oliveira, and C. S. Jr.,
``An evaluation of convolutional neural networks for music classification using spectrograms,''
\emph{Appl. Soft Comput.}, 52 (2017), pp.~28--38.

\bibitem{CNN4}
S. K. Dasha and G. S. Rao,
``Arrhythmia detection using Wigner--Ville distribution-based neural networks,''
\emph{Procedia Comput. Sci.}, 85 (2016), pp.~806--811.

\bibitem{CNN3}
J. Brynolfsson and M. Sandsten,
``Classification of one-dimensional nonstationary signals using the Wigner--Ville distribution in convolutional neural networks,''
in \emph{Proc. 25th Eur. Signal Process. Conf. (EUSIPCO)}, 2017, pp.~1--5.

\bibitem{fr222c}
A. H. Dar, A. M. Huda, J. G. Dar, and A. Sundus,
``New fractional Wigner distribution using fractional instantaneous autocorrelation,''
\emph{Signal Image Video Process.}, 18 (2024), pp.~1--11.

\bibitem{scale2}
M. Y. Bhat and A. H. Dar,
``Scaled Wigner distribution in the offset linear canonical transform domain,''
\emph{Optik}, 2022, Art.~169286.

\bibitem{fra11}
J. Zhong and Y. Huang,
``Time representation based on an adaptive short-time Fourier transform,''
\emph{IEEE Trans. Signal Process.}, 58 (2010), pp.~5118--5128.

\bibitem{fra12}
Y. Huang,
``Short-time Fourier transform with adaptive window width based on chirp rate,''
\emph{IEEE Trans. Signal Process.}, 60 (2012), pp.~4065--4080.

}}
\end{enumerate}
\end{document}